\documentclass[11pt]{article}

\usepackage{amsfonts, amssymb}
\usepackage{amsmath, eqnarray}
\usepackage{color}
\usepackage{indentfirst}
\usepackage{verbatim}   
\usepackage{mathrsfs}   
\usepackage{yhmath}     
\usepackage{pb-diagram}

\usepackage{enumerate}

 \usepackage{tikz}

\def\F {{\cal F}}
\def\a {\alpha}
\def\av {\alpha^\vee}
\def\Piv {\Pi^\vee}

\def\ai {\a_i}

\def\aiv {\a^\vee_i}

\def\fall   {~\text{for all}~}
\def\with   {~\text{with}~}
\def\for    {~\text{for}~}

\def\iff    {\quad\text{if}\quad}

\def\fa     {\quad\text{for all}~}

\def\g {\mathfrak g}
\def\h {\mathfrak h}
\def\nn{\mathfrak n}

\def\I0{\overline I^0}

\def\Endv{{\rm End}(V)}
\def\Hom{\rm{Hom}({\it P}, \C^\times)}

\def\br {\boldsymbol{\rho}}
\def\ri {{\rm{ri}}\,}

\def\m {\mu}

\def\C {\mathbb C}
\def\D {\Delta}

\def\D {\Delta}
\def\Dr {\Delta^{re}}
\def\Di {\Delta^{im}}
\def\Dp {\Delta_{+}}
\def\Dm {\Delta_{-}}
\def\Drp {\Delta^{re}_+}
\def\Drm {\Delta^{re}_-}

\def\Qp {Q_{+}}
\def\Zp {\mathbb Z_{+}}

\def\la {\langle}
\def\ra {\rangle}
\def\ang{\la\a_j, \av_i\ra}

\def\hR {\mathfrak h_{\mathbb R}}
\def\Z {\mathbb Z}
\def\R {\mathbb R}

\def\n { {\bf m} }

\def\T  {\overline T}
\def\N  {\overline N}

\def\bG{\boldsymbol{G}}
\def\bL{\boldsymbol{L}}
\def\bT{\boldsymbol{T}}
\def\bN{\boldsymbol{N}}
\def\bU{\boldsymbol{U}}
\def\bB{\boldsymbol{B}}

\def\Cai{\star}

\newtheorem{theorem}{Theorem}[section]
\newtheorem{lemma}[theorem]{Lemma}
\newtheorem{corollary}[theorem]{Corollary}
\newtheorem{prop-def}{Proposition-Definition}[section]

\newtheorem{remark}[theorem]{Remark}
\newtheorem{proposition}[theorem]{Proposition}

\newtheorem{exam}{Example}[section]

\newenvironment{proof}{\trivlist \item[\hskip \labelsep{\it Proof.}]}{
 \endtrivlist}

\begin{document}
\title  {Infinite-dimensional reductive monoids associated to highest weight representations of Kac-Moody groups
        }
\date{}
\maketitle

\vspace{ -1.8cm}
\begin{center}
Zhenheng Li \qquad Zhuo Li\footnote{Partially supported by national NSF of China (No 11171202).}\qquad Claus Mokler
\end{center}

\begin{abstract}
Starting with a highest weight representation of a Kac-Moody group over the complex numbers, we construct a monoid whose unit group is the image of the Kac-Moody group under the representation, multiplied by the nonzero complex numbers. We show that this monoid has similar properties to those of a $\mathcal J$-irreducible reductive linear algebraic monoid. In particular, the monoid is unit regular and has a Bruhat decomposition, and the idempotent lattice of the generalized Renner monoid of the Bruhat decomposition is isomorphic to the face lattice of the convex hull of the Weyl group orbit of the highest weight.

\vspace{ 0.3cm}

\vspace{ 0.3cm}

\end{abstract}

\section{Introduction}
Let $\bG$ be a semisimple linear algebraic group over $\C$. Let $\br:\bG\to GL(V)$ be an irreducible rational representation of $\bG$, i.e., an irreducible highest weight representation of $\bG$ with dominant highest weight $\mu$. Then the Zariski closure
\begin{eqnarray}\label{M as Zariski closure}
   M(\br):= \overline{\C^\times \br(\bG)} \subseteq \Endv
\end{eqnarray}
is an irreducible linear algebraic monoid with reductive unit group $\C^\times \br(\bG)$. It belongs to the class of $\mathcal J$-irreducible reductive linear algebraic monoids \cite[Proposition 4.2]{PR88}. It is normal if and only if $\mu$ is minuscule \cite[Theorem 3.1]{DC04}. The most familiar example is $M(m+1,\C)$, which can be obtained from $SL(m+1,\C)$ and its natural representation on $\C^{m+1}$. Here the corresponding highest weight is the first fundamental dominant weight, which is minuscule.

The theory of reductive linear algebraic monoids has been developed mainly by M. S. Putcha and L. E. Renner. Contributions to the theory have been made also by several other people, for example M. Brion, S. Doty, W. Huang, J. Okni{$\rm\acute{n}$}ski, A. Rittatore, L. Solomon, and E. B. Vinberg. Excellent accounts can be found in \cite{Pu88, Re05, LS95}. Much of the information about a reductive linear algebraic monoid is encoded in its combinatorial objects such as cross-section lattice, type functions, and Renner monoid which is a finite unit regular monoid whose unit group is the Weyl group and whose idempotent lattice is isomorphic to the face lattice of a certain polyhedral cone. In particular, $\mathcal J$-irreducible reductive linear algebraic monoids are quite accessible because here the idempotent lattice of a Renner monoid is isomorphic to the face lattice of the convex hull of a single Weyl group orbit. They have been further investigated in \cite{L2, LR1, LP1, LC1}; some special $\mathcal J$-irreducible reductive linear algebraic monoids, referred to as classical algebraic monoids, have been studied in \cite{Li03a, Li03b, LR03}.

To construct the monoids in (\ref{M as Zariski closure}) we may restrict to semisimple simply connected linear algebraic groups. Now the minimal Kac-Moody groups as defined in \cite{Ku02, MP95}, which we simply call Kac-Moody groups,  generalize semisimple simply connected linear algebraic groups. Is it possible to carry out a similar construction for these groups and their irreducible highest weight representations with dominant highest weights?

There are differences and obstructions. These Kac-Moody groups generalize semisimple simply connected linear algebraic groups only as groups. They are constructed by representation and group theoretical methods and do not carry a coordinate ring or even a topology by their construction. The common theory to deal with finite- as well as infinite-dimensional algebraic geometric objects is the theory of schemes, but a more elementary approach would be appropriate in our context. It is also worth to note the following difference. If the generalized Cartan matrix is degenerate then the Kac-Moody group contains a nontrivial maximal central torus. To descend to the factor group is not a reasonable option because many of the highest weight representations would be lost. For example, an affine Kac-Moody group factored by its maximal central torus has only one-dimensional highest weight representations.

Let $\bG$ be a Kac-Moody group over $\C$, and let $\br:\bG\to {\rm GL}(V)$ be an irreducible highest weight representation with dominant highest weight $\mu$. To generalize the construction (\ref{M as Zariski closure}) we proceed as follows:
\begin{itemize}
\item[a)] We define the monoid $M(\br)$ algebraically and show that it has similar algebraic properties to those of a $\mathcal J$-irreducible reductive linear algebraic monoid.
\item[b)] We equip $M(\br)$ with a coordinate ring and show that $M(\br)$ has similar algebraic geometric properties to those of a $\mathcal J$-irreducible reductive linear algebraic monoid. In general, $M(\br)$ is infinite-dimensional. We determine its Lie algebra.
\item[c)]  We provide a certain subalgebra of $\Endv$, which contains $\mathbb{C}^\times \br(\bG)$, with a coordinate ring, and show that $M(\br)=\overline{\mathbb{C}^\times \br(\bG)}$.
\end{itemize}
This requires some amount of work. In this article we carry out part a). In a subsequent article parts b) and c) are treated. We also restrict to the case where no indecomposable component of $\bG$ stabilizes the highest weight space $V_\mu$, which captures already all relevant ideas and is less technical to write down.

Let $W$ be the associated Weyl group, and $S$ the set of simple reflections. Let $H$ be the orbit hull of $\mu$, which is the convex hull of the Weyl group orbit $W\mu$. For every face $F$ of the orbit hull $H$ we get a linear projection $e(F)$ on $V$ acting on the elements $v_\lambda\in V$ of weight $\lambda$ as
\begin{equation*}
  e(F) v_\lambda := \left\{ \begin{array}{ll} v_\lambda & \text{ if } \lambda \in F, \\
   0 & \text{ else. }\end{array} \right.
\end{equation*}
We define $M(\br)$ to be the monoid generated by $G:=\C^\times \br(\bG)$ and the linear projections $e(F)$ of the faces $F$ of the orbit hull $H$.

For our investigation of $M(\br)$ it is necessary to describe at first the face lattice $\F(H)$ of the orbit hull $H$ of $\mu$, and the action of the Weyl group $W$ on $\F(H)$ by lattice isomorphisms. We even allow $\mu$ to be an arbitrary point of the closed fundamental chamber $\overline{C}$, generalizing some results obtained for finite Weyl groups by W. A. Casselman \cite[Sections 3 and 4]{Ca} building on the articles of A. Borel, J. Tits \cite[Sections 12.14  --  12.17]{BT} and I. Satake \cite[Section 2.3]{Sa}, by E. B. Vinberg \cite[Section 3.1]{V91}, and by M. S. Putcha, L. E. Renner \cite[Section 4]{PR88}.
We show that the set of fundamental faces
\begin{equation*}
  \F := \{ F\in \F(H)\mid F\cap\overline{C}\neq \emptyset \} \cup\{\emptyset \}
\end{equation*}
is a sublattice of $\F(H)$, and a cross section for the action of $W$ on $\F(H)$. We obtain the following detailed descriptions of the fundamental faces, of certain $W$-stabilizers, and of the lattice operations:

(a) A subset $I\subseteq S$ is called $\mu$-connected, if no connected component of $I$ is contained in $J_0:=\{s\in S \mid s\mu=\mu\}$. We show that the convex hull $F_I$ of $W_I\mu$, where $W_I$ is the standard parabolic subgroup associated to $I$, is a fundamental face. It can be described as $F_I=W_I (F_I\cap \overline{C})$, and its relative interior as $\text{\rm ri}(F_I)  = W_I \,(\text{\rm ri}(F_I) \cap \overline{C} )$ with
\begin{eqnarray*}
 \text{\rm ri}(F_I) \cap \overline{C} =  (\mu-\R_>I)\cap  \overline{C}
   = \bigcup_{I_f\subseteq I ,\, (I_f)^0 =I_f} \underbrace{ (\mu-\R_>I)\cap C_{I_*\cup I_f}   }_{\neq\emptyset}\; ,
\end{eqnarray*}
where $(I_f)^0$ denotes the union of the connected components of $I_f$ of finite type,
\[
   I_*:=\{s\in J_0 \setminus I \mid st=ts \text{ for all } t\in I\},
\]
and $C_{I_*\cup I_f}$ is the open facet of $\overline{C}$ of type $I_*\cup I_f$. Even for finite Weyl groups this description of the relative interior seems to be new. Furthermore, $\mu+ \R I$ is the affine hull of $F_I$.

The map from the set of all $\mu$-connected subsets of $S$ to the set $\F\setminus\{\emptyset\}$ of all nonempty fundamental faces, which maps $I$ to $F_I$, is bijective.

(b) Let $I\subseteq S$ be $\mu$-connected. We show that $W_{I_*}$ is the stabilizer of $F_I$ in $W$, and $W_{I\cup I_* } = W_I \times W_{I_*}$ is the isotropy group of $F_I$ in $W$.

(c) Let $I, I'\subseteq S$ be $\mu$-connected and $w,w'\in W$. We show that $w F_I\subseteq w' F_{I'}$ if and only if $I\subseteq I'$ and $w^{-1}w'\in W_{I_*}W_{I'}$.

Let $I, I'\subseteq S$ be $\mu$-connected, and let $w\in W$ be a minimal coset representative of $W_{I_*\cup I }  \backslash  W/ W_{I'_*\cup I'}$. If $w=1$, we set $\text{red}(w):=\emptyset$. If $w=s_1 s_2\cdots s_k$ is a reduced expression, we set $\text{red}(w):=\{s_1,s_2,\ldots,s_k\}$, which is independent of the chosen reduced expression. We show that $I\cup I'\cup \text{\rm red}(w)$ is $\mu$-connected and
\begin{eqnarray*}
 F_I\vee w F_{I'} = F_{I\cup I'\cup \text{\rm red}(w)} \quad \text{ and }\quad    F_I \cap w F_{I'} = \left\{\begin{array}{cll}
  F_{(I\cap w I')^*} & \text{if} &  w\in W_{J_0}, \\
   \emptyset & \text{if} &  w\notin W_{J_0} ,
\end{array}\right.
\end{eqnarray*}
where $(I\cap wI')^*$ denotes the biggest $\mu$-connected subset of $I\cap wI'$. The lattice join and lattice meet of two arbitrary faces can be reduced to these formulas. Even for finite Weyl groups these descriptions of the lattice join and lattice meet seem to be new.

In addition to the results of (a), (b), (c) we obtain characterizations for $H\cap\overline{C}$ to be closed, and for $H$ to have finitely many edges containing $\mu$.

Let $(\bB^\pm,\bN)$ be the twin BN-pairs obtained by the construction of the Kac-Moody group $\bG$. The Weyl group is obtained as $W=\bN/\bT$, where $\bT:=\bB^+\cap\bN=\bB^-\cap \bN$ is a maximal torus of $\bG$. We set $B^\pm:=\C^\times \br(\bB^\pm)$, $N:=\C^\times \br(\bN)$, and $T:=\C^\times \br(\bT)$. Since we restrict to the case where no indecomposable component of $\bG$ stabilizes the highest weight space $V_\mu$, or equivalently where $\Pi$ is $\mu$-connected, the Weyl group $W$ identifies with $N/T$. We obtain the following results, which show that the monoid $M(\br)$ has similar algebraic properties to those of a $\mathcal J$-irreducible reductive linear algebraic monoid as listed in \cite[Chapters 7 and 8]{Re05}.

The group $G$ is the unit group of $M(\br)$, and we have the $G\times G$-orbit decomposition
\begin{eqnarray*}
      M(\br) = \bigcup_{F\in \F} G e(F) G  \quad\text{(disjoint)}.
\end{eqnarray*}
Every idempotent of $M(\br)$ is $G$-conjugate to a unique idempotent $e(F)$, $F\in\F$. In particular, the monoid $M(\br)$ is unit regular.

We have Bruhat and Birkhoff decompositions
\begin{eqnarray*}
     M(\br) = \bigcup_{x\in R} B^\epsilon x B^\delta  \quad\text{(disjoint)},
\end{eqnarray*}
where $\epsilon,\delta\in\{+,-\}$, and $R$ is the monoid generated by $N$ and the idempotents $e(F)$, $F\in \F(H)$, factored by the maximal torus $T$.
The monoid $R$ is a generalized Renner monoid in the sense of E. Godelle \cite[Definition 1.4]{Go}, but we simply call it a Renner monoid. In particular, it is unit regular with unit group the Weyl group $W$, and its idempotent lattice is isomorphic to the face lattice $\F(H)$ of the orbit hull $H$. Furthermore, for $s\in S$, $x\in R$, and $\epsilon,\delta\in\{+,-\}$ we have
\begin{eqnarray*}
   (B^\epsilon s B^\epsilon)(B^\epsilon x B^\delta) \subseteq B^\epsilon s x B^\delta \cup B^\epsilon x B^\delta
     \quad \text {and }\quad
   (B^\delta x B^\epsilon)(B^\epsilon s B^\epsilon) \subseteq B^\delta  x s  B^\epsilon \cup B^\delta x B^\epsilon ,
\end{eqnarray*}
generalizing some properties of the twin BN-pairs of $G$.

All the results on the monoid $ M(\br)$ are reached purely algebraically by explicit calculations, for which several sorts of centralizers and stabilizers have to be determined. Important examples are the left and right centralizers monoids
\[
    \begin{aligned}
        C_G^l(e) := \{g\in G\mid ge =ege\}  \quad\text{and}\quad C_G^r(e) := \{g\in G\mid eg = ege\}
    \end{aligned}
\]
for $e:=e(F_I)$, where $I\subseteq S$ is $\mu$-connected. These coincide with the opposite standard parabolic subgroups $P_{I\cup I_*}^+$ and $P_{I\cup I_*}^-$ of $G$. In particular, this requires the investigation of the weight strings through weights contained in the faces of $H$.

Another class of analogues of reductive algebraic monoids whose unit groups are Kac-Moody groups, called face monoids associated to Kac-Moody groups, have been described and investigated in  \cite{Mo02, Mo05, Mo04, Mo07, Mo09, Mo15a}. Here the idempotent lattice of a Renner monoid is isomorphic to the face lattice of the Tits cone. It is surprising and unexpected that such analogues exist, because they reduce classically to the groups themselves. In general, the face lattice of the Tits cone and the face lattice of the convex hull of a Weyl group orbit in the Tits cone differ drastically. In particular, the face monoids have one idempotent different from the identity for affine and strongly hyperbolic Kac-Moody groups and infinitely many idempotents for indefinite, not strongly hyperbolic Kac-Moody groups,  whereas the monoids discussed in this article have infinitely many idempotents for affine and indefinite Kac-Moody groups.

In some subsequent articles the authors investigate the analogues of normal reductive algebraic monoids over $\mathbb{C}$ whose unit groups are general Kac-Moody groups (Kac-Moody groups which generalize reductive linear algebraic groups). A first step of this program has been reached in \cite{Mo15b} by describing the faces and face lattices of arbitrary Coxeter group invariant convex subcones of the Tits cone for a certain class of root bases, where the simple roots and simple coroots may be linearly dependent.

The contents of the sections of this article are the following: In Section 2 we collect some basic facts about Kac-Moody algebras, Kac-Moody groups, their images under highest weight representations, and some needed facts from convex geometry. It is unfortunate that sometimes the usual notation of Kac-Moody theory and of the theory of $\mathcal J$-irreducible reductive linear algebraic monoids are in conflict. In these instances we keep the notation of the latter.
In Section 3 we establish some facts about the faces and the face lattice of the convex hull of a single Weyl group orbit in the Tits cone.
Section 4 is the main part of the article. Here the monoid $M(\br)$ is introduced and investigated. In particular, all the results stated above on $M(\br)$ are proved.

\vskip 2mm
{\bf Acknowledgement} {We would like to thank M. Putcha and L. Renner for valuable conversations and helpful email communications, and R. Koo for useful comments.
}

\tableofcontents

%
\newpage
%

\section{Preliminaries}

We gather necessary notation and some basic facts about Kac-Moody algebras, Kac-Moody groups, and convex geometry from \cite{Kac90, KP83a, KP85, Ku02, Mo05, MP95, Ro96}.

We denote by $\Z_>$, $\R_>$ the sets of strictly positive numbers of $\Z$, $\R$, and $\Z_+$, $\R_+$ contain in addition the zero.
If $M=\bigcup_{i\in I}M_i$ is a disjoint union of sets we write $M=\bigsqcup_{i\in I}M_i$ briefly. We say that a set $B$ intersects a set $C$, if $B\cap C\neq\emptyset$.

\subsection{Kac-Moody algebras}
A {\it generalized Cartan matrix} is an integral matrix $A=(a_{ij})_{i, j = 1}^m$ such that $a_{ii} = 2, ~a_{ij}\le 0$ for $i\ne j$, ~and $a_{ij} = 0$ implies $a_{ji} = 0$. In this article we fix such a matrix $A$ of rank $l$.

A realization of $A$ is a triple $(\h, \Pi, \Piv)$, where $\h$ is a $(2m-l)$-dimensional complex vector space, $\Pi=\{\a_1, \ldots, \a_m\}\subset \h^*$ and $\Piv=\{\av_1, \ldots, \av_m\}\subset \h$ are linearly independent subsets such that $\ang=a_{ij}$, where $\la\,, \ra: \h^* \times \h \rightarrow \C$ is the natural pairing and $i, j\in\n=\{1, \ldots, m\}$. There exists a realization of $A$, unique up to isomorphism.

Let $\tilde{\g}(A)$ be the complex Lie algebra generated by the abelian Lie algebra $\h$ and the symbols $e_i, f_i$, where $i\in\n$, with the following relations
\begin{equation*}
    [e_i, f_j] = \delta_{i j}\av_i ,   \qquad
    [h, e_i]      = \la\a_i, h\ra e_i ,      \qquad
    [h, f_i]      = -\la\a_i, h\ra f_i ,
\end{equation*}
where $i, j \in\n$, $h\in \h$.
There exists a biggest ideal of $\tilde{\g}(A)$ whose intersection with $\h$ is $\{0\}$. The {\it Kac-Moody algebra} $\g=\g(A)$ is the corresponding quotient Lie algebra. We keep the same notation for the images of $e_i, f_i, \h$ in $\g$.

The set $\Pi$ is called the {\it root basis} and $\Piv$ the {coroot basis}; elements in $\Pi$ are referred to as {\it simple roots} and those in $\Piv$ {\it simple coroots.} The free abelian group
\[
    Q := \bigoplus_{i=1}^m\Z\a_i \subseteq \h^*
\]
is called the {\it root lattice}.
The Lie algebra $\g$ has the {\it root space decomposition}
\[
    \g = \bigoplus_{\a\in Q} \g_\a \quad\text{ where }\quad \g_\a=\{x\in \g \mid [h, x] = \la\a, h\ra x \text{ for all } h\in\h\},
\]
and $\g_\a$ is called the {\it root space} associated to $\a$. In particular, $\g_0=\h$, $\g_{{\a}_i}= \C e_i$, $\g_{{-\a}_i}= \C f_i$, $i\in\n$.
The {\it set of roots} is
\[
    \D:=\{\a\in Q\setminus\{0\} \mid \g_\a\ne \{0\}\}.
\]

The {\it Chevalley anti-involution} $\Cai$ of $\g$ is determined by
\[
    (e_i)^\Cai = f_i, \quad (f_i)^\Cai = e_i, \quad h^\Cai = h, \quad\text{ for all } i\in \n, ~h\in\h.
\]
It satisfies $(\g_\a)^\Cai=\g_{-\a}$, $\a\in Q$.

We put $\Qp := \sum\limits_{i=1}^m\Zp\a_i$, $Q_-:=-\Qp$, and give $\h^*$ the partial order
\[
    \mu \ge\mu_1 \quad \Leftrightarrow \quad\m-\mu_1\in\Qp.
\]
Then $\D=\Dp\cup\Dm$ where $\Dp=\{\a\in \D \mid \a > 0\}$ is the set of positive roots and $\Dm=\{\a\in \D \mid \a < 0\}$ is the set of negative roots.  Accordingly, there is the triangular decomposition
\[
    \g = \nn_- \oplus \h \oplus \nn_+ \quad\text{ with } \quad\nn_\pm :=  \bigoplus_{\a\in \D_\pm} \g_\a,
\]
and $(\nn_-)^\Cai=\nn_+$, ~$\h^\Cai=\h$, and $(\nn_+)^\Cai=\nn_-$.

As in \cite[Section 4.7]{Kac90} we associate to a generalized Cartan matrix $A$ its Dynkin diagram, a certain graph whose vertices can be identified with the elements of $\n$ or $\Pi$. The connected components of the Dynkin diagram correspond to the indecomposable generalized Cartan submatrices of $A$. We call $I,J\subseteq\Pi$ separated if $\la\a,\beta^\vee\ra=0$ for all $\a\in I$ and $\beta\in J$.

For the classification of the Kac-Moody algebras $\g(A)$ whose generalized Cartan matrices $A$ are indecomposable into finite, affine, and indefinite type we refer to Chapter 4 of \cite{Kac90}. The Kac-Moody algebra $\g(A)$ is of strongly hyperbolic type if it is of indefinite type and every proper nonempty indecomposable generalized Cartan submatrix of $A$ is of finite type.

For each $i \in\n$ define the {\it fundamental reflection} $r_i\in GL(\h^*)$ by
\[
    r_i(\m) = \m - \la \m, \av_i\ra \a_i , \fa  \m \in \h^* .
\]
The Weyl group $W$ of $\g$ is the subgroup of $GL(\h^*)$ generated by $S:=\{r_i\mid i\in\n\}$. Moreover, $(W, S)$ is a Coxeter system.

For $I, J\subseteq \Pi$ we denote by $W_I$ the standard parabolic subgroup generated by $I$. We denote by $\mbox{}^I W$ the set of minimal coset representatives of $W_I\backslash W$, and by $W^J$ the set of minimal coset representatives of $W/W_J$. We use $\mbox{}^I W^J$ to denote the set of minimal coset representatives of $W_I\backslash W/W_J$.

A {\it real root} is an element of $\Dr:=W\Pi$, and an {\it imaginary root} is an element of $\Di:=\D\setminus\Dr$. If $\a\in\Dr$, then $\dim g_\a = 1$ and $\D\cap\Z\a = \{\a, -\a\}$. If $\a\in\Di$, then $\Z\a \subset \D\cup\{0\}$.

The Weyl group $W$ acts dually on $\h$. In particular, for $i\in\n$ we have
\[
    r_i(h) = h - \la \a_i, h\ra h , \fa  h \in \h.
\]
A {\it real coroot} is an element of $(\Dr)^\vee :=W\Pi^\vee$.  We obtain a $W$-equivariant bijective map $\mbox{}^\vee:\Dr \to (\Dr)^\vee$ by mapping $\a=w\a_i$ to $\av = w\av_i $, $w\in W$, $i\in\n$. Furthermore, $\a>0$ if and only if $\av>0$; the partial order on $\h$ is defined similarly as on $\h^*$.

The reflection with respect to $\a\in\Dr$ is defined by
\[
    r_\a(\m) = \m - \la \m, \av\ra \a , \fa  \m \in \h^* .
\]
If $\a=w(\a_i)$ for some $w\in W$ and $i\in\n$, then $ r_\a = w r_i w^{-1}$. In particular, $r_{\a_i}=r_i$.

Let $\h_\R\subset\h$ be a real vector space of dimension $2m-l$ such that $(\h_\R, \Pi, \Piv)$ is a realization of $A$ over $\R$. Then $\h^*_\R\subset\h^*$ is stable under $W$. The set
\[
    \overline{C} = \{\mu \in \h_\R^* \mid \la \m, \a^\vee \ra \ge 0, \text{ for all } \a \in \Pi \}
\]
is called the fundamental chamber, and
\[
    X = \bigcup_{w\in W} w \overline{C}
\]
is referred to as the Tits cone.

For $J\subseteq \Pi$ we set
\begin{eqnarray*}
  &&  C_J =\{ \mu \in \h_\R^*\mid  \la\mu , \a^\vee\ra = 0\text{ for all }\a\in J,\; \la \mu , \a^\vee\ra > 0\text{ for all }\a\in \Pi\setminus J\},\\
  && \overline{C}_J =\{ \mu \in \h_\R^*\mid \la \mu, \a^\vee\ra = 0\text{ for all }\a\in J,\;  \la \mu , \a^\vee\ra \geq 0\text{ for all }\a\in \Pi\setminus J\} .
\end{eqnarray*}
We call $C_J$ the open, and $\overline{C}_J$ the closed standard facet of type $J$. In particular, $\overline{C}_\emptyset=\overline{C}$. For every $w\in W$ we call $w C_J$ an open, and $w\overline{C}_J$ a closed facet of type $J$.

Any $W$-orbit contained in $X$ intersects the fundamental chamber $\overline{C}$ in exactly one point. We have
\begin{equation*}
\overline{C} =\bigsqcup_{J\subseteq\Pi} C_J ,
\end{equation*}
and for every $\mu\in C_J$ its isotropy group $ W_\m := \{w\in W\mid w \mu=\mu\}$ is given by the standard parabolic subgroup $W_J$. We also call $J$ the type of $\mu$.

We have $wC_J=w'C_{J'}$ if and only if $J=J'$ and $wW_J=w'W_{J'}$. The open facets $\{ w C_J\mid w\in W,\, J\subseteq\Pi \}$ give a $W$-invariant partition of the Tits cone $X$.

A $\g$-module $V$ is called {\it $\h$-diagonalizable} if
\begin{eqnarray*}
    V=\bigoplus_{\eta\in\h^*}V_\eta \quad\text{ where }\quad V_\eta = \{v\in V\mid h v = \la\eta,\,h\ra v \fall h\in\h\},
\end{eqnarray*}
 and $V_\eta$ is called the weight space associated to $\eta$. The set of weights is
\begin{eqnarray*}
   P(V):=\{\eta\in\h^*\mid V_\eta\neq\{0\} \}.
\end{eqnarray*}

An $\h$-diagonalizable $\g$-module $V$ is called {\it integrable} if $e_i$ and $f_i $ are locally nilpotent on $V$ for all $i\in\n$. Its set of weights $P(V)$ is $W$-invariant. For $\eta\in P(V)$ and $\a\in\D^{re}$ the set $P(V)\cap(\eta+\mathbb{Z}\a)$ is called the $\a$-weight string through $\eta$. If $V_\eta$ is finite dimensional then the $\a$-weight string through $\eta$ is of the form
\[
    \eta-p\a, ~\ldots~, ~\eta-\a, ~\eta, ~\eta+\a, ~\ldots~, ~\eta+q\a,
\]
where $p$ and $q$ are nonnegative integers and $ p-q = \la\eta,\,\a^\vee\ra$.

Associated to each $\mu\in \h^*$ is, up to isomorphism, a unique irreducible highest weight module $V$ with highest weight $\m$. It is $\h$-diagonalizable with finite dimensional weight spaces. There is a nondegenerate symmetric contravariant bilinear form on $V$, unique up to a non-zero multiplicative scalar, such that
\begin{equation*}
    (gx\,|\,y)=(x\,| \,g^\Cai y) \quad\for g\in\g, ~x, y\in V.
\end{equation*}
Moreover, $V$ is integrable if and only if $\la \mu , \alpha_i^\vee\ra \in\mathbb{Z}^+$ for all $i\in\n$.

\subsection{Kac-Moody groups}\label{KM}

There are different versions of Kac-Moody groups. We use the one given in \cite{Mo05, MP95}, which will be described below. Others can be found in \cite{KP83a, KP85, Ku02, T85, T87, T92}.
The construction requires certain dual free abelian groups as additional data, which are used to specify a torus algebraic geometrically.

The set $\Piv$ can be extended to a basis of $\h$ by adding elements $\a_{m+1}^\vee, \ldots, \a_{2m-l}^\vee\in\h_{\mathbb{R}}$ such that $\ang\in\Z$ for $j\in\n$ and $i=m+1, \ldots, 2m-l$.
Let
\[
    \h_\Z := \Z\av_1 + \cdots + \Z\av_{2m-l} \subset \h_{\mathbb{R}}\subset  \h,
\]
which is a free abelian group of rank $2m-l$. It is $W$-invariant.

Let $\{\m_1, \ldots, \m_{2m-l}\}$ be the basis of $\h^*$ dual to the basis $\{\av_1, \ldots, \av_{2m-l}\}$ of $\h$. The {\it weight lattice}
\[
    P := \Z\m_1 + \cdots + \Z\m_{2m-l} \subset \h_{\mathbb{R}}^*\subset \h^*
\]
is a free abelian group of rank $2m-l$, and is dual to $\h_\Z$. It is $W$-invariant. Note also that $\Pi\subset P$, and hence $Q\subseteq P$.
The elements of $P$ are called {\it integral weights}, or simply {\it weights}. The elements of
\[
     P_+ := P\cap \overline{C} = \{\m\in P \mid \la\m, \aiv\ra \in\Zp\fall i\in\n\}
\]
are called  {\it dominant} integral weights, or simply {\it dominant} weights, and $\m_1,\ldots,\m_m\in P_+$ are called {\it fundamental} dominant weights.

An integrable representation $(V, \rho)$ of $\g$ is called {\it admissible} if $P(V)\subseteq P$. The adjoint representation of $\g$ is admissible.
Let $\widetilde{\bG}$ be the free product of the torus
\[
 \text{Hom}((P,+), (\C^\times,\cdot))
\]
 and the additive groups $\g_\a$ for $\a\in\Dr$. For any admissible representation $(V, \rho)$ of $\g$, there is a unique group homomorphism $\widetilde{\rho}: \widetilde{\bG}\rightarrow GL(V)$ which maps $\chi\in\Hom$ to $\chi_\rho\in GL(V)$ given by
\begin{equation*}
    \chi_\rho(v_\eta) := \chi(\eta)v_\eta \iff v_\eta\in V_\eta ,\, \eta\in P(V),
\end{equation*}
and which maps $x_\a\in\g_\a$ to $\exp(\rho(x_\a))\in GL(V)$, $\a\in\D^{re}$.

Let $\widetilde{K}$ be the intersection of the kernels of all homomorphisms $\widetilde{\rho}\,'$ where the intersection is taken over all admissible representations $(V', \rho\,')$. The Kac-Moody group is defined to be
\[
    \bG = \widetilde{\bG}/\widetilde{K}.
\]
Let $q: \widetilde{\bG}\rightarrow \bG$ be the natural quotient homomorphism. Then there is a unique group representation $\boldsymbol{\rho}: \bG \rightarrow GL(V)$ such that the diagram
\[
\begin{diagram}
\node{\widetilde{\boldsymbol{G}}}\arrow{s,l}{q}\arrow{e,t}{\widetilde{\rho}}\node{GL(V)}\\
\node{\boldsymbol{G}}\arrow{ne,r}{\br}\node{}
\end{diagram}
\]
commutes.

Let $i_0: {\Hom} \rightarrow  \widetilde{\bG} $ and $i_\a : \g_\a \rightarrow \widetilde{\bG} $, $\a\in\D^{re}$, be the canonical inclusions. The composition $\boldsymbol{t}:=qi_0: {\Hom}\rightarrow \bG$ is an injective group homomorphism. Its image $\bT$ is called the torus of $\bG$ associated to $\h$. For any $\a\in\D^{re}$ the composition $\exp:=qi_\a:\g_\a\to \bG$ is an injective group homomorphism. Its image $\bU_\a$ is called the root group of $\bG$ associated to $\g_\a$ or simply to $\a$. Note that
\begin{align*}
      \boldsymbol{\rho}({\boldsymbol t}(\chi)) &= \chi_\rho \quad\quad\quad\quad \fall \chi\in\Hom,
\\
      \boldsymbol{\rho}(\exp (x)) &= \exp (\rho(x)) \quad \fall x\in\g_\a,\,\a\in\D^{re},
\end{align*}
for every admissible representation $(V,\rho)$ of $\g$.

The torus $\bT$ can be also described in a different way. For every $h\in\h_\Z$ and $c\in\C^\times$ we get an element $\chi_h(c)\in \Hom$ by $ \chi_h(c) \eta := c^{\la\eta, h\ra}$, $ \eta\in P$, inducing an isomorphism of the abelian groups $(\h_\Z,+)\otimes_\Z (\C^\times,\cdot)$ and  $\text{Hom}((P,+), (\C^\times,\cdot))$.
For every $h\in\h_\Z$ and $c\in\C^\times$ the corresponding element $t_h(c):={\boldsymbol t}(\chi_h(c))\in \bT$ acts on each admissible representation $(V,\rho)$ of $\g$ by
\begin{eqnarray*}
     \br( t_h(c)) v_\eta = c^{\la\eta, h\ra} v_\eta ,\quad  v_\eta\in V_\eta,\; \eta\in P(V).
\end{eqnarray*}
Each element $t\in \bT$ can be written uniquely as
\begin{eqnarray*}
    t=\prod_{i=1}^{2m-l} t_i(c_i), \quad\text{for some } c_1, \ldots, c_{2m-l}\in \C^\times,
\end{eqnarray*}
where we have set $t_i(c_i):=t_{h_i}(c_i)$, $i=1,\ldots, 2m-l$.

For every $\a\in\Dr$ we choose $e_\alpha\in\g_{\alpha}$, $f_{\alpha}\in\g_{-\alpha}$ such that $[e_{\alpha},f_{\alpha}]=\alpha^\vee$. For simplicity we choose $e_{\a_i}=e_i$ and $f_{\a_i}=f_i$ if $i\in\n$, and $e_{-\a}=f_\a$ and $f_{-\a}=e_\a$ if $\a\in\Dr_+$.
There is a unique injective homomorphism $\varphi_\a: \rm{SL_2(\C)} \rightarrow \bG$ satisfying
\[
    \varphi_\a \begin{pmatrix}
                        1   &   c \\
                        0   &   1 \\
                    \end{pmatrix}
            = \exp (c e_\a)
    \quad\text{and}\quad
    \varphi_\a    \begin{pmatrix}
                        1   &   0 \\
                        c   &   1 \\
                    \end{pmatrix}
             = \exp (c f_\a),
            \fa
            c\in \C.
\]
Its image $\boldsymbol{SL}_2^{(\a)}$ is the subgroup of $\bG$ generated by $\bU_\a$ and $\bU_{-\a}$.
For $c\in \C^\times$ we set
\[
    n_\a(c) :=  \varphi_\a \begin{pmatrix}
                0       &   c \\
                -c^{-1} &   0 \\
            \end{pmatrix} .
 \]
Then
\begin{eqnarray*}
    n_\a(c) = \exp (ce_\a) \exp(-c^{-1} f_\a)\exp(ce_\a) =  \exp (-c^{-1} f_\a) \exp(c e_\a)\exp(-c^{-1}f_\a) ,
\end{eqnarray*}
and
\begin{equation*}
    t_{\a^\vee}(c) =  n_\a(c)n_\a(1)^{-1} = \varphi_\a  \begin{pmatrix}
            c   &   0 \\
            0   &   c^{-1} \\
         \end{pmatrix},\quad c\in\C^\times.
\end{equation*}
We write $n_i(c)$ for $n_{\a_i}(c)$, and write $n_i$ for $n_{\a_i}(1)$, $i=1,\ldots, m$.
We denote by $\bN$ the subgroup of $\bG$ generated by the elements of the torus $\bT$ and $n_\a(1)$, $\a\in\Dr$.

The system $(\bG, \bT, (\bU_\a)_{\a\in\Dr})$ is a root group data system of type $(W,S)$ with associated twin BN-pairs $(\bB^\pm, \bN)$. These structures are explained  in \cite[Chapter 8]{AB08}. In particular,
\begin{eqnarray*}
    \bigcap_{\a\in\D^{re}}N_{\bG}(\bU_\a)=\bT \quad\text{ and }\quad \bB^+\cap\bB^-=\bB^+\cap \bN=\bB^-\cap \bN=\bT
\end{eqnarray*}
where $N_{\bG}(\bU_\a)$ denotes the normalizer of $\bU_\a$ in $\bG$. We denote by $\bU^\pm$ the subgroup of $\bG$ generated by the root groups $\bU_\a$, $\a\in\D^{re}_\pm$.

The Coxeter system given by the Weyl group $\bN/\bT$ and the set of simple reflections ${\boldsymbol S} =\{n_i{\boldsymbol T}\mid i\in\n\}$ is isomorphic to $(W,S)$. We identify these Coxeter systems. If $w\in W$ is represented by $n_w\in \bN$ then for every $\a\in\Dr$ we have $n_w \bU_\a n_w^{-1} =\bU_{w\a}$. For every admissible $\g$-module $V$ we have
\begin{eqnarray}\label{NTAction}
  n_w V_\eta= V_{w\eta} \quad\text{ for all }\quad\eta\in P(V).
\end{eqnarray}

The Kac-Moody group $\bG$ has the Birkhoff and Bruhat decompositions
\[
    \bG = \bigsqcup_{w\in W} \bB^- w \bB^+ =  \bigsqcup_{w\in W} \bB^+w \bB^-
    \quad\text{and}\quad
    \bG = \bigsqcup_{w\in W} \bB^+  w \bB^+ = \bigsqcup_{w\in W} \bB^- w \bB^-.
\]

Let $I\subseteq \Pi$. The standard parabolic subgroups ${\boldsymbol P}_I ^+= \bB^+ W_I \bB^+$ and ${\boldsymbol P}_I^- = \bB^-W_I \bB^-$ admit the Levi decompositions
\begin{equation*}
   {\boldsymbol  P}_I ^\pm = \bL_I \ltimes \bU^I_\pm .
\end{equation*}
Here $\bL_I = {\boldsymbol P}_I^+\cap {\boldsymbol P}_I^-$ is the subgroup of $\bG$ generated by $\bT$ and the root groups $\bU_\a$,  $\a\in W_I I$. Furthermore,
$\bU^I_\pm = \cap_{w\in W_I}w\bU^\pm w^{-1}$ are the normal subgroups of $\bU^\pm$ generated by the root groups $\bU_\a$, $\a\in\D^{re}_\pm\setminus W_I I$. In particular,
\begin{eqnarray*}
 \bB^\pm = \bT \ltimes \bU^\pm.
\end{eqnarray*}

The center of the Kac-Moody group ${\boldsymbol G}$ is
\[
    Z(\boldsymbol{G}) = \bigg\{ \prod_{i=1}^{2m-l}t_{i}(c_i)\in\bT \biggm |  \prod_{i=1}^{2m-l}c_i^{\la\a, ~\av_i\ra}=1\text{ for all }\a\in\Dr \bigg\} \subseteq\bT.
\]

An irreducible highest weight representation $(V,\rho)$ of $\g$ with highest weight $\mu$ is admissible if and only if $\mu\in P_+$. We call the corresponding representation $(V,\br)$ of $\bG$ an irreducible highest weight representation of $\bG$ of highest weight $\mu$.
%
%
%
%
%
%
The following Lemma describes in two cases the kernel $\ker(\br)= \br^{-1}(\text{\rm id}_V)$ and the normal subgroup $\br^{-1}(\C^\times \text{\rm id}_V)$ of $\bG$, where $\text{\rm id}_V$ denotes the identity map on $V$. To prove part (b), \cite[Chapter IV, \S 2.7, Lemma 2]{BBK02} is useful.
\begin{lemma}\label{kernel}
Let $(V, \br)$ be an irreducible highest weight representation of $\bG$ of highest weight $\mu\in P_+$. Let $J_0$ be the type of $\mu$.
\begin{enumerate}
\item[{\rm(a)}] If $J_0=\Pi$ then $\br^{-1}(\C^\times \text{\rm id}_V)= \bG$ and $\ker(\br)\supseteq \bG'$.
\item[{\rm(b)}]  If no connected component of $\Pi$ is contained in $J_0$ then $\br^{-1}(\C^\times \text{\rm id}_V)= Z(\bG)$ and
        \[
            \ker(\br) = \bigg\{ \prod_{i=1}^{2m-l}t_{i}(c_i)\in Z({\boldsymbol{G}}) \biggm | \prod_{i=1}^{2m-l}c_i^{\la\mu, ~\av_i\ra}=1 \bigg\}\subseteq Z({\boldsymbol{G}}) \subseteq \bT.
        \]
\end{enumerate}
\end{lemma}

The Chevalley anti-involution $\Cai$ of $\g$ induces an anti-involution of the group $\bG$, denoted still by $\Cai$, defined by
\begin{equation*}
    t^\Cai = t \quad\text{and}\quad (\exp(x_\a))^\Cai = \exp((x_\a)^\Cai)  \quad\for t\in \bT, ~x_\a\in\g_\a, ~\a\in\Dr.
\end{equation*}
If $(V,\br)$ is an irreducible highest weight representation of $\bG$ and  $(\, \mid \,)$ a contravariant nondegenerate symmetric bilinear form on $V$, then
\begin{equation*}
    (gx \,|\, y) = (x \,|\, g^\Cai y)\quad\text{ for all } x, y\in V, ~g\in \bG.
\end{equation*}

\subsection{The group $\C^\times\br(\bG)$}
%
%
%

In this subsection $(V, \br)$ is an irreducible highest weight representation of $\bG$ of highest weight $\mu\in P_+$, such that no connected component of $\Pi$ is contained in the type  $J_0$ of $\mu$. Furthermore, $(\, \mid \,)$ is a contravariant nondegenerate symmetric bilinear form on $V$.

Viewing each $c\in\C^\times$ as the natural scalar multiplication by $c$ on $V$, the product
\[
    G:=\C^\times\br(\bG)
\]
is a subgroup of End$(V)$. This group is the unit group of the monoid $M(\br)$ defined in Section \ref{TheMonoid}.
The adjoints of the elements of $G$ with respect to the nondegenerate bilinear form $(\, \mid \,)$ exist, and are contained in $G$. In this way we get an anti-involution $\Cai$ on $G$, which can be described by
\begin{equation*}
    (c\br(g))^\Cai = c\br(g^\Cai), \quad c\in\C^\times, \,g\in\bG.
\end{equation*}
We call this anti-involution the Chevalley anti-involution of $G$.

Because of Lemma \ref{kernel} (b) the group $G$ inherits most structures from the Kac-Moody group $\bG$. The following Corollary is easy to see.
\begin{corollary}\label{twinned BN for G}
{\rm (a)} We obtain a root group data system of $G$ of type $(W,S)$ by
\[
   T : = \C^\times \br({\boldsymbol{T}}) \quad\text{and}\quad U_\a  : = \br({\boldsymbol{U}_\a}), ~ \a\in\Dr.
\]
In particular, $ \bigcap_{\a\in\D^{re}}N_{G}(U_\a)= T $ where $N_{G}(U_\a)$ denotes the normalizer of $U_\a$ in $G$.

{\rm (b)} The associated twin BN-pairs are given by
\[
  B^\pm:=\C^\times \br(\boldsymbol{B}^\pm)   \quad\text{and}\quad  N := \C^\times \br({\boldsymbol{N}}) .
\]
In particular, $ B^+\cap B^-= B^+\cap N= B^-\cap N= T$.
The Coxeter system given by the Weyl group $N/T$ and the set of simple reflections $\{\br(n_i) T\mid i\in\n\}$ is isomorphic to the Coxeter system $(W, S)$.
\end{corollary}

We identify $N/T$ and $W$, by identifying $\br(n_i) T$ and $r_i$ for all $i\in\n$. If $w\in W$ is represented by $n_w\in N$ then $ n_w U_\a n_w^{-1} = U_{w\a}$ for all $\a\in\Dr$, and
\begin{eqnarray}\label{NTAction2}
  n_w V_\eta= V_{w\eta} \quad\text{ for all }\eta\in P(V).
\end{eqnarray}

To investigate the subgroup $T$ of $G$ we need the following result from the theory of free abelian groups. We add its proof for convenience.
\begin{lemma}\label{restriction tori} Let $L$ be a free abelian group of finite rank. Every subgroup $L_1$ of $L$ is a free abelian group of finite rank, and the restriction map $\text{\rm Hom}(L,\C^\times) \to  \text{\rm Hom}(L_1,\C^\times)$ is a surjective morphism of tori.
\end{lemma}
\begin{proof} The Lemma holds trivially for $L_1=\{0\}$. Let $L_1\neq \{0\}$. Clearly, the restriction map is a morphism of tori. By \cite[Theorem II 1.6]{Hu74} there exists a base $x_1,\ldots,x_n$ of $L$, and $d_1,\ldots,d_r\in\Z_>$ where $r\leq n$, such that $d_1|d_2|\cdots|d_r$ and $d_1x_1,\ldots, d_rx_r$ is a base of $L_1$. In particular, $L_1$ is a free abelian group.
Let $\gamma\in  \text{Hom}(L_1,\C^\times)$. We choose $\xi_1,\ldots,\xi_r\in\C^\times$ such that $\xi_1^{d_1}=\gamma(d_1 x_1)$, \ldots, $\xi_r^{d_r}=\gamma(d_r x_r)$, and define $\beta\in \text{Hom}(L,\C^\times)$ by $\beta(x_1) := \xi_1$, \ldots,  $\beta(x_r):= \xi_r $, and $\beta(x_k):=1$ for all $ k\in \{1,\ldots,n\}\setminus\{1,\ldots, r\}$. Then $\beta$ restricts to $\gamma$ on $L_1$.
\hfill$\Box$\end{proof}

The following easy lemma is also useful.
\begin{lemma}\label{add weight and root} For every $\a\in\Delta^{re}$ there exists $\eta\in P(V)$ such that $\eta + \a \in P(V)$.
\end{lemma}
\begin{proof} Suppose that $\eta+\a\not \in P(V)$ for all $\eta\in P(V)$. Then $\boldsymbol{U}_\a$ is contained in the kernel of $\br$, which is contained in $\bT$ by Lemma \ref{kernel} (b). This contradicts $\bT\cap \bU_\a=\{1\}$.
\hfill$\Box$\end{proof}

In the following theorem we show that the center $Z(G)$ is a one-dimensional torus. We describe $T$ as a direct product of the center $Z(G)$ and an $m$-dimensional torus.
\begin{theorem}\label{description of center and T}
\begin{itemize}
\item[\rm (a)] The center of $G$ is $Z(G)=\C^\times id_V $.
\item[\rm (b)] We get an isomorphism of groups $t: \C^\times \times \text{\rm Hom}(Q,\C^\times)\to T$, if we define for $c\in\C^\times$ and $\gamma\in \text{\rm Hom}(Q,\C^\times)$ the endomorphism $t(c,\gamma)$ of $V$ by
\begin{eqnarray*}
  t(c,\gamma) v_\eta := c\gamma(\eta-\mu) v_\eta
\end{eqnarray*}
for all $v_\eta\in V_\eta$, $\eta\in P(V)$.
In particular, we have $t(\C^\times \times \{1_Q\})= Z(G)$, where $1_Q$ denotes the identity of $\text{\rm Hom}(Q,\C^\times)$.
\end{itemize}
\end{theorem}
\begin{proof} Clearly, $\C^\times id_V \subseteq Z(G)$. To show the reverse inclusion let $z\in Z(G)$. Since the representation $\br$ is irreducible, $V$ is spanned by $\br(\bG) V_\mu =G V_\mu$. From  Corollary \ref{twinned BN for G} (a) it follows that $z\in Z(G)\subseteq \bigcap_{\a\in\D^{re}}N_{G}(U_\a)= T$. Therefore, $z$ is the multiplication by a scalar $c\in\C^\times$ on $V_\mu$, $GV_\mu$, and $V$.

For the proof of (b) we denote by $r:\text{Hom}(P,\C^\times)\to\text{Hom}(Q,\C^\times)$ the restriction morphism. For $c\in\C^\times$ and $\beta\in \text{Hom}(P,\C^\times)$ we have
\begin{eqnarray*}
  t(c, r(\beta)) v_\eta = c\beta (\eta-\mu) v_\eta = c\beta (-\mu)\beta(\eta) v_\eta = c\beta (-\mu)  \br(\boldsymbol{t}(\beta) ) v_\eta
\end{eqnarray*}
for all $v_\eta\in V_\eta$, $\eta\in P(V)$. Hence,
\begin{eqnarray*}
   t(c, r(\beta) )) = c\beta (-\mu) \br( \boldsymbol{t} (\beta)) .
\end{eqnarray*}
The map  $r:\text{Hom}(P,\C^\times)\to\text{Hom}(Q,\C^\times)$ is surjective by Lemma \ref{restriction tori}. For $\beta\in \text{Hom}(P,\C^\times)$ the map $\C^\times\to\C^\times$ given by multiplication by $\beta (-\mu) $ is bijective. We conclude that the map $t: \C^\times \times \text{Hom}(Q,\C^\times)\to T$ is well defined and surjective.

Clearly, the map $t$ is a morphism of groups. To show that it is injective let $c\in\C^\times$, $\gamma\in \text{Hom}(Q,\C^\times)$ such that $t(c,\gamma)=id_V$. Evaluating both sides at $v_\mu\in V_\mu\setminus\{0\}$ we get $c=1$.
Now let $\a\in\Pi$. By Lemma \ref{add weight and root} there exists $\eta\in P(V)$ such that $\eta+\a\in P(V)$. Evaluating both sides of $t(1,\gamma) =id_V$ at $v_\eta\in V_\eta\setminus\{0\}$ and $v_{\eta+\a}\in V_{\eta+\a}\setminus\{0\}$ we find
\begin{eqnarray*}
 \gamma(\eta-\mu)=1 \quad \text{ and }\quad  \gamma(\eta-\mu)\gamma(\a) =\gamma(\eta+\a -\mu)=1 .
\end{eqnarray*}
It follows that $\gamma(\a)=1$. Since the lattice $Q$ is spanned by $\Pi$, the homomorphism $\gamma$ is the identity of $ \text{Hom}(Q,\C^\times)$.
\hfill$\Box$\end{proof}

The Birkhoff and Bruhat decompositions of $G$, which correspond to the twinned BN-pairs of Corollary \ref{twinned BN for G} (b), are
\begin{equation*}
    G = \bigsqcup_{w\in W} B^- w B^+ =  \bigsqcup_{w\in W} B^+ w B^-
    \quad\text{and}\quad
    G = \bigsqcup_{w\in W} B^+ w B^+ = \bigsqcup_{w\in W} B^- w B^-.
\end{equation*}

We refer to \cite[Chapter 6.2]{AB08} for the definition and properties of the parabolic subgroups of $G$, which correspond to the twinned BN-pairs of Corollary \ref{twinned BN for G} (b).
For $I\subseteq \Pi$ the standard parabolic subgroups
\begin{equation*}
     P_I^+ := B^+ W_I B^+  \quad \text{ and }\quad P_I^- := B^-W_I B^-
\end{equation*}
coincide with $\C^\times \rho(\boldsymbol{P}_I^+)$ and $\C^\times\rho(\boldsymbol{P}_I^-)$, respectively. From the Levi decompositions of $\boldsymbol{P}_I^+$, $\boldsymbol{P}_I^+$, and from Lemma \ref{kernel} (b) we obtain the following Corollary.
We denote by $ U^\pm$ the subgroup of $G$ generated by the groups $U_\a$, $\a\in\D^{re}_\pm$. It coincides with $\br({\boldsymbol{U}^\pm})$.
\begin{corollary}
The standard parabolic subgroups $P_I^\pm$ admit the Levi decompositions
\begin{equation*}
    P_I ^\pm = L_I \ltimes U^I_\pm ,
\end{equation*}
where the standard Levi subgroup $L_I = P_I^+\cap P_I^-$ is the subgroup of $G$ generated by $T$ and $U_\a$,  $\a\in W_I I$. Furthermore,
$U^I_\pm = \cap_{w\in W_I}wU^\pm w^{-1}$, which are the normal subgroups of $U^\pm$ generated by $U_\a$, $\a\in\D^{re}_\pm\setminus W_I I$. In particular, we have $ B^\pm = T \ltimes U^\pm$.
\end{corollary}

It is often convenient to identify the subset $I\subseteq \Pi$ with the set of indices of the simple roots in $I$. We will do so as needed.
The torus $T$ can be written as a product of subgroups
\begin{equation}\label{T-decomposition 2}
    T = \C^\times T_I T^I
\end{equation}
where
\begin{equation*}
  T_I  :=  \bigg\{  \prod_{i\in I} \br (t_i(c_i) \,\big|\, c_i\in\C^\times \bigg\} \quad \text{ and }\quad
  T^I  :=  \bigg\{\prod_{i \in\{1, \,\ldots, \,2m-l\} \setminus I} \br( t_i(c_i))  \,\big|\,  c_i\in\C^\times \bigg\}.
\end{equation*}
It is convenient to work with (\ref{T-decomposition 2}), although, in general, the product is not direct.

The standard Levi subgroup $L_I$ and its derived group $G_I:=L_I'$ have the Birkhoff and Bruhat decompositions
\begin{equation*}
    L_I = U_I^\pm TN_I U_I^\pm \quad\text{ and }\quad   G_I = U_I^\pm N_I U_I^\pm ,
\end{equation*}
where $U_I^\pm=L_I\cap U^\pm = G_I\cap U^\pm$ is the subgroup of $U^\pm$ generated by $U_\a$, $\a\in W_I I \cap \D^{re}_\pm$, and $N_I=G_I\cap N $ is the subgroup of $N$ generated by $T_I$ and $\br(n_i)$ for $i\in I$. Note also that $G_I$ is generated by the root groups $U_{\pm\a}$, $\a\in I$, and $L_I = T G_I = \C^\times T^I G_I$.

We often omit the index ``$+$" of the groups $B^+$, $U^+$, $P_I^+$, $U_I^+$, $U^I_+$, writing $B$, $U$, $P_I$, $U_I$, $U^I$ instead. \vspace*{1ex}

Classically, when $\bG$ is a semisimple simply connected group, we have $\br(\bG)\subsetneqq G$. Investigating $G$ for some examples of Kac-Moody groups $\bG$, it catches one's eye that $\br(\bG) = G$ is also possible. We call the sublattice of $P$ defined by
\begin{eqnarray*}
   Q^{sat} := \{\eta\in P \mid \text{ there exists }n\in\Z_> \text{ such that } n\eta\in Q \}
\end{eqnarray*}
the saturation of the sublattice $Q$ in $P$.

\begin{lemma}\label{criterium br(G)=G} The following are equivalent:
\begin{itemize}
\item[\rm (a)] $\mu\notin Q^{sat}$.
\item[\rm (b)] $\br(\bT) = T$.
\item[\rm (c)] $\br(\bG) =  G$.
\end{itemize}
\end{lemma}
\begin{proof} Note that $\br(\bT) = T$ is equivalent to $\C^\times id_V\subseteq \br(\bT)$, and $\br(\bG) = G$ is equivalent to $\C^\times id_V\subseteq \br(\bG)$.

We first show that (c) implies (b). From Lemma \ref{kernel} (b) we get $\bT\supseteq \br^{-1}(\C^\times id_V)$. By (c) we obtain
$\br(\bT)\supseteq \br(\br^{-1}(\C^\times id_V)) = \br(G)\cap \C^\times id_V  =\C^\times id_V$, which is (b).

We now show that (b) implies (a). We denote by $r:\text{Hom}(P,\C^\times)\to\text{Hom}(Q,\C^\times)$ the restriction morphism. For $\beta\in \text{Hom}(P,\C^\times)$ we find
\begin{equation*}
   t(\beta(\mu), r(\beta) ) = \br({\boldsymbol t}(\beta) )
\end{equation*}
from the definition of $t(\beta(\mu), r(\beta) )$ in Theorem \ref{description of center and T} (b). If $\C^\times id_V\subseteq \br(\bT)$ then there exists $\beta\in \text{Hom}(P,\C^\times)$, such that
\begin{equation*}
         t(\beta(\mu), r(\beta) )= \br({\boldsymbol t}(\beta) ) = 2\, id_V= t(2, 1_Q).
\end{equation*}
Since the map $t:\C^\times \times \text{Hom}(Q,\C^\times)\to T$ is bijective, $\beta(\mu)=2$ and $\beta(q)=1$ for all $q\in Q$.  For every $n\in \Z_>$ we obtain $\beta(n\mu)=\beta(\mu)^n=2^n\neq 1$. Thus, $\mu\notin Q^{sat}$.

It remains to show that (a) implies (c). If $\mu\notin Q^{sat}$ then $\mathbb{Z}\mu+Q=\mathbb{Z}\mu\oplus Q$. The restriction morphism $\text{Hom}(P,\C^\times)\to \text{Hom}(\mathbb{Z}\mu+Q,\C^\times)$ is surjective by Lemma \ref{restriction tori}. Therefore, for every $c\in \C^\times$ we can choose $\beta_c\in \text{Hom}(P,\C^\times)$, such that $\beta_c(\mu)=c$ and $\beta_c(q)=1$ for all $q\in Q$. It follows that $c\, id_V=\br({\boldsymbol t}(\beta_c))\in\br(\bG)$.
\hfill$\Box$\end{proof}

\begin{corollary}\label{nondeg br(g) neq G} If the generalized Cartan matrix $A$ is nondegenerate, then $\br(\bG)\subsetneqq G$. If $A$ is degenerate then there exists $\mu$ such that $\br(\bG)= G$.
\end{corollary}
\begin{proof} If $A$ is nondegenerate then $Q$, $P$ are lattices of rank $m$, and $Q\subseteq P$. From \cite[Theorem II 1.6]{Hu74} it follows that $Q^{sat}=P$. Hence $\br(\bG)\subsetneqq G$ by Lemma \ref{criterium br(G)=G}.

If $A$ is degenerate, then there exists a connected component $\Pi_0$ of $\Pi$, such that the corresponding generalized Cartan submatrix of $A$ is degenerate. Hence there exists a nonzero tuple $(r_j)_{j\in\Pi_0}\in\mathbb{R}^{|\Pi_0|}$, such that
\begin{eqnarray}\label{eq kernel}
   \a_k( \sum_{j\in\Pi_0 } r_j\a^\vee_j ) =  \sum_{j\in\Pi_0}  r_j a_{j k} = 0 \quad \text{ for all }\quad k \in \Pi.
\end{eqnarray}
Choose $i\in \Pi_0$ such that $r_i\neq 0$. Then $\mu:=\mu_i + \sum_{k\in\Pi\setminus\Pi_0}\mu_k $ is a dominant weight of type $\Pi_0\setminus \{i\}$; moreover, $\Pi_0\setminus \{i\}$ contains no connected component of $\Pi$.

Now assume that $\mu\in Q^{sat}$. Then there exist $n\in\Z_>$ and $q\in Q$, such that $n\mu=q$. Evaluating both sides at $\sum_{j\in\Pi_0} r_j\a^\vee_j $, using (\ref{eq kernel}), we get $n r_i=0$, which contradicts $r_i\neq 0$. Hence $\mu\notin Q^{sat}$. From Lemma \ref{criterium br(G)=G} we find $ \br(\bG) = G $.
\hfill$\Box$\end{proof}

\begin{corollary} Let the generalized Cartan matrix $A$ be indecomposable.
\begin{itemize}
\item[\rm (a)]  If $A$ is of finite type, then $ \br(\bG)\subsetneqq G $.
\item[\rm (b)]  If $A$ is of affine type, then $ \br(\bG) = G $.
\item[\rm (c)]  If $A$ is nondegenerate of indefinite type, then $ \br(\bG)\subsetneqq G $.
\item[\rm (d)]  If $A$ is degenerate of indefinite type, then there exists $\mu$ such that $ \br(\bG)\subsetneqq G $, and there exists $\mu$ such that $ \br(\bG) = G $.
\end{itemize}
\end{corollary}
\begin{proof} Because of Corollary \ref{nondeg br(g) neq G} it remains to show (b), and the first statement of (d).

Let $A$ be of affine type and assume that $\mu\in Q^{sat}$. Then there exist $n\in\Z_>$, $n_1,\ldots, n_m\in\Z$ such that $n\mu=\sum_{i=1}^m n_i\alpha_i$. We obtain
\begin{eqnarray*}
  0\leq n \mu(\alpha^\vee_j)=\sum_{i=1}^m a_{ji}n_i
\end{eqnarray*}
for all $j\in \{1,\ldots,m\}$. From \cite[ Theorem 4.3 (Aff)]{Kac90} we find $\mu(\alpha^\vee_j)=0$ for all $j\in \{1,\ldots,m\}$. Thus $J_0=\Pi$, which contradicts our assumption on $\mu$. Now (b) follows from Lemma \ref{criterium br(G)=G}.

Let $A$ be degenerate of indefinite type. Then $Q^- \cap C \neq\emptyset $ by \cite[Theorem 5.6 c)]{Kac90}. We may choose $\mu\in Q^- \cap C$ since its type is $\emptyset$. Thus $ \br(\bG)\subsetneqq G $ by Lemma \ref{criterium br(G)=G}.
\hfill$\Box$\end{proof}

\subsection{Convex geometry}
We collect from \cite{Ro96} some basic results of convex geometry, which will be used to describe the facial structure of the convex hull of certain Weyl group orbits in $\h_\R^*$.
For a convex set $K$ in $\h_\R^*$ we denote by $\text{aff}(K)$ its affine hull, and by $\text{ri}(K)$ its relative interior.

\begin{lemma}{\rm(\cite[Section 1, and Theorem 6.6]{Ro96})} \label{relativeInteriorProperties}
Let $\sigma$ be a linear transformation from $\hR^*$ to $\hR^*$. If $K$ is a convex set in $\hR^*$, then $\text{\rm aff}(\sigma K) = \sigma \,\text{\rm aff} (K)$, and $\text{\rm ri}(\sigma K) = \sigma \,\text{\rm ri}(K)$.
\end{lemma}

Let ${\F}(H)$ be the set of all faces of a convex set $H$. We give ${\F}(H)$ the partial order
\[
    F \le F' \quad \Leftrightarrow\quad F \subseteq F'.
\]
We agree that any subset of ${\F}(H)$ inherits this order. The set ${\F}(H)$ is a complete lattice, called the {\it face lattice} of $H$, where the meet of two faces is their intersection, and the join is the smallest face containing both faces.

\begin{lemma}{\rm(\cite[Theorem 6.2, Corollary 18.12, and Theorem 18.2]{Ro96})} Let $H$ be a nonempty convex set in $\hR^*$. Then
\begin{equation*}
  H= \bigsqcup_{F\in\F(H)\setminus\{\emptyset\}}  \text{\rm ri}(F),
\end{equation*}
and $\text{\rm ri}(F)\neq\emptyset$ for all $F\in\F(H)\setminus\{\emptyset\}$.
\end{lemma}

\begin{lemma}{\rm(\cite[Section 18]{Ro96})}\label{faceProperties2}
Let $F$ be a face of a convex set $H$, and let $F'$ be a subset of $F$. Then $F'$ is a face of $H$ if and only if $F'$ is a face of $F$.
\end{lemma}

\begin{lemma}{\rm(\cite[Theorem 18.3]{Ro96})}\label{generatingProperties} Let $F$ be a face of a convex set $H$. If $H$ is generated by $K$ as a convex set then $F$ is generated by  $K\cap F$ as a convex set.
\end{lemma}

\section{Geometry of $W$-Orbits\label{s3}}
%
%
In this section we fix an element $\m$ of the fundamental chamber $\overline{C}$ of the Tits cone $X$ of a Kac-Moody algebra $\g(A)$. We set
\begin{eqnarray*}
   J_0: =\{\a\in\Pi\mid \la\mu, \a^\vee \ra = 0\} \quad\text{ and }\quad  J_>:=\Pi\setminus J_0 = \{\a\in\Pi\mid \la\mu, \a^\vee \ra > 0\} ,
\end{eqnarray*}
and call $J_0$ the type of $\mu$.
We call the convex hull
\begin{eqnarray*}
   H:=\text{co}(W\mu)\subseteq \h_\R^*
\end{eqnarray*}
of the Weyl group orbit $W\mu$ in $\h_\R^*$ the {\it orbit hull} of $\mu$. Clearly, $H$ is $W$-invariant.

The action of $W$ on $H$ induces an action of $W$ on the face lattice $\F(H)$ of $H$ by lattice isomorphisms. For each face $F\in \F(H)$ we denote by \begin{equation*}\label{stabWF1}
           W(F) := \{ w\in W \mid w F = F \}
\end{equation*}
its {\it isotropy group}, and by
\begin{equation*}\label{stabWF2}
    W_*(F) := \{ w\in W \mid w\eta = \eta \fall \eta\in F \}
\end{equation*}
its {\it stabilizer}, which is a normal subgroup of $W(F)$. For $w\in W$ we have $  W(wF)=wW(F)w^ {-1}$ and $W_*(wF)=wW_*(F)w^ {-1} $.

A face is called {\it fundamental} if it is either empty or if its relative interior intersects the fundamental chamber $\overline{C}$. The vertex $\{\m\}$ and the orbit hull $H$ are fundamental faces. We set
\begin{equation*}
    \mathcal F := \{F \in {\F}(H) \mid \overline{C}\cap\, \ri F\ne\emptyset \} \cup \{\emptyset\}.
\end{equation*}

We give three examples which illustrate $H$ and its faces for finite, affine, and strongly hyperbolic Kac-Moody algebras. For indefinite, not strongly hyperbolic Kac-Moody algebras the situation is more complicated as can be seen by Corollaries \ref{closed-orbit-hull} and \ref{finite-edges}.
\begin{exam}\label{A2}{\rm
Let $A=
                        \begin{pmatrix}
                            2 & -1 \\
                            -1 & 2
                        \end{pmatrix}.
                        $
Then the Kac-Moody algebra $\g(A)$ is the simple Lie algebra of type $A_2$. Its Tits cone is the whole space $\h_\R^*=\R^2$. Let $\mu=3\m_1+2\m_2\in\overline{C}$. The Weyl group orbit
\[
    W\m = \{\mu, \, \mu-3\a_1,\, \mu-5\a_1-2\a_2, \,\mu-5\a_1-5\a_2, \,\mu-3\a_1-5\a_2, \,\mu-2\a_2 \},
\]
and its convex hull $H$ are indicated in the following picture:\\
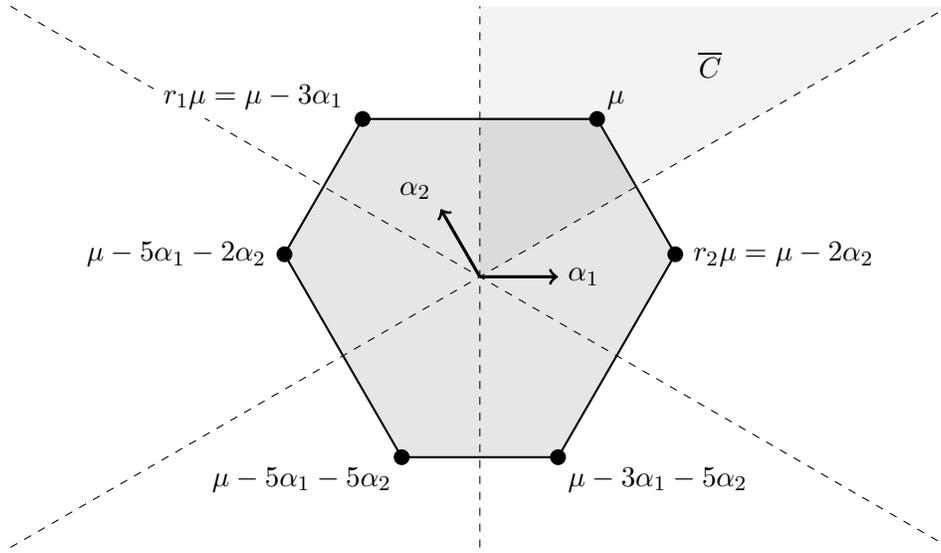
\begin{center}
\begin{tikzpicture}[scale=0.6]
%
%
%
%
\draw[dashed] (0,-6) -- (0, 6);
\draw[dashed] (-10.4 , 6) -- (10.4,-6);
\draw[dashed] (-10.4,-6) -- (10.4,6);
\fill[ultra nearly transparent] (0,0)--(0,6)--(10.4,6);
\node at (5.1,4.7){$\overline{C}$};
%
%
%
\draw[very thick, ->] (0,0) -- (1.732, 0);
\draw[very thick, ->] (0,0) -- (-0.866, 1.5);
%
%
%
\node[right] at  (1.732 , 0) {$\alpha_1$};
\node[above left] at  (-0.866 , 1.5) {$\alpha_2$};
%
%
%
\fill  (2.598, 3.5) circle (5pt)  (4.330, 0.5) circle (5pt) (1.732,-4)  circle (5pt)  (-1.732 ,-4)   circle (5pt)  (-4.330, 0.5)  circle (5pt)  (-2.598, 3.5)  circle (5pt) ;
\node[above right] at  (2.598,3.5) {$\mu$};
\node[right] at (4.330, 0.5)  {$\;r_2\mu=\mu-2\alpha_2$};
\node[below right] at (1.732, -4)   {$\mu-3\alpha_1 - 5\alpha_2\;$};
\node[below left] at (-1.732, -4)   {$\mu-5\alpha_1 - 5\alpha_2$};
\node[left] at (-4.330, 0.5)  {$\mu-5\alpha_1 - 2\alpha_2\;$};
\node[above left, fill = white] at  (-2.8, 3.5)  {$r_1\mu=\mu-3 \alpha_1$};
%
%
%
\draw[thick ] (2.598,3.5) -- (4.330, 0.5) -- (1.732, -4) -- (-1.732, -4) --  (-4.330, 0.5) -- (-2.598, 3.5) -- (2.598, 3.5);
\fill[very nearly transparent] (2.598,3.5) -- (4.330, 0.5) -- (1.732, -4) -- (-1.732, -4) --  (-4.330, 0.5) -- (-2.598, 3.5) -- (2.598, 3.5);
\end{tikzpicture}
\end{center}
The Weyl group orbit $W\mu$ is contained in a circle. Its convex hull $H$ has 14 faces: The empty face, 6 vertices, 6 edges, and $H$. The set of fundamental faces is
\[
    \mathcal F = \{\emptyset,\, \{\mu\}, \,\overline{\mu\,r_1\mu}, \,\overline{\mu \,r_2\mu} , \,H\},
\]
where $\overline{\mu \,r_i\mu}$ is the closed line segment between $\mu$ and $r_i\mu$, $i=1, 2$.

}
\end{exam}
\begin{exam}\label{A11}
{\rm Let
$A=
                        \begin{pmatrix}
                            2 & -2 \\
                            -2 & 2
                        \end{pmatrix}
                        $.
Then the Kac-Moody algebra $\g(A)$ is the affine Lie algebra of type $A_1^{(1)}$. Its Tits cone in $\h_\R^*=\R^3$ consists of the line $\R(\a_1+\a_2)$ and an open half space bounded by $\R\a_1+\R a_2$. The Tits cone is subdivided by the reflection planes like an infinite open book by its pages.
Let $\mu=\mu_1\in\overline{C}$. The Weyl group orbit
\[
    W\mu = \{\mu - n^2\a_1  - n(n+1)\a_2 \mid n \in \Z\},
\]
and its convex hull $H$ are contained in the affine plane $\pi=\mu+\R\a_1+\R\a_2$, which is contained in the Tits cone,  \cite[Appendix]{MP95}.  A part of this affine plane is indicated in the following picture.
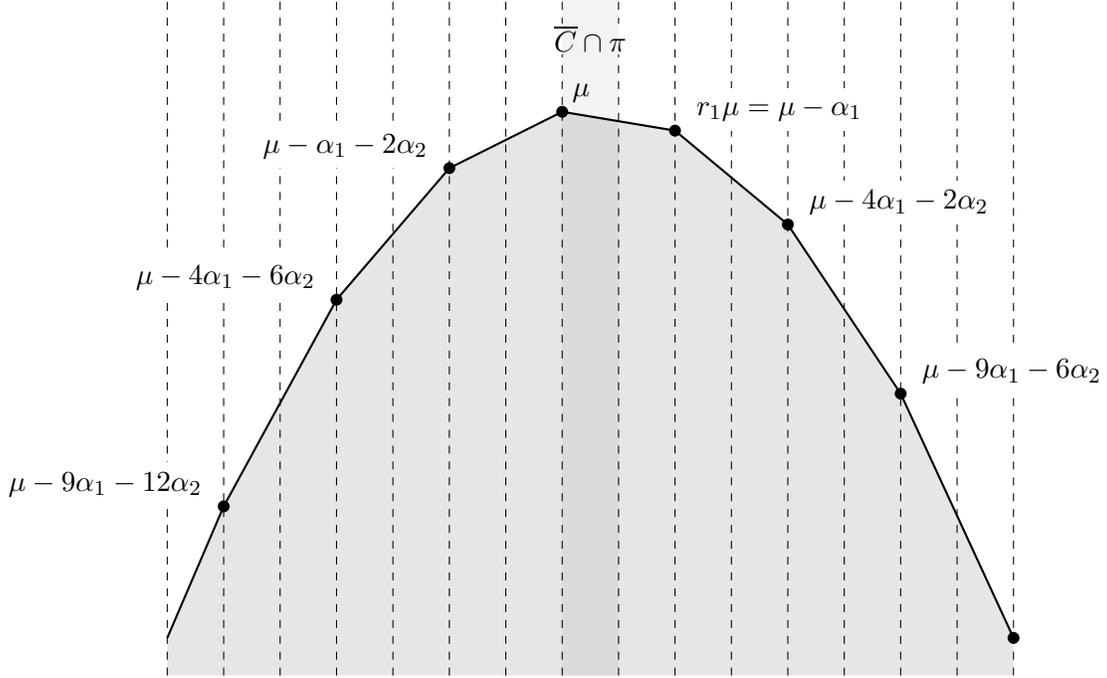
\begin{center}
\begin{tikzpicture}[scale=0.75]
%
%
%
\draw[dashed] (-7,-6) -- (-7, 6);
\draw[dashed] (-6,-6) -- (-6, 6);
\draw[dashed] (-5,-6) -- (-5, 6);
\draw[dashed] (-4,-6) -- (-4, 6);
\draw[dashed] (-3,-6) -- (-3, 6);
\draw[dashed] (-2,-6) -- (-2, 6);
\draw[dashed] (-1,-6) -- (-1, 6);
\draw[dashed] (0,-6) -- (0, 6);
\draw[dashed] (1,-6) -- (1, 6);
\draw[dashed] (2,-6) -- (2, 6);
\draw[dashed] (3,-6) -- (3, 6);
\draw[dashed] (4,-6) -- (4, 6);
\draw[dashed] (5,-6) -- (5, 6);
\draw[dashed] (6,-6) -- (6, 6);
\draw[dashed] (7,-6) -- (7, 6);
\draw[dashed] (8,-6) -- (8, 6);
\node[fill=white,inner sep=1pt] at (0.5,5.3){$\overline{C}\cap\pi$};
\fill[ultra nearly transparent] (0,-6)--(0,6)--(1,6) -- (1, -6);

%
%
%
\fill  (-6, -3) circle (3pt) ;
\fill  (-4, 0.666) circle (3pt) ;
\fill  (-2, 3) circle (3pt) ;
\fill  (0, 4) circle (3pt) ;
\fill  (2, 3.666) circle (3pt) ;
\fill  (4,  2) circle (3pt) ;
\fill  (6, -1) circle (3pt) ;
\fill  (8, -5.333) circle (3pt) ;
%
%
%
\node[above left, fill=white] at  (-6.2,-3) {$\mu-9\a_1-12\a_2$};
\node[above left, fill=white] at  (-4.2, 0.666) {$\mu-4\a_1-6\a_2$};
\node[above left, fill=white] at  (-2.2, 3) {$\mu-\a_1-2\a_2$};
\node[above right] at  (0,4) {$\mu$};
\node[above right, fill=white] at  (2.2, 3.666)  {$r_1\mu=\mu-\a_1$};
\node[above right, fill=white] at  (4.2, 2)  {$\mu-4\a_1-2\a_2$};
\node[above right, fill=white] at  (6.2, -1)  {$\mu-9\a_1-6\a_2$};
%
%
%
%
%

%
%
\draw[thick ] (-7,-5.333) -- (-6, -3) -- (-4, 0.666) -- (-2, 3) -- (0, 4) -- (2, 3.666) -- (4,  2) -- (6, -1) -- (8, -5.333) ;
\fill[very nearly transparent] (-7,-6) -- (-7,-5.333) -- (-6, -3) -- (-4, 0.666) -- (-2, 3) -- (0, 4) -- (2, 3.666) -- (4,  2) -- (6, -1) -- (8, -5.333) -- (8,-6) -- (-7,-6);

\end{tikzpicture}
\end{center}
The Weyl group orbit $W\mu$ is contained in a parabola. Its convex hull $H$ has infinitely many faces. The set of fundamental faces is
\[
   \F = \{ \emptyset,\,\{\mu\}, \,\overline{\mu\,r_ 1\mu} ,\,H \} .
\]
}
\end{exam}

\begin{exam}\label{Aab}
{\rm Let $A=    \begin{pmatrix}
                            2 & -2\\
                            -3 & 2
                        \end{pmatrix} $.
Then $\g(A)$ is a Kac-Moody algebra of indefinite, strongly hyperbolic type. The form of its Tits cone in $\h_\R^*=\R^2$ is an open wedge with the additional point $\{0\}$. Let $\mu= \mu_1 + \mu_2 \in\overline{C}$. For $n\in \Z$ set
\begin{eqnarray*}
  g(n) = f(n)^2+  \sqrt{\frac{2}{3}\,} f(n)f(n-1)\quad \text{and }\quad  h(n) = f(n)^2 + \sqrt{\frac{3}{2}\,} f(n)f(n+1)
\end{eqnarray*}
where
\begin{eqnarray*}
  f(n) = \frac{\sinh(n\theta)}{\sinh(\theta)}  \quad \text{ with }\quad\theta=\text{Arcosh}(\frac{1}{2}\sqrt{6}).
\end{eqnarray*}
The Weyl  group orbit
\[
   W\mu = \{\mu- g(n) \a_1 - h(n) \a_2 ,  \;\mu- g(n+1) \a_1 - h(n)\a_2 \mid n\in\Z \},
\]
\cite[Lemma 2.6]{F14}, and its convex hull $H$ are indicated in the following picture:
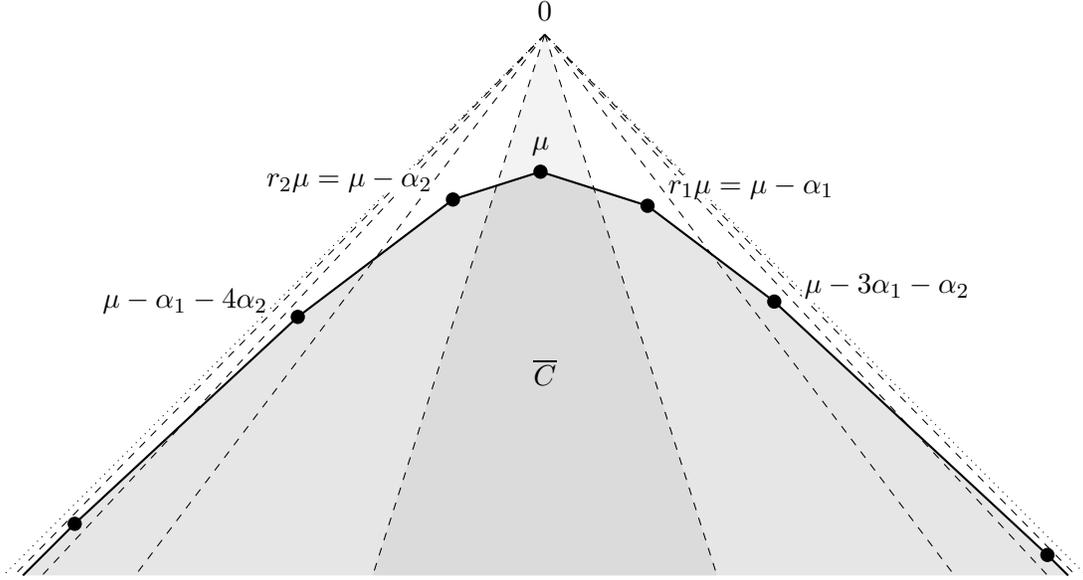
\begin{center}
\begin{tikzpicture}[scale=0.9]
%
%
%
\draw[dotted] (0,0)-- (-8,-8);
\draw[dotted] (0,0)-- (8,-8);
%
%
%
%
%
%
%
%
%
\draw[dashed] (0,0)-- (-7.8422,-8);
\draw[dashed] (0,0)-- (-7.4267,-8);
\draw[dashed] (0,0)-- (-6.051,-8);
\draw[dashed] (0,0)-- (2.5427, -8);
\draw[dashed] (0,0)-- (-2.5427,-8);
\draw[dashed] (0,0)-- (6.051,-8);
\draw[dashed] (0,0)-- (7.4267,-8);
\draw[dashed] (0,0)-- (7.8422,-8);
\fill[ultra nearly transparent] (0,0) -- (-2.5427, -8) -- (2.5427, -8) --  (0,0) ;
\node at (0,-5){$\overline{C}$};
\node[above] at (0,0+0.05) {$0$};
%
%
%
%

\fill (-4.630*1.5,-4.824*1.5 ) circle (3pt);
\fill (-2.433*1.5,-2.785*1.5 ) circle (3pt);
\fill (-0.905*1.5,-1.628*1.5 ) circle (3pt);
\fill (-0.043*1.5,-1.354*1.5 ) circle (3pt);
\fill (1.0112*1.5,-1.690*1.5 ) circle (3pt);
\fill (2.259*1.5,-2.634*1.5 ) circle (3pt);
\fill (4.949*1.5,-5.131*1.5 ) circle (3pt);
%
%
%
\node[above left, fill = white, inner sep=1pt] at (-2.433*1.5-0.4,-2.785*1.5) {$\mu-\alpha_1- 4\alpha_2$};
\node[above left, fill=white, inner sep=0pt] at (-0.905*1.5-0.3,-1.628*1.5+0.1)  {$r_2\mu=\mu- \alpha_2$};
\node[above] at (-0.043*1.5,-1.354*1.5 +0.1) {$\mu$};
\node[above right, fill=white, inner sep=0pt]at (1.0112*1.5+0.3,-1.690*1.5+0.1)  {$r_1\mu=\mu-\alpha_1$} ;
\node[above right, fill=white, inner sep=1pt] at (2.259*1.5+0.4,-2.634*1.5)   {$\mu- 3\alpha_1- \alpha_2$};
%
%
%
\draw[thick ] (-7.71,-8) -- (-4.630*1.5,-4.824*1.5 ) -- (-2.433*1.5,-2.785*1.5 ) -- (-0.905*1.5,-1.628*1.5 ) --  (-0.043*1.5,-1.354*1.5 )  --  (1.0112*1.5,-1.690*1.5 ) -- (2.259*1.5,-2.634*1.5 ) -- (4.949*1.5,-5.131*1.5 ) -- (7.73,-8) ;
\fill[very nearly transparent] (-7.71,-8) -- (-4.630*1.5,-4.824*1.5 ) -- (-2.433*1.5,-2.785*1.5 ) -- (-0.905*1.5,-1.628*1.5 ) --  (-0.043*1.5,-1.354*1.5 )  --  (1.0112*1.5,-1.690*1.5 ) -- (2.259*1.5,-2.634*1.5 ) -- (4.949*1.5,-5.131*1.5 ) -- (7.73,-8) ;
\end{tikzpicture}
\end{center}
The Weyl group orbit $W\mu$ is contained in a hyperbola. Its convex hull has infinitely many faces. The set of fundamental faces is
\[
    \mathcal F = \{\emptyset,\, \{\mu\}, \,\overline{\mu\,r_1\mu}, ,\overline{\mu\,r_2\mu}, \,H\}.
\]
}
\end{exam}

\vspace*{1ex} Classically, for finite Weyl groups, the face lattices of orbit hulls have been investigated by several people. A detailed literature survey can be found in \cite[Section 2]{Kh16}.
The description of $\F$ in Corollary \ref{dfct}, the cross section lattice property of $\F$ in Corollaries \ref{ff-cs} and \ref{ff-sl}, and the descriptions of  $W_*(F)$ and $W(F)$, $F\in\F$, in Theorem \ref{faceIsotropyGroup} have been obtained for finite Weyl groups by W. A. Casselman in \cite[Sections 3 and 4]{Ca}, building on some work of A. Borel, J. Tits in \cite[Sections 12.14  --  12.17]{BT}, and of I. Satake in \cite[Section 2.3]{Sa}.
For orbit hulls of integral weights these results have been shown independently by E. B. Vinberg in \cite[Section 3.1]{V91}.
Equivalent versions have been obtained independently in the theory of $\mathcal J$-irreducible reductive linear algebraic monoids
by M. S. Putcha and L. E. Renner in \cite[Section 4]{PR88}.

The classical approaches for the investigation of the face lattice $\F(H)$ do not generalize to infinite Weyl groups. Instead we combine some ideas of \cite[Section 4]{Li15}, \cite[Section 4]{Mo15b}, and \cite[Section 2]{Mo09}. In addition to the results mentioned above we obtain a description of the relative interiors of the faces of $H$, which even for finite Weyl groups seems to be new. We obtain characterizations for $H\cap\overline{C}$ to be closed, and for $H$ to have finitely many edges containing $\mu$.
We reach a description of the lattice operations of the face lattice $\F(H)$, which even for finite Weyl groups seems to be new.

For $I\subseteq \Pi$ we denote by
\begin{equation*}
  F_I:=\text{co}(W_I\mu)\subseteq \h_\R^*
\end{equation*}
the convex hull of $W_I\mu$ in $\h_\R^*$. To prove that these sets are fundamental faces we use two Lemmas. The first can be found for example in \cite[Proposition 1.11]{Loo}.
\begin{lemma}\label{VinbergLemma}
We have $H\subset \mu - \mathbb R_+\Pi$.
\end{lemma}

The second Lemma generalizes \cite[Lemma 4.2]{Li15}, as well as the fifth paragraph of \cite[Section 3.1]{V91} for orbit hulls where the  Weyl group is finite. M. Dyer obtains  independently the same result by a different proof in \cite[Lemma 2.4 (d)]{Dy13}.

\begin{lemma} \label{zhuoLemma} Let $w\in W$ and $I\subseteq \Pi$ with $w\mu\in \mu-\R_+ I$. Then there exists $u\in W_I$ such that $w\mu=u\mu$.
\end{lemma}
\begin{proof} We use induction on the length $l(w)$. It is clear that the result holds for $l(w)=0$. Now suppose that $l(w)\ge 1$.
By \cite[Proposition 2.20]{AB08} there exist $w^I\in W^I$, $w_I\in W_I$ such that $w =w^I w_I$ and $l(w) =l(w^I)+l(w_I)$.
If $w^I=1$ then $w\in W_I$, and the proof is complete. If $w^I\ne 1$, let $w^I=r_{\gamma_1}\dots r_{\gamma_k}$ be a reduced expression. Then $\gamma_k\notin I$. Moreover, $\a:=w^I\gamma_k<0$ and $w_I^{-1}\gamma_k >0$ by  \cite[Lemma 3.11]{Kac90} and \cite[Lemma 2.15]{AB08}. Since  $\a^\vee =w^I\gamma_k^\vee$ and $w_I^{-1}\gamma_k^\vee =(w_I^{-1}\gamma_k)^\vee >0$ we obtain $\la w\mu, \a^\vee\ra =\la w_I\mu, \gamma_k^\vee\ra = \la\mu,w_I^{-1}\gamma_k^\vee\ra \geq 0$.

If $\la w_I\mu,\ \gamma_k^\vee\ra=0$, then $r_{\gamma_k}w_I\mu =w_I\mu$. Hence,
$
    w\mu= r_{\gamma_1}\dots r_{\gamma_{k-1}}w_I\mu.
$
From the induction hypothesis, there exists an element $u\in W_I$ such that $r_{\gamma_1}\dots r_{\gamma_{k-1}}w_I\mu=u\mu$. Thus, $w\mu=u\mu$.

If $\la w_I\mu,\ \gamma_k^\vee\ra=\la w\mu, \a^\vee\ra > 0$ we write $w\mu=\mu-\sum_{\beta\in I}a_\beta \beta$ with $a_\beta\in\R_+$. By Lemma \ref{VinbergLemma} we have
\[
    r_\a w\mu = \mu-\sum_{\beta\in I}a_\beta \beta  - \la w\mu,\ \a^\vee\ra \a \subseteq  \mu - \R_+ \Pi,
\]
since $r_\a w\mu\in H$. But $\a$ is a negative root, so it is a linear combination of simple roots from $I$. Thus $r_\a\in W_I$ and
$ r_{\alpha}w\mu \in \mu-\R_+I$. Moreover, the length of $ r_{\alpha}w= (w^I r_{\gamma_k})w_I$ is smaller than the length of $w$. From the induction hypothesis, there exists $v\in W_I$ such that $r_\a w\mu = v\mu$. Hence $w\mu= r_\a v\mu$ with $r_\a v\in W_I$.
 \hfill $\Box$\end{proof}

\begin{theorem}\label{fiface}  Let $I\subseteq \Pi$. Then $F_I$ is an exposed fundamental face of H, and
\begin{eqnarray*}
  F_I = H\cap (\mu-\R_+I)= W_I (F_I\cap \overline{C} ) .
\end{eqnarray*}
Moreover, $W_I\mu = F_I\cap W\mu$.
\end{theorem}
\begin{proof} As  $\mu-\mathbb R_+I$ is an exposed face of $\mu-\mathbb R_+\Pi$ we obtain that $H\cap (\mu-\mathbb R_+I)$ is an exposed face of  $H\cap (\mu-\mathbb R_+\Pi)=H$. Now Lemma \ref{generatingProperties} shows that $ H\cap (\mu-\mathbb R_+I)$ is generated by $W\mu\cap (\mu-\mathbb R_+I)$ as a convex set, which coincides with $W_I\mu$ by Lemma \ref{zhuoLemma}. So $ H\cap (\mu-\mathbb R_+I)=F_I$.

Since $F_I$ is the convex hull of $W_I\mu$ we have $W_I F_I=F_I$. Clearly, $W_I (F_I\cap \overline{C})\subseteq F_I$. To show the reverse inclusion let $\eta_1\in F_I$. Then there exist $w\in W$ and $\eta\in\overline{C}$ such that $\eta_1=w\eta$. By Lemma \ref{VinbergLemma} we have  $\eta=w^{-1}\eta_1\in H\subseteq \mu-\R_+ \Pi $. Applying Lemma \ref{VinbergLemma} to $\eta$ in place of $\mu$ we obtain
\begin{eqnarray*}
  \eta_1 = w\eta =\eta-\sum_{\a\in\Pi} a_\a \a = \mu + (\eta-\mu) -\sum_{\a\in\Pi} a_\a \a = \mu -\sum_{\a\in\Pi} b_\a \a  -\sum_{\a\in\Pi} a_\a \a
\end{eqnarray*}
with $a_\a, b_\a\in\R_+$, $\a\in\Pi$. Since $\eta_1 \in F_I = H\cap ( \mu-\R_+I) $ we find that $a_\a=b_\a=0$ for all $\a\in \Pi\setminus I$.
Hence $\eta_1=w\eta\in\eta-\R_+I$. By Lemma \ref{zhuoLemma}, applied for $\eta$ in place of $\mu$, there exists $w_1\in W_I$ such that $\eta_1=w_1\eta$. We also get $\eta=w_1^{-1}\eta_1\in W_I F_I=F_I$. Thus, $\eta\in F_I\cap \overline{C}$, and hence $\eta_1\in W_I(F_I\cap \overline{C})$.

We next show that $F_I$ is fundamental. Choose some $\gamma\in \text{ri}(F_I)$. Because of $F_I=W_I(F_I\cap \overline{C})$ there exists some $w\in W_I$ such that $w\gamma \in F_I\cap\overline{C}$. Moreover, $w\gamma\in w\,  \text{ri}(F_I)= \text{ri}(w F_I)=\text{ri}(F_I)$.
\hfill$\Box$\end{proof}

A nonempty subset $I\subseteq\Pi$ is called {\it  $\mu$-connected} if every connected component of $I$ intersects $J_>$. Note that $I$ is $\mu$-connected if and only if all its connected components are $\mu$-connected. We agree that the empty set is $\mu$-connected.
The next result is obvious.
\begin{proposition}\label{pmuc} An arbitrary union of $\mu$-connected sets is again $\mu$-connected. For $I\subseteq \Pi$ there exists a biggest $\mu$-connected set $I^*$ contained in $I$, which we call the $\mu$-connected part of $I$.

Ordered partially by inclusion, the $\mu$-connected sets form a lattice. The lattice meet and lattice join of two $\mu$-connected sets $I_1$, $I_2$ are given by
\begin{eqnarray*}
   I_1\wedge I_2= (I_1\cap I_2)^* \quad \text{ and }\quad I_1\vee I_2 =I_1\cup I_2.
\end{eqnarray*}
Moreover, $\emptyset$ is the smallest, and $\Pi^*$ is the biggest $\mu$-connected set.
\end{proposition}

For $I\subseteq \Pi$ we set
\begin{equation}
    I_* := \{\a \in J_0 \setminus I^*\mid r_\a r_\beta= r_\beta r_\a \text{ for all } \beta \in I^*\}.
\end{equation}
Note also that for $\a, \beta\in\Pi$, $\a\neq\beta$, we have $ r_\a r_\beta= r_\beta r_\a $ if and only if $\la\beta,\a^\vee \ra=0$, if and only if $\la\a,\beta^\vee \ra=0$.

The $\mu$-connected part $I^*$ is the union of the connected components of $I$ that intersect $J_>$. Let $I_r$ be the union of the connected components of $I$ that do not intersect $J_>$. Then
\begin{equation}
   I= I^*\sqcup I_r \quad \text{ and }\quad I_r\subseteq I_*.
\end{equation}

The following Proposition shows that for our investigations it is sufficient to consider the faces $F_I$ for $\mu$-connected sets $I$.
\begin{proposition}\label{fiiswic}
Let $I$ be a subset of $\Pi$. Then $F_I = F_{I^*}$.
\end{proposition}
\begin{proof} We have $W_I\mu = W_{I^*}W_{I_r}\mu = W_{I^*}\mu$. Hence $F_I=F_{I^*}$.
\hfill$\Box$\end{proof}

Let $I$ be a $\mu$-connected set. Next we determine the affine hull of the face $F_I$. We obtain an interior point of $F_I$, whose isotropy group coincides, as we will see later, with the stabilizer $W_*(F_I)$ of the whole face $F_I$. Such points are useful for some proofs.

The way to do this is to construct a simplex of maximal dimension contained in $F_I$, which is formulated by the following technical lemma.
\begin{lemma}\label{Simplex} Let $I$ be a nonempty $\mu$-connected subset of $\Pi$.

{\rm (a)} There exists a linear order on $I$ with the following property: For every $\beta\in I$ there exists a chain $\gamma_0<\gamma_1<\cdots <\gamma_l=\beta $ in $I$ such that $\gamma_0\in J_>$ and $\gamma_{i-1}$, $\gamma_i$ are adjacent for $i=1,\ldots, l $.

{\rm (b)} Let $ I=\{\beta_1,\ldots,\beta_k\}$ such that $\beta_1<\beta_2<\cdots<\beta_k$ is a linear order as in {\rm (a)}. Then the convex hull $D$ of
\begin{eqnarray*}
   \eta_0:=\mu,\quad \eta_1:=r_{\beta_1}\mu,\quad\eta_2:=r_{\beta_2}r_{\beta_1}\mu,\quad\ldots,\quad \eta_k:=r_{\beta_k}\cdots r_{\beta_1}\mu
\end{eqnarray*}
is a $k$-dimensional simplex, and its affine hull is $\mu + \R I$. Furthermore,
\begin{equation}\label{ri-pt-simplex}
  \text{\rm ri}(D)  \cap  (\mu-\R_{>}I )\cap C_{I_*}  \neq \emptyset .
\end{equation}

\end{lemma}
\begin{proof} To show (a) first assume that $I$ is connected. For $\alpha,\alpha'\in I$ we define the distance $d(\alpha,\alpha')$ to be the minimum of the length $l$ over the set of all chains $\alpha=\gamma_0,\gamma_1,\ldots,\gamma_l=\alpha'$ such that $\gamma_{j-1}, \gamma_j$ are adjacent for $j=1,\ldots ,l$. Since $I$ is $\mu$-connected, the intersection $I\cap J_>$ is not empty. Choose $\gamma_0\in I\cap J_>$, and define $I(p):=\{\gamma\in I \mid d(\gamma_0,\gamma)=p\}$ for $p\in\Z_+$. Then $I(0)=\{\gamma_0\}$ and $I(p)=\emptyset$ at least for $p\geq |I|$. Ordering the elements in each of $I(0), I(1), I(2),\ldots$ linearly according to their indices, and defining $I(0)<I(1)<I(2)<\cdots $, we obtain a linear order on $I$ with the property of (a).

Now let $I$ have the connected components $I_1,I_2,\ldots, I_p$ with $p>1$. Since these components are also $\mu$-connected there exists on every component a linear order with the property of (a). We obtain a linear order on $I$ with the property of (a) by defining $I_1<I_2<\cdots< I_p$.

Next we prove the first part of (b). Set $b_t:=\la \eta_{t-1}, \beta_t^\vee\ra $ for $t=1,\ldots, k$. We first show
\begin{equation}\label{ZhenhengEq}
   \eta_t = \mu -b_1\beta_1 - b_2 \beta_2- \cdots - b_t \beta_t \quad \text{ and }\quad b_1,\,b_2,\,\ldots,\,b_t\in\R_>
\end{equation}
for all $t=1,\ldots, k$ by induction on $t$. If $t=1$, then we have $r_{\beta_1}\mu = \mu - b_1 \beta_1$ with $b_1=\la\mu, \beta_1^\vee\ra>0$ since $\beta_1\in J_>$. Now let $2\leq t\leq k$ and suppose that (\ref{ZhenhengEq}) holds for $t-1$. Then we obtain
\begin{equation*}
   \eta_t= r_{\beta_t}\eta_{t-1}=\eta_{t-1}- \la \eta_{t-1}, \beta_t^\vee\ra  \beta_t= \mu -b_1\beta_1 - \cdots - b_{t-1} \beta_{t-1}-b_t\beta_t ,
\end{equation*}
and
\begin{equation*}
  b_t =\la \eta_{t-1}, \beta_t^\vee\ra =  \la  \mu -b_1\beta_1 - b_2 \beta_2- \cdots - b_{t-1} \beta_{t-1} ,\beta_t^\vee  \ra \geq 0
\end{equation*}
since $\la\mu,\ \beta_t^\vee\ra\ge 0$ and $ -b_j\la \beta_j,\ \beta_t^\vee\ra\geq 0$ for all $j=1,\ldots, t-1$. If $\beta_t\in J_>$ then we have $\la\mu,\ \beta_t^\vee\ra> 0$. If $\beta_t\not\in J_>$ then there exists $s<t$ such that $\beta_s$, $\beta_t$ are adjacent, from which we get $ -b_s\la \beta_s,\ \beta_t^\vee\ra> 0$. Thus $b_t>0$.

Denote by $\text{lin}(D)$ the translation space of the affine hull $\text{aff}(D)$ of $D$. From (\ref{ZhenhengEq}) we obtain $b_t\beta_t=\eta_{t-1}-\eta_t\in  \mbox{lin}(D)\subseteq \R I$ and $b_t\neq 0$ for all $t=1,\ldots, k$. It follows that $ \mbox{lin}(D)= \R I$, and $ \mbox{aff}(D)=\mu+\R I$. In particular, $D$ is a $k$-dimensional simplex.

We now prove the second part of (b) by showing that there exist $\epsilon_1,\ldots \epsilon_k\in \R_>$ such that $ x_t:=  \eta_0 +  \epsilon_1 \eta_1  + \cdots + \epsilon_t \eta_t $ is contained in
\begin{equation}
    \{\,\eta\in (1 +\epsilon_1+\cdots+\epsilon_t) \mu - \R_>\{ \beta_1,\ldots, \beta_t\}   \mid   \la\eta,\beta_j^\vee\ra >0 \text{ for } j=1,\ldots, t   \,\}
 \end{equation}
for $t=1,\ldots , k$. We use induction on $t$. Let $t=1$. Since $\beta_1\in J_>$ we have  $\la \mu,\beta_1^\vee\ra>0 $. Therefore, there exists $\epsilon_1\in\R_>$ such that
\begin{eqnarray*}
  \la \eta_0+\epsilon_1\eta_1,\beta_1^\vee \ra =   \la \mu, \beta_1^\vee\ra +\epsilon_1\la \eta_1,\beta_1^\vee \ra  >0  .
\end{eqnarray*}
We find from  (\ref{ZhenhengEq}) that
\begin{eqnarray*}
    \eta_0+\epsilon_1\eta_1= (1+\epsilon_1)\mu - \epsilon_1b_1\beta_1 \in (1+\epsilon_1)\mu -\R_>\{ \beta_1\}.
\end{eqnarray*}
Now let $2\leq t\leq k$. By induction hypothesis we have
\begin{equation}\label{riSEq2}
   x_{t-1} \in (1+\epsilon_1+\cdots+\epsilon_{t-1})  \mu - \R_> \{\beta_1,  \ldots, \beta_{t-1}\} .
\end{equation}
Here $\la\mu,\beta_t^\vee \ra\geq 0$ and $\la\beta_j,\beta_t^\vee \ra\leq 0$ for all $j<t$. If $\beta_t\in J_>$ then $\la\mu,\beta_t^\vee \ra>0$. If $\beta_t\not\in J_>$ then there exists $s<t$ such that $\beta_s$, $\beta_t$ are adjacent, from which we get $\la \beta_s,\ \beta_t^\vee\ra <0$. We conclude that $\la x_{t-1},\beta_t^\vee\ra>0$. By induction hypothesis we also have $\la x_{t-1},\beta_j^\vee\ra>0$ for all $j<t$. Hence we can choose $\epsilon_t\in\R_>$ such that
\begin{eqnarray*}
   \la x_t ,\beta_j^\vee\ra = \la x_{t-1},\beta_j^\vee\ra +\epsilon_t  \la\eta_t,\beta_j^\vee\ra > 0
\end{eqnarray*}
for all $j=1,\ldots, t$. Furthermore, from (\ref{riSEq2}) and (\ref{ZhenhengEq}) we find that $ x_t=x_{t-1}+\epsilon_t\eta_t \in
 (1+\epsilon_1+\cdots+\epsilon_t)  \mu - \R_> \{\beta_1,  \ldots, \beta_t \}$.

Set  $\epsilon:= 1+\epsilon_1+\cdots+\epsilon_k$. We have shown that
\begin{eqnarray*}
   x_k=  \eta_0 +  \epsilon_1 \eta_1  + \cdots + \epsilon_k \eta_k = \epsilon\mu - c_1\beta_1 - \cdots - c_k\beta_k
\end{eqnarray*}
for some $c_1,\ldots , c_k\in \R_>$, and that $\la x_k, \a^\vee\ra>0$ for all $\a\in I$. For $\a\in\Pi\setminus I$ we find
\begin{eqnarray*}
 \la x_k,\a^\vee\ra  = \epsilon\la\mu,\a^\vee \ra  -  c_1 \la\beta_1,\a^\vee \ra -\cdots - c_k \la\beta_k,\a^\vee \ra\geq 0
\end{eqnarray*}
since $\la\mu,\a^\vee \ra\geq 0$ and  $ \la\beta_j,\a^\vee \ra\leq 0$ for $j=1,\ldots, k$. Furthermore,  $\la x_k,\a^\vee\ra =0$ if and only if $\la\mu,\a^\vee \ra= 0$ and  $\la\beta_j,\a^\vee \ra= 0$ for $j=1,\ldots, k$. This is equivalent to $\a\in I_*$. Hence $\frac{1}{\epsilon}x_k\in \text{ri}(D)  \cap  (\mu-\R_{>}I )\cap C_{I_*}$.
\hfill$\Box$\end{proof}

In the following we set $\R_>\emptyset:=\{0\}$.
\begin{corollary}\label{point-ri-hull}
Let $I\subseteq \Pi$ be $\mu$-connected. Then the face $F_I$ is fundamental with
\begin{eqnarray}\label{ri-pt-face}
  \text{\rm ri}( F_I) \cap  (\mu-\R_{>}I )\cap  C_{I_*} \neq \emptyset .
\end{eqnarray}
The affine hull of $F_I$ is $\mu+ \R I $. In particular, $\dim F_I = |I|$.
\end{corollary}
\begin{proof} The Corollary holds for $I=\emptyset$, where $F_\emptyset=\{\mu\}$ and $\emptyset_*=J_0$. Let $I\neq\emptyset$. Choose a simplex $D$ as in Lemma \ref{Simplex}. Then $D\subseteq F_I$ and $\mu+ \R I=\mbox{aff}(D)\subseteq \mbox{aff}(F_I) $. Furthermore, $\mbox{aff}(F_I) \subseteq \mu+\R I$ by Theorem \ref{fiface}. Since $D$ and $F_I$ have the same affine hull $\mu+\R I$, we obtain $\mbox{ri}(D)\subseteq \mbox{ri}(F_I)$. Thus (\ref{ri-pt-face}) follows from (\ref{ri-pt-simplex}).
\hfill$\Box$\end{proof}

\begin{corollary}\label{mu-con-inj}The map $I \mapsto F_I$  from the set of $\mu$-connected subsets to the set $\F\setminus\{\emptyset\}$ of nonempty fundamental faces  is an isomorphism of partially ordered sets onto its image.
\end{corollary}
\begin{proof} Let $I_1,I_2$ be $\mu$-connected. Clearly, $I_1\subseteq I_2$ implies $F_{I_1}=\text{co}(W_{I_1}\mu)\subseteq \text{co}(W_{I_2}\mu)=F_{I_2}$. If $F_{I_1}\subseteq F_{I_2}$ then by Corollary \ref{point-ri-hull} we obtain $\mu+\R I_1\subseteq\mu+\R I_2$. Hence $I_1\subseteq I_2$.
\hfill$\Box$\end{proof}

We next describe the relative interiors of the nonempty fundamental faces. For $\emptyset\neq K\subseteq \Pi$ we denote by $K^0$ the union of the connected components of $K$ that are of finite type. We set $\emptyset^0=\emptyset$.
\begin{theorem}\label{ri-FI} Let $I$ be $\mu$-connected. Then $ \text{\rm ri}(F_I)  = W_I \,(\text{\rm ri}(F_I) \cap \overline{C} )$ and
\begin{eqnarray*}
 \text{\rm ri}(F_I) \cap \overline{C} =  (\mu-\R_>I)\cap  \overline{C}
   = \bigsqcup_{I_f\subseteq I ,\, (I_f)^0 =I_f} \underbrace{ (\mu-\R_>I)\cap C_{I_*\cup I_f}   }_{\neq\emptyset}\; .
\end{eqnarray*}
\end{theorem}

\begin{proof} The theorem holds trivially for $I=\emptyset$ since $F_\emptyset=\{\mu\}$, $\text{ri}(F_\emptyset)=\{\mu\}$, $W_\emptyset=\{1\}$, $\R_>\emptyset=\{0\}$, and $\emptyset_*=J_0$. We divide the proof for $I\neq\emptyset$ into several parts.

(i) We prove $\text{ri}(F_I)=W_I (\text{ri}(F_I)\cap\overline{C})$. By Theorem \ref{fiface} we have $F_I=W_I(F_I\cap \overline{C})$, from which we get $w F_I=F_I$ for all $w\in W_I$. By Lemma \ref{relativeInteriorProperties} we conclude that $w\,\text{ri}(F_I) = \text{ri}(F_I)$ for all $w\in W_I$. Thus $W_I (\text{ri}(F_I)\cap\overline{C})\subseteq \text{ri}(F_I)$. To show the reverse inclusion let $\eta\in \text{ri}(F_I)$. Then there exists $w\in W_I$ such that $w^{-1}\eta\in w^{-1}\,\text{ri}(F_I)\cap \overline{C}=\text{ri}(F_I)\cap\overline{C}$.

(ii) We show that
\begin{equation*}
      (\mu- \R_> I) \cap \overline{C} =  \bigsqcup_{I_f\subseteq I ,\, (I_f)^0=I_f} (\mu-\R_>I)\cap C_{I_*\cup I_f}  .
\end{equation*}
The union on the right is disjoint and contained in the set on the left. To show the reverse inclusion let $\eta=\mu- \sum_{\a\in I} m_\a \a \in \overline{C}$ with $m_\a\in\R_>$ for $\a\in I$. If $\beta\in\Pi\setminus I$ then
\begin{eqnarray*}
     \la \eta , \beta^\vee\ra  = \underbrace{ \la \mu, \beta^\vee \ra}_{\geq 0}- \sum_{\a\in I}m_\a \underbrace{\la\a, \beta^\vee\ra}_{\leq 0} =0
\end{eqnarray*}
if and only if $\la \mu, \beta^\vee \ra=0$ and $\la\a, \beta^\vee\ra=0$ for all $\a\in I$. This is equivalent to $\beta\in I_*$.

Let $K$ be a component of $\{\beta\in I \mid   \la \eta , \beta^\vee\ra =0 \}$. For every $\beta\in K$ we have
\begin{equation}\label{C-partition-eq1}
     \sum_{\a\in K} m_\a \la\a, \beta^\vee\ra  = \underbrace{ \la \mu, \beta^\vee \ra}_{\geq 0}- \sum_{\a\in I\setminus K }m_\a \underbrace{\la\a, \beta^\vee\ra}_{\leq 0} \geq 0 .
\end{equation}
From \cite[Theorem 4.3]{Kac90} it follows that either $K$ is of finite type, or that $K$ is of affine type and in (\ref{C-partition-eq1}) we have equality for all $\beta\in K$. In the latter case we conclude that $\la \mu, \beta^\vee \ra=0$ and $\la\a, \beta^\vee\ra=0$ for all $\a\in I\setminus K$ and $\beta\in K$. Hence $K$ is a component of $I$, which is contained in $J_0$. However, this is impossible because every component of $I$ is $\mu$-connected.

(iii) We show that $ (\mu- \R_> I) \cap C_{I_*}\subseteq F_I $ by contradiction. Let $\eta \in  (\mu- \R_> I) \cap C_{I_*}$ and suppose that $\eta\notin F_I$. By Corollary \ref{point-ri-hull} there exists
\begin{eqnarray*}
 \eta_1\in  \mbox{ri}( F_I) \cap (\mu-\R_{>}I )\cap C_{I_*}.
\end{eqnarray*}
Since $\mu\in F_I$ and $I_*\subseteq J_0$ also
\begin{eqnarray*}
 \eta_s:=s\eta_1+ (1-s)\mu \in  \mbox{ri}( F_I) \cap (\mu-\R_{>}I ) \cap C_{I_*}
\end{eqnarray*}
for all $0< s\leq 1$. We choose some $s$ such that $\eta_s - \eta\in\R_> I$, and consider the line segment
\begin{equation*}
\overline{\eta_s \eta}\subseteq  (\mu+\R I)\cap  C_{I_*} \subseteq \{ \,\eta' \in \h_\R^* \mid   \la\eta',\a^\vee\ra>0 \text{ for all } \a\in I\} .
\end{equation*}
Since $\eta_s\in\text{ri}(F_I)\subseteq F_I$, $\eta\notin F_I$, and $F_I$ is convex, there exists some $\eta_b\in \overline{\eta_s \eta}$, $\eta_b\neq \eta_s$, such that $\overline{\eta_s \eta_b}\setminus\{\eta_b\}\subseteq F_I$ and  $(\overline{\eta_b\eta}\setminus \{\eta_b\})\cap F_I=\emptyset$.

Since $\overline{\eta_s \eta}$ is compact and $I$ is finite there exists $\epsilon>0$ such that $\la\eta',\a^\vee\ra\geq \epsilon$ for all $\eta'\in \overline{\eta_s \eta}$ and $\a\in I$. There exist $\eta_+ \in  \overline{\eta_s \eta_b}\setminus\{\eta_b\}$ and $\eta_-\in \overline{\eta_b\eta}\setminus F_I$ such that
\begin{eqnarray*}
   \eta_+ - \eta_- \in \Big\{ \sum_{\a\in I} t_\a \a \mid  0\leq t_\a \leq \frac{\epsilon}{|I|}\Big\}  \subseteq \Big\{ \sum_{\a\in I} s_\a\la\eta_+,\a^\vee\ra\a\mid  0\leq s_\a\leq \frac{1}{|I|} \Big\} .
\end{eqnarray*}
But then we get
\begin{eqnarray*}
    \eta_- \in \eta_+ - \Big\{ \sum_{\a\in I} s_\a\la\eta_+,\a^\vee\ra\a\mid  s_\a\in \R^+, \sum_{\a\in I} s_\a\leq 1 \Big\}
               =  \text{co}(\eta_+,r_\a\eta_+\mid\a\in I)   \subseteq F_I,
\end{eqnarray*}
which contradicts $\eta_-\notin F_I$.

(iv) We refine (iii) by showing that $ (\mu- \R_> I) \cap C_{I_*}\subseteq \text{ri}(F_I )$. Let $\eta_0\in  (\mu- \R_> I) \cap C_{I_*}$. Let $\eta\in F_I$, $\eta\neq\eta_0$. For $\epsilon \in \R_+$ we set $  \eta_\epsilon := \eta_0+ \epsilon (\eta_0-\eta) $. By Corollary \ref{point-ri-hull} and the definition of $I_*$ we get
\begin{equation*}
    \eta_0 - \eta \in  \R I  =  \R I \cap \{ \,\eta'\in \h_\R^* \;\mid \la\eta',\a^\vee\ra=0 \text{ for all } \a\in I_*\} .
\end{equation*}
Since $\mu-\R_>I$ is open in $\mu+\R I$, and $C_{I_*}$ is open in $\{ \,\eta'\in \h_\R^* \;\mid \la\eta',\a^\vee\ra=0 \text{ for all } \a\in I_*\}$ there exists $\epsilon_0>0$ such that $ \eta_\epsilon\in  (\mu- \R_> I) \cap C_{I_*}$ for all $0\leq\epsilon\leq\epsilon_0$. From \cite[Theorem 3.5]{B83} it follows that $\eta_0\in \text{ri}(F_I)$.

(v) Let $\emptyset\neq I_f\subseteq I$ such that $(I_f)^0=I_f$. We show that $ (\mu- \R_> I) \cap C_{I_*\cup I_f }\subseteq \text{ri}(F_I)$. Let $\eta_0\in  (\mu- \R_> I) \cap C_{I_*\cup I_f }$. From \cite[Theorem 4.3 (Fin)]{Kac90} it follows that there exists $\gamma\in \R_> I_f$ such that $\la\gamma,\a^\vee\ra>0$ for all $\a\in I_f$. For $t\in\R_+$ we set $\eta_t:=\eta_0+t\gamma $. By \cite[Lemma 3.50]{Mo15b} there exists $t_0 >0$ such that
\begin{eqnarray*}
  \eta_t \in C_{I_*} \quad\text{and}\quad \frac{1}{|W_{I_f}|}\sum_{w\in W_{I_f}} w\eta_t=\eta_0
\end{eqnarray*}
for all $0<t\leq t_0$. Since $\eta_0\in\mu-\R_> I$ and $I_f\subseteq I$ we can choose $0<t\leq t_0$ such that $\eta_t=\eta_0+t\gamma\in \mu-\R_> I$. By (iv) we obtain $\eta_t\in\text{ri}(F_I)$. Since $w\eta_t\in w \,\text{ri}(F_I)=\text{ri}(F_I)$ for all $w\in W_{I_f}$ by (i), we find
\begin{eqnarray*}
 \eta_0=\frac{1}{|W_{I_f}|}\sum_{w\in W_{I_f}} w\eta_t\in \text{ri}(F_I).
\end{eqnarray*}

(vi) Let $I_f\subseteq I$ such that $(I_f)^0=I_f$. We show that $  (\mu- \R_> I) \cap C_{I_*\cup I_f} \neq\emptyset$. By Corollary \ref{point-ri-hull} there exists $ \eta\in (\mu-\R_{>}I )\cap  C_{I_*} $. For $w\in W_{I_f}$ we get $w\eta\in \eta-\R_+ I_f\subseteq \mu-\R_>I$. From \cite[Lemma 3.50]{Mo15b} it follows that
\begin{eqnarray*}
 \frac{1}{|W_{I_f}|}\sum_{w\in W_{I_f}} w\eta \in     (\mu- \R_> I) \cap C_{I_*\cup I_f} .
\end{eqnarray*}

(vii) From (ii) and (iv), (v) we get $ (\mu-\R_> I)\cap \overline{C} \subseteq  \text{ri}(F_I)\cap \overline{C}$. To prove the reverse inclusion it is sufficient to show $\text{ri}(F_I)\subseteq \mu-\R_> I $.
From Theorem \ref{fiface} it follows that $F_I\subseteq \mu-\R_+I$. By Corollary \ref{point-ri-hull}, the affine hull of $F_I$ is $\mu+\R I$, which coincides with the affine hull of $\mu-\R_+I$. Hence $\text{ri}(F_I)\subseteq \text{ri}(\mu-\R_+I)=\mu-\R_> I $.
\hfill$\Box$\end{proof}

The following Proposition has been obtained when $\mu$ is an integral dominant weight and $I=\Pi^*$ in \cite[Lemma 11.2]{Kac90} by Lie theoretic methods.
\begin{proposition}\label{H-dec-mu-c} Let $I$ be $\mu$-connected. We have
\begin{equation}\label{H-dec-mu-con}
      F_I\subseteq \bigsqcup_{\mu\text{-connected } K\subseteq I} \mu-\R_> K .
\end{equation}
\end{proposition}
\begin{proof} Clearly, the union is disjoint. If  $K_1,\ldots ,K_p$ are $\mu$-connected and $r_1,\ldots, r_p\in \R_>$ such that $r_1+\cdots+r_p=1$ then $r_1(\mu-\R_>K_1) +\cdots+ r_p(\mu-\R_>K_p)\subseteq \mu-\R_>(K_1\cup\cdots\cup K_p)$, and $K_1\cup\cdots\cup K_p\subseteq I$ is $\mu$-connected by Proposition \ref{pmuc}. Therefore, it is sufficient to show that $W_I\mu$ is contained in the union of (\ref{H-dec-mu-con}).

Let $w\in W_I$. Then $w\mu\in  \mu-\R_> K$ for some $K\subseteq I$. By Lemma \ref{zhuoLemma} there exists $w_K\in W_K$ such that $w\mu=w_K\mu$. Inserting the decomposition $w_K=w^* w_*$ with $w^*\in W_{K^*}$ and $w_*\in W_{K_r}\subseteq W_{K_*}$, we get $w\mu=w_K\mu=w^*\mu\in\mu -\R_+ K^*$. It follows that $K\subseteq K^*$. Hence $K$ is $\mu$-connected.
\hfill$\Box$\end{proof}

The decompositions of the next corollary are key for the proofs of several of our results. When $\mu$ is an integral dominant weight and $I=\Pi^*$, the second decomposition of $H\cap\overline{C}$ implies the description \cite[Propositions 11.2 a)]{Kac90} of the weights of an irreducible highest weight module of highest weight $\mu$, given by V. Kac and D. Peterson. The proof of \cite[Proposition 11.2 a)]{Kac90} uses Lie theoretic and integral methods, and does not generalize to our situation.

\begin{corollary}\label{C-partition2} Let $I$ be $\mu$-connected. Then
\begin{equation*}
   F_I\cap \overline{C}= \bigsqcup_{\mu\text{-connected }K\subseteq I} \text{ri}(F_K)\cap \overline{C} =  \bigsqcup_{\mu\text{-connected }K\subseteq I} (\mu-\R_> K)\cap \overline{C} .
\end{equation*}
\end{corollary}
\begin{proof} From Proposition \ref{H-dec-mu-c} and Theorem \ref{ri-FI} it follows that
\begin{equation*}
      F_I \cap \overline{C} \subseteq \bigsqcup_{\mu\text{-connected }K\subseteq I} (\mu-\R_> K)\cap \overline{C} = \bigsqcup_{\mu\text{-connected }K\subseteq I} \text{ri}(F_K)\cap \overline{C} \subseteq F_I\cap \overline{C}.
\end{equation*}
\hfill$\Box$\end{proof}

Equation (\ref{eq-KPL}) below has been obtained for $I=\Pi^*$ and $(J_0)^0=J_0$ by E. Looijenga in \cite[Corollary 1.14, Proposition 2.4]{Loo} by a direct, involved proof. When $\mu$ is in addition an integral dominant weight, this equation implies the description \cite[Propositions 11.2 b)]{Kac90} of the weights of an irreducible highest weight module of highest weight $\mu$, given by V. Kac and D. Peterson. Our proof follows that of \cite[Proposition 11.2 b)]{Kac90} with more details.
\begin{corollary}\label{full-face-riX} Let $I$ be $\mu$-connected. If $(I\cap J_0)^0=I\cap J_0$ then
\begin{eqnarray}\label{eq-KPL}
    F_I\cap \overline{C} =  (\mu-\R_+ I)\cap \overline{C}  =  (\mu-\R_+ I)\cap \overline{C}_{I_*}.
\end{eqnarray}
\end{corollary}
\begin{proof} From Corollary \ref{C-partition2} we get
\begin{eqnarray*}
  F_I\cap\overline{C}= \bigsqcup_{\mu\text{-connected } K\subseteq I} (\mu-\R_>K)\cap \overline{C}\subseteq (\mu-\R_+ I)\cap \overline{C}.
\end{eqnarray*}
To show the reverse inclusion let $\eta\in (\mu-\R_+ I)\cap\overline{C}$. Since $\mu\in F_I\cap\overline{C}$ we may assume $\eta\neq\mu$. There exists $\emptyset\neq K\subseteq I$ such that $\eta$ is of the form $\eta=\mu- \sum_{\a\in K}m_\a \a \in\overline{C}$ with $m_\a\in\R_>$ for all $\a\in K$.

Suppose that $K$ is not $\mu$-connected. Then there exists a connected component $L$ of $K$ such that $L\subseteq K\cap J_0\subseteq I\cap J_0$. For every $\beta\in L$ we find
\begin{equation*}
     0\leq\la \eta,\beta^\vee \ra= \underbrace{ \la \mu, \beta^\vee \ra}_{= 0} -  \sum_{\a\in L} m_\a \la\a, \beta^\vee\ra - \sum_{\a\in K\setminus L}m_\a \underbrace{\la\a, \beta^\vee\ra}_{= 0}  .
\end{equation*}
From \cite[Theorem 4.3]{Kac90} it follows that $L$ is of nonfinite type. This contradicts that $L$ is contained in $I\cap J_0$, which
is a union of components of finite type.

Obviously, $  (\mu-\R_+ I)\cap\overline{C} \supseteq (\mu-\R_+ I)\cap \overline{C}_{I_*}$. The reverse inclusion follows by the definition of $I_*$.
\hfill$\Box$\end{proof}

We describe the closeness of the intersection of the orbit hull with the closed fundamental chamber.
%
\begin{corollary}\label{closed-orbit-hull} The following are equivalent.
\begin{itemize}
\item[\rm (a)] $(\Pi^*\cap J_0)^0=\Pi^*\cap J_0$.
\item[\rm (b)] $H\cap \overline{C}$ is closed.
\end{itemize}

In particular, if the Kac-Moody algebra $\g(A)$ is of finite, affine, or strongly hyperbolic type, then $H\cap\overline{C}$ is closed. 
\end{corollary}
\begin{proof} 
If we have $(\Pi^*\cap J_0)^0=\Pi^*\cap J_0$ then $H\cap \overline{C}=(\mu-\R_+ \Pi^*)\cap \overline{C}$ by Corollary \ref{full-face-riX}. Hence $H\cap \overline{C}$ is closed.

Now let $(\Pi^*\cap J_0)^0\neq\Pi^*\cap J_0$, and choose a component $K$ of $\Pi^*\cap J_0$ of nonfinite type. By \cite[Theorem 4.3]{Kac90} there exists $\gamma\in \R_> K$ such that $\la\gamma,\a^\vee\ra\leq 0$ for all $\a\in K$. Clearly,  $\la\gamma,\a^\vee\ra\leq 0$ for all $\a\in\Pi\setminus K$. By Theorem \ref{ri-FI} there exists $\eta\in  (\mu-\R_> \Pi^*)\cap\overline{C}=  \text{ri}(H)\cap\overline{C}$. For $0\leq t \leq 1$ we set $\tau_t:= (1-t) \mu + t \eta- \gamma$. Then
\begin{eqnarray*}
    \tau_t\in (\mu-\R_> \Pi^*)\cap\overline{C} =\text{ri}(H)\cap\overline{C} \subseteq H\cap \overline{C}
\end{eqnarray*}
for all $0<t\leq 1$, and $\tau_0=\mu-\gamma\in  (\mu-\R_> K)\cap \overline{C}$.  Since $K$ is not $\mu$-connected it follows from Corollary \ref{C-partition2} for $I=\Pi^*$ that $\tau_0$ is not contained in $H\cap \overline{C}$. So $H\cap \overline{C}$ is not closed.


Let $\g(A)$ be of finite, affine, or strongly hyperbolic type. If $J_0=\Pi$ then $\Pi^*=\emptyset$. Thus $\Pi^*\cap J_0=\emptyset$. If $J_0\neq\Pi$ then $\Pi^*\cap J_0$ is a proper subset of $\Pi$. Hence it is either empty or a union of components of finite type.
\hfill$\Box$\end{proof}

Now we can show that the map of Corollary \ref{mu-con-inj} is onto.
\begin{corollary}\label{dfct} The map $I \mapsto F_I$  from the set of all $\mu$-connected subsets to the set $\F\setminus\{\emptyset\}$ of all nonempty fundamental faces is an isomorphism of partially ordered sets.
\end{corollary}
\begin{proof} By Corollary \ref{mu-con-inj} it remains to show that the map is surjective. If $F$ is a nonempty fundamental face, then $\text{ri}(F)\cap \overline{C}\neq\emptyset$. Since
\begin{eqnarray*}
   H\cap \overline{C}=\bigsqcup_{\mu\text{-connected } I}  \text{ri}(F_I) \cap \overline{C}
\end{eqnarray*}
by Corollary \ref{C-partition2}, there exists some $\mu$-connected set $I$ such that $\text{ri}(F)\cap \text{ri}(F_I)\neq\emptyset$. Thus $F=F_I$.
\hfill$\Box$\end{proof}

To avoid case distinctions between the empty fundamental face and nonempty fundamental faces in the following theorem on the stabilizers and isotropy groups, and in many of the results of Section \ref{TheMonoid} we introduce some notation. For $F\in\F$ we set
\begin{eqnarray}\label{typeF3}
         \lambda^*(F) :=\left\{   \begin{array}{ccl}
         I  & \text{if} & F=F_I ,\;I \;\mu\text{-connected}, \\
          \emptyset & \text{if} &F=\emptyset .
           \end{array} \right. \\
      \lambda_*(F) :=  \left\{   \begin{array}{ccl}
         I_*  & \text{if} & F=F_I,\;I \;\mu\text{-connected} ,\\
           \Pi & \text{if} &F=\emptyset .  \label{typeF2}
\end{array}\right.
\end{eqnarray}
Furthermore, for $F\in\F$ we set
\begin{equation}\label{typeF1}
        \lambda(F) := \lambda_*(F)\cup \lambda^*(F)=\left\{   \begin{array}{ccl}
         I\cup I_*  & \text{if} & F=F_I ,\;I \;\mu\text{-connected}, \\
           \Pi & \text{if} &F=\emptyset .
\end{array}\right.
\end{equation}
The map $\lambda: \mathcal F \rightarrow 2^\Pi$ is called the {\it type map}.

The next result describes the isotropy groups and stabilizers of the fundamental faces.
\begin{theorem}\label{faceIsotropyGroup} If $F\in\F$ then
\begin{equation*}
  W_*(F) = W_{\lambda_*(F)} \quad\text{ and }\quad  W(F) = W_{\lambda(F)} = W_{\lambda_*(F)} \times W_{\lambda^*(F)}.
\end{equation*}
In particular, the sets (\ref{typeF2}) and (\ref{typeF1}) can be given by
\begin{eqnarray*}
     \lambda_*(F) =  \{\a \in \Pi \mid r_\a \eta = \eta ~\fall \eta\in F\} \quad\text{ and }\quad  \lambda(F) = \{\a\in \Pi \mid r_\a  F = F\}  .
\end{eqnarray*}
\end{theorem}
\begin{proof} Clearly, the theorem holds for $F=\emptyset$. Let $F=F_I$ where $I$ is a $\mu$-connected set. The elements of $W_I$ and $W_{I_*}$ commute because $I$ and $I_*$ are separated. Hence $W_{I\cup I_*} = W_{I} \times W_{I_*}$. Furthermore, since $F_I$ is the convex hull of $W_I\mu$, and the elements of $W_{I_*}\subseteq W_{J_0}$ fix $\mu$, we have $W_{I_*}\subseteq W_*(F_I)\subseteq W(F_I)$ and $W_I\subseteq W(F_I)$.

From Corollary \ref{point-ri-hull}, there exists $\eta \in \text{ri}( F_I) \cap  (\mu-\R_{>}I )\cap  C_{I_*}$. The isotropy group of $\eta$ is $W_{I_*}$ by \cite[Proposition 3.12 a)]{Kac90}. Hence $W_*(F_I)\subseteq W_{I_*}$.

Now let $w\in W(F_I)$. Then $w\eta\in F_I$. By Theorem \ref{fiface} there exist $w'\in W_I$ and $\eta'\in F_I\cap\overline{C}$ such that $w\eta= w'\eta'$. Since every Weyl group orbit intersects $\overline{C}$ in only one point we get $\eta'=\eta$. Thus $w\eta= w'\eta$, from which it follows that $w\in w'W_{I_*}\subseteq W_I W_{I_*}= W_{I\cup I_*}$.
\hfill$\Box$\end{proof}

Next, we describe the inclusion of nonempty faces.
\begin{proposition}\label{incl-ffaces} Let $I, I'$ be $\mu$-connected and $w,w'\in W$. Then $w F_I\subseteq w' F_{I'}$ if and only if $I\subseteq I'$ and $w^{-1}w'\in W_{I_*}W_{I'}=W(F_I)W(F_{I'})$.
\end{proposition}
\begin{proof} If $I\subseteq I'$ then $I_*\supseteq I'_*$, and from Theorem \ref{faceIsotropyGroup} we obtain
\begin{equation*}
   W(F_I) W(F_{I'}) = W_{I\cup I_*}W_{I_*'}W_{I'} = W_{I\cup I_*}W_{I'} =  W_{I_*}W_I W_{I'} =  W_{I_*}W_{I'} .
\end{equation*}

By Corollary \ref{point-ri-hull} we find that there exists $\eta \in \text{ri}( F_I) \cap  (\mu-\R_{>}I )\cap  C_{I_*}$. If $wF_I\subseteq w' F_{I'}$ then $(w')^{-1}w \eta\in F_{I'}$. By Theorem \ref{fiface}  there exist $\eta'\in F_{I'}\cap\overline{C}$ and $w''\in W_{I'}$ such that $(w')^{-1}w \eta=w''\eta'$. Since every $W$-orbit intersects $\overline{C}$ in only one point we get $\eta=\eta'$. Hence $(w')^{-1}w \in w'' W_\eta\subseteq W_{I'}W_{I_*}$.

Now suppose that $(w')^{-1}w \in W_{I'}W_{I_*}$. Write $(w')^{-1}w =ab$ with $a \in  W_{I'}$ and $b\in W_{I_*}$. Then $(w')^{-1}w  F_I\subseteq  F_{I'}$ is equivalent to $F_I=b F_I\subseteq a^{-1}F_{I'}=F_{I'}$ by Theorem \ref{faceIsotropyGroup}, which in turn is equivalent to $I\subseteq I'$ by Corollary \ref{dfct}.
\hfill$\Box$\end{proof}

\begin{corollary}\label{ff-cs} The set of fundamental faces $\F$ is a cross-section for the action of $W$ on $\F(H)$, that is, every face of $H$ is $W$-equivalent to a unique fundamental face of $H$.
\end{corollary}
\begin{proof} The face $\emptyset\in\F(H)$ is $W$-equivalent only to $\emptyset\in\F$. Let $F\in \F(H)\setminus\{\emptyset\}$ and $\eta_0\in \text{ri}(F)$. There exists $w\in W$ such that $w\eta_0\in\overline{C}$. Since
\begin{eqnarray*}
   H\cap \overline{C}=\bigsqcup_{\mu\text{-connected } K}  \text{ri}(F_K) \cap \overline{C}
\end{eqnarray*}
by Corollary \ref{C-partition2}, there exists some $\mu$-connected set $K$ such that $\text{ri}(F_K)\ni w\eta_0\in w\,\text{ri}(F)=\text{ri}(wF)$. Therefore, $wF=F_K$.

If also $w' F= F_{K'}$ for $w'\in W$ and some $\mu$-connected set $K'$, then $w^{-1}F_K= (w')^{-1}F_{K'}$. From Proposition \ref{incl-ffaces} we get $K=K'$.
\hfill$\Box$\end{proof}

We need also the following refinement which follows immediately from Proposition \ref{incl-ffaces} and Corollary \ref{ff-cs}.
\begin{corollary}\label{ffl2} Let $F$ be a fundamental face. Then every face contained in $F$ is $W(F)$-equivalent to a unique fundamental face contained in $F$.
\end{corollary}

From the Corollaries \ref{ff-cs}, \ref{dfct}, and \ref{point-ri-hull} we obtain that the edges of $H$ containing the vertex $\mu$ are given by
\begin{equation*}
 \overline{\mu \, r_\a\mu} \quad \text{ where }\quad \a\in W_{J_0}J_>= \bigsqcup_{\beta\in J_>}W_{J_0}\beta  .
\end{equation*}

\begin{corollary}\label{finite-edges} The following are equivalent.
\begin{itemize}
\item[\rm (a)] $(\Pi^*\cap J_0)^0=\Pi^*\cap J_0$.
\item[\rm (b)] There are only finitely many edges of $H$ that contain the vertex $\mu$.
\end{itemize}

In particular, if the Kac-Moody algebra $\g(A)$ is of finite, affine, or strongly hyperbolic type, then there are only finitely many edges of $H$ which contain the vertex $\mu$
\end{corollary}
\begin{proof} We have $\Pi=\Pi^*\cup \Pi_*$ and $\Pi^*$, $\Pi_*$ are separated with $J_>\subseteq \Pi^*$, $\Pi_*\subseteq J_0$. Hence
\begin{eqnarray*}
   W_{J_0} J_>= W_{(\Pi^*\cap J_0)\cup \Pi_*} J_>= W_{\Pi^*\cap J_0}W_{\Pi_*} J_>=  W_{\Pi^*\cap J_0} J_> .
\end{eqnarray*}
If $(\Pi^*\cap J_0)^0=\Pi^*\cap J_0$ then  $W_{\Pi^*\cap J_0}$ is finite. Thus $W_{\Pi^*\cap J_0} J_>$ is finite.
Now let $(\Pi^*\cap J_0)^0\neq\Pi^*\cap J_0$. We show that $ W_{\Pi^*\cap J_0} J_>$ is infinite.

There exists a component $K$ of $\Pi^* \cap J_0$ of nonfinite type. Let $K'$ be the component of $\Pi^*$ that contains $K$. Since $K'$ is connected and $\mu$-connected there exists $\beta\in K'\cap  J_>$ adjacent to some $\gamma \in K$. Now $\{\a^\vee\mid \a\in K\}\subseteq \h_\R$ and $\{\a\mid \a\in K\}\subseteq \h_\R^*$ is a free root base for the Weyl group $W_K$. Moreover, $-\beta$ is in its fundamental chamber in $\h_\R^*$, since $-\la\beta,\a^\vee\ra\geq 0$ for all $\a\in K$. Hence the isotropy group of $\beta$ in $W_K$ is given by  $W_L $ where $L=\{\a\in K\mid r_\a\beta=\beta \}$. Furthermore, $L\subsetneqq K$ since $\gamma\notin L$. By \cite[Proposition 2.43]{AB08} the set $W_K\beta$ is infinite.

Let $\g(A)$ be of finite, affine, or strongly hyperbolic type. If $J_0=\Pi$ then $\Pi^*=\emptyset$. Thus $\Pi^*\cap J_0=\emptyset$. If $J_0\neq\Pi$ then $\Pi^*\cap J_0$ is a proper subset of $\Pi$. Hence it is either empty or a union of components of finite type.
\hfill$\Box$\end{proof}

We define the map  $\text{red}:W\to\Pi$ as follows: We set $\text{red}(1):=\emptyset$. If $w\in W$ and $w=r_{i_1}r_{i_2}\cdots r_{i_k}$ is a reduced expression, we set $\text{red}(w):=\{\a_{i_1},\a_{i_2},\ldots,\a_{i_k}\}$; it is independent of the chosen reduced expression by \cite[Theorem 2.33]{AB08}.

\begin{theorem} \label{lattice-operations}
{\rm (a)} Let $I, I'$ be $\mu$-connected and $w\in \mbox{}^I W^{I'}$. Then
\begin{eqnarray*}
     F_I \cap w F_{I'} = \left\{\begin{array}{cll}
  F_{(I\cap w I')^*} & \text{if} &  w\in W_{J_0}, \\
   \emptyset & \text{if} &  w\notin W_{J_0} .
\end{array}\right.
\end{eqnarray*}

{\rm (b)} Let $I, I'$ be $\mu$-connected and $w\in \mbox{}^{I_*}W^{I'_*}$. Then $I\cup I'\cup \text{\rm red}(w)$ is $\mu$-connected and
\begin{eqnarray*}
    F_I\vee w F_{I'} = F_{I\cup I'\cup \text{\rm red}(w)}.
\end{eqnarray*}
\end{theorem}
\begin{remark}
{\rm
 The lattice intersection and lattice join of two arbitrary nonempty faces can be reduced to Theorem \ref{lattice-operations}, by using $\mbox{}^{I\cup I_*}W^{I'\cup I'_*}\subseteq \mbox{}^I W^{I'}$, $\mbox{}^{I\cup I_*}W^{I'\cup I'_*}\subseteq \mbox{}^{I_*}W^{I'_*}$, and Theorem \ref{faceIsotropyGroup}.
}
\end{remark}
\begin{proof} We first show (a). From Lemma \ref{generatingProperties} and Theorem \ref{fiface} we find that the face $F_I\cap w F_{I'}$ is generated by
\begin{equation*}
  (F_I\cap w F_{I'})\cap W\mu=   (F_I\cap W\mu)\cap (w F_{I'}\cap W\mu) = W_I\mu\cap w W_{I'}\mu .
\end{equation*}

Suppose that $F_I\cap w F_{I'}\neq\emptyset$. Let $w_1\in W_I$, $w_1'\in W_{I'}$ such that $w_1\mu=w w_1'\mu$. Then $(w_1)^{-1}w w_1'\in W_{J_0}$. From \cite[Proposition 2.23]{AB08}, there exist $w_2\in W_I$, $w_2'\in W_{I'}$ such that
\begin{equation*}
(w_2)^{-1}w w_2'=(w_1)^{-1}w w_1'\in W_{J_0} \quad \text{ and }\quad l((w_2)^{-1}w w_2')= l((w_2)^{-1})+l(w)+ l(w_2').
\end{equation*}
It follows that $w\in W_{J_0}$, $w_2\in W_{J_0}\cap W_I= W_{J_0\cap I}$, and $w_2'\in  W_{J_0}\cap W_{I'} = W_{J_0\cap I'}$. Therefore, we obtain
\begin{equation*}
  W_I \ni  w_2(w_1)^{-1}= w w_2'(w_1')^{-1}w^{-1}\in w W_{I'}w^{-1} .
\end{equation*}
Hence $ w_2(w_1)^{-1}\in W_{I\cap wI '}$ by \cite[Lemma 2.25]{AB08}, and $w_1\mu\in  W_{(I\cap wI ')}w_2\mu =W_{(I\cap wI ')}\mu $.
Thus $F_I\cap w F_{I'}\subseteq F_{I\cap wI'}=F_{(I\cap wI')^*}$.

We have $W_{I\cap wI '}=W_I\cap w W_{I'}w^{-1}$ by \cite[Lemma 2.25]{AB08}. If $w\in W_{J_0}$ then
\begin{eqnarray*}
  W_{I\cap wI '}\mu \subseteq  W_I\mu\cap w W_{I'}w^{-1}\mu  =   W_I\mu\cap w W_{I'}\mu .
\end{eqnarray*}
Hence, $F_{(I\cap wI')^*}=F_{I\cap wI'}\subseteq F_I\cap w F_{I'}$.

For the proof of (b) we first show that $J:= I\cup I'\cup \text{red}(w)$ is a $\mu$-connected set. Since $w\in W_J=W_{J^*}W_{J_r}$ with $J_r\subseteq J_*$ we can write $w=w^*w_*$ with $w^*\in W_{J^*}$ and $w_*\in W_{J_*}$. Thanks to $J_*\subseteq I'_*$, we get
\begin{equation*}
    W_{I_*} w W_{I'_*} = W_{I_*} w^*w_* W_{I'_*} =  W_{I_*} w^* W_{I'_*}.
\end{equation*}
Since $w$ is a minimal double coset representative we find $l(w^{*})\geq l(w)=l(w^*)+l(w_*)$. We conclude that $l(w_*)=0$ and  $w_*=1$. Thus $\text{red}(w)=\text{red}(w^*)\subseteq J^*$. Therefore $J= I\cup I' \cup \text{red}(w)\subseteq J^*$, which shows that $J$ is $\mu$-connected.

We now prove the second part of (b). By Corollary \ref{ff-cs} and Corollary \ref{dfct} there exist $w_1\in W$ and some $\mu$-connected set $K$ such that $F_I\vee wF_{I'}\ =w_1 F_K$.

Clearly, $F_I, wF_{I'}\subseteq F_{I\cup I'\cup\text{red}(w)}$. Hence $w_1F_K=F_I\vee wF_{I'}\subseteq F_{I\cup I'\cup\text{red}(w)}$. From Proposition \ref{incl-ffaces} we get $K\subseteq I\cup I'\cup\text{red}(w)$.

Since $F_I, wF_{I'}\subseteq w_1 F_K$ we find from Proposition \ref{incl-ffaces} that $I, I'\subseteq K$, and $w_1\in W_{I_*}W_K$, $w^{-1}w_1 \in W_{I'_*}W_K$. Eliminating $w_1$, we obtain $w\in W_{I_*}W_K W_{I'_*}$, and equivalently  $W_{I_*}w W_{I'_*}\cap W_K\neq \emptyset$. It follows from \cite[Proposition 2.23]{AB08}, similarly as in the proof of part (a), that $w\in W_K$. We conclude that $I\cup I'\cup\text{red}(w)\subseteq K$.

We have shown that $w_1 F_K\subseteq F_{I\cup I'\cup\text{red}(w)}$, and $K=I\cup I'\cup\text{red}(w)$, from which it follows that the dimension of both are equal. Therefore, $w_1 F_K=  F_{I\cup I'\cup\text{red}(w)}$.
\hfill$\Box$\end{proof}

Let the partially ordered set $(L,\leq)$ be a lattice. In this article a subset $L_1\subseteq L$ is called a sublattice of $L$, if the partially ordered set $(L_1,\leq)$ is a lattice with the same smallest and biggest elements, and the same lattice operations as in $L$.
\begin{corollary}\label{ff-sl} The set $\F$ of fundamental faces is a sublattice of the face lattice $\F(H)$.
\end{corollary}
\begin{proof} Trivially, $\emptyset, H\in\F $. Let $F_1, F_2\in \F$. We obtain from Theorem \ref{lattice-operations} that
\begin{equation}\label{ffl-gl}
F_1\cap F_2\in \F \quad\text{ and }\quad F_1\vee F_2\in \F
\end{equation}
for all $F_1, F_2\in \F\setminus\{\emptyset\}$. Clearly, (\ref{ffl-gl}) holds if $F_1=\emptyset$ or $F_2=\emptyset$.
\hfill$\Box$\end{proof}
\section{The monoid $M(\br)$\label{TheMonoid}}
In this section we fix an irreducible highest weight representation $(V, \rho)$ of the Kac-Moody algebra $\g$ of highest weight $\mu\in P_+$. We assume that $\Pi$ is $\mu$-connected, i.e., no connected component of $\Pi$ is contained in the type  $J_0$ of $\mu$. We denote by $(V, \br)$ the corresponding irreducible highest weight representation of the Kac-Moody group $\bG$. We fix a contravariant nondegenerate symmetric bilinear form $(\, \mid \,)$ on $V$.

\subsection{The definition of $M(\br)$}
The linear space $V$ decomposes into a direct sum of weight spaces
\[
    V = \bigoplus_{\eta\in P(V)}V_\eta.
\]
The weight hull $H$ of $\rho$ is the convex hull of the set of weights $P(V)\subset \h_\R^*$. By \cite[Proposition 11.3 a)]{Kac90} $H$ coincides with the orbit hull of $\mu$.

If $F$ is a face of $H$ we get a decomposition
\begin{equation}\label{ImageKernel}
  V = V_F \oplus V_F^\perp \quad\text{ with }\quad V_F := \bigoplus_{\eta\in F\cap P(V)} V_\eta \quad\text{ and }\quad V_F^\perp := \bigoplus_{\eta\in P(V)\setminus F} V_\eta
\end{equation}
which is also orthogonal with respect to the nondegenerate bilinear form $(\, \mid \,)$. We call $V_F$ the {\it face vector space} associated to the face $F$. We denote by $e(F)$ the corresponding linear projection defined by
\begin{equation}\label{idemDefinition}
    e(F)v = \bigg\{
                \begin{aligned}
                &v,   \iff v\in V_F,        \\
                &0,   \iff v\in V_F^\perp.  \\
                \end{aligned}
\end{equation}
The projection $e(F)$ is an idempotent of $\Endv$. Its adjoint $e(F)^\Cai$ with respect to the nondegenerate bilinear form $(\, \mid \,)$ exists and we have $e(F)^\Cai=e(F)$.

The set
\[
    E := \{e(F) \mid F\in {\F}(H) \}
\]
is a commutative submonoid of $\Endv$ consisting of idempotents. Its multiplication is given by
\begin{equation}\label{multE}
    e(F)e(F') = e(F\cap F')\quad\text{where}\quad F,\, F'\in  {\F}(H).
\end{equation}
Moreover, $e(\emptyset)$ is the zero and $e(H)$ is the identity of $E$ as well as of $\Endv$.

Recall that $G:=\C^\times\br(\bG)$. We define the monoid $M(\br)$ to be the submonoid of $\Endv$ generated by $G$ and $E$, that is,
\[
    M(\br) := \la G, ~E\ra .
\]
We often write $M$ instead of $M(\br)$.
The adjoints of the elements of $M$ with respect to the nondegenerate bilinear form $(\, \mid \,)$ exist, and are contained in $M$. In this way we get an anti-involution $\Cai$ on $M$, extending the Chevalley anti-involution on $G$, which we call the Chevalley anti-involution of $M$. \vspace*{1ex}

We define $\overline T$ to be the submonoid of $M$ generated by $T$ and $E$. We define $\overline N$ to be the submonoid of $M$ generated by $N$ and $E$. To describe these submonoids we need the following Lemma.

\begin{lemma} \label{NNormalizesE}
Let $F$ be a face of $H$, and let $n\in N$ represent $w\in W$. Then
\[
    ne(F)n^{-1} = e(wF).
\]
\end{lemma}
\begin{proof}
This is straightforward from (\ref{idemDefinition}) and (\ref{NTAction}). \hfill$\Box$
\end{proof}

If $Y$ is a submonoid of $M$ we denote by $Y^\times$ its unit group, and by $E(Y)$ its set of idempotents.
\begin{proposition}\label{idempotentOfNbar}
{\rm (a)} $\T$ is a unit regular commutative monoid with $\T^\times=T$ and $E(\T)=E$. In particular, $\T=TE=ET$.

{\rm (b)} $\N$ is a unit regular inverse monoid with $\N^\times=N$ and $E(\N)=E$. In particular, $\N=NE=EN$.
\end{proposition}
\begin{proof} From Lemma \ref{NNormalizesE} it follows that $\T$ is commutative. Hence $\T=TE=ET$. Again from Lemma  \ref{NNormalizesE} we also find that $\N=NE=EN$.

Clearly, $T\subseteq \T^\times$. Conversely, $x\in\T^\times$ is of the form $x=te$ for some $t\in T$ and $e\in E$. Therefore, $e=t^{-1}x\in E(\T)\cap T^\times=\{1\}$. Thus $x=t\in T$.  Similarly, we get $\N^\times =N$.

We have $E\subseteq E(\T)\subseteq E(\N)$. Conversely, any idempotent $x\in E(\N)$ can be written in the form $x=n_we(F)$ for some $n_w\in N$ projecting to $w\in W$, and some face $F$ of $H$. From $x=x^2$ we get $e(F)=e(F)n_we(F)$. Evaluating at $v_\eta\in V_\eta\setminus\{0\}$ where $\eta\in F$ we find  $0\neq v_\eta = e(F)n_w v_\eta $. Since $n_wv_\eta$ is homogeneous it follows that $v_\eta=n_w v_\eta $. Hence  $xv_\eta = n_w v_\eta = v_\eta$. For $v_\eta\in V_\eta$ where $\eta\notin F$ we obtain $xv_\eta =0$. We conclude that $x=e(F)\in E$.

Clearly, $\T$ and $\N$ are unit regular monoids. The idempotents of $\N$ commute. From \cite[Theorem 5.1.1]{Ho95} it follows that $\N$ is an inverse monoid.
\hfill$\Box$\end{proof}

Our aim below is to establish the unit regularity, and the Bruhat and Birkhoff decompositions for $M$. To this end, we need intensive preparations. We
proceed similarly as in \cite[Chapter 2]{Mo05}.
\subsection{Weight strings}
Let $F$ be a face of $H$. The weights in $P(V)\cap F$ are called {\it $F$-weights}. We investigate explicitly the $\a$-weight string through an $F$-weight where $\a$ is a real root.

Recall from Section \ref{s3} the definition and the properties of the parabolic subgroups $W(F)$, $W_*(F)$. We define
\[
\begin{aligned}
    \D(F)  &            := \{ \a\in\Dr \mid r_\a\in W(F)\}        \, \;=\{\a\in\D^{re}\mid r_\a F=F\}  ,\\
    \D_*(F)&          := \{ \a\in\Dr \mid r_\a\in W_*(F)\}    \, =\{\a\in\D^{re}\mid r_\a \eta=\eta \text{ for all }\eta\in F\} ,\\
     \D^*(F) &         :=\D(F)\setminus \D_*(F)   .
\end{aligned}
\]
Note that if $F=wF'$ with $w\in W$ and $F'$ a fundamental face then
\begin{eqnarray*}
\begin{aligned}
  \D(F)&=w\D(F')=w W_{\lambda(F')}\lambda(F'), \\
   \D_*(F)&=w\D_*(F')=w W_{\lambda_*(F')}\lambda_*(F'), \\
   \D^*(F)&=w\D^*(F')=w W_{\lambda^*(F')}\lambda^*(F') .
\end{aligned}
\end{eqnarray*}
Since the isotropy group $W(F')$ of $F'$ leaves $\Dr_\pm\setminus \D(F')$ invariant we can define
\begin{equation*}
  \D_p(F):=w(\Dr_+\setminus \D(F')) \quad\text{ and }\quad   \D_n(F):=w(\Dr_-\setminus \D(F'))   .
\end{equation*}

In particular, we have $\D(\emptyset)=\D_*(\emptyset)=\Dr$, $\D^*(\emptyset)=\D_p(\emptyset)=\D_n(\emptyset)=\emptyset$, and
$\D(H)=\D^*(H)=\Dr$, $\D_*(H)=\D_p(H)=\D_n(H)=\emptyset$. Note that
\begin{equation*}
 \D^{re} = \D(F) \sqcup \D_p(F) \sqcup \D_n(F)  =  \D_*(F) \sqcup \D^*(F)\sqcup \D_p(F) \sqcup \D_n(F) .
\end{equation*}

\begin{lemma}\label{alphaInDLowerStar}
Let $F$ be a face of $H$, and $w\in W$.

{\rm (a)} $w\in W(F)$ if and only $w( P(V)\cap F)=P(V)\cap F$.

{\rm (b)} $w\in W_*(F)$ if and only $w\eta=\eta$ for all $\eta\in P(V)\cap F$.

\end{lemma}
\begin{proof} The Lemma holds trivially for $F=\emptyset$. Since $P(V)$ is $W$-invariant it is sufficient to show the Lemma for a nonempty fundamental face $F$. From
\begin{eqnarray*}
   W_{\lambda^*(F)}\mu\subseteq P(V)\cap F\subseteq F=\text{co}(W_{\lambda^*(F)}\mu)
\end{eqnarray*}
we get $F=\text{co}(P(V)\cap F)$, from which the Lemma follows immediately.
\hfill$\Box$\end{proof}

\begin{theorem}\label{wsF} Let $F$ be a nonempty face of $H$.

{\rm (a)} If $\a \in \D(F)$ then for every $F$-weight $\eta$ the $\a$-weight string through $\eta$ lies completely in $F$.

{\rm (b)} If $\a \in \D_*(F)$ then for every $F$-weight $\eta$ the $\a$-weight string through $\eta$ has only one element. In particular, $\la\eta,\a^\vee\ra=0$.

{\rm (c)} If $\a \in \D^*(F)$ then there exists an $F$-weight $\eta$ such that the $\a$-weight string through $\eta$ has more than one element. In particular, there exist F-weights $\eta_+,\,\eta_- $ such that $\la\eta_+,\av\ra>0$ and $\la\eta_-,\av\ra<0$.
\end{theorem}
\begin{proof}
To prove (a) let $\eta$ be an $F$-weight. If the $\a$-weight string through $\eta$ has only one element, it is contained in $F$. Suppose that the string has more than one element. If $\eta$ is one of the two ends of the weight string, then $r_\a\eta$ is the other end. Moreover, $r_\a\eta$ is an $F$-weight because of $\a \in \D(F)$. Since $F$ is convex, the full string is in $F$. If $\eta$ is not an end of the $\a$-weight string through $\eta$, then the string is of the form
\[
    \eta - p\a, ~\ldots~, ~\eta, ~\ldots~, ~\eta + q\a, \quad  \with p, ~q \ge 1.
\]
Since $F$ is a face of the weight hull $H$ the whole string is contained in $F$.

To prove (b) let $\eta$ be an $F$-weight. The $\a$-weight string through $\eta$ is contained in $F$ by (a), and the reflection $r_\a$ interchanges its ends. Since $\a\in \D_*(F)$, the reflection $r_\a$ also fixes its ends. Hence the $\a$-weight string through $\eta$ has only one element.

We now prove (c). Since $r_\a\notin W_*(F)$ it follows from Lemma \ref{alphaInDLowerStar} (b) that $r_\a\eta\neq\eta$ for some $F$-weight $\eta$. Hence the $\a$-weight string through $\eta$ has more than one element, and is contained in $F$ by (a). Choose $\eta_+$ to be the end of the string with $\la\eta_+,\av\ra>0$, and $\eta_-$ to be the end of the string with $\la\eta_-,\av\ra<0$.
$\hfill\Box$\end{proof}

\begin{corollary}\label{wsNotF}
Let $F$ be a nonempty face of $H$ and $\a \in \D(F)$. If $\eta\in P(V)$ is not an $F$-weight then no weight in the $\a$-weight string through $\eta$ is an $F$-weight.
\end{corollary}

\begin{theorem}\label{WeightString1} Let $F$ be a nonempty face of $H$.

{\rm (a)} If $\a\in \D_p(F)$ and $\eta$ is an $F$-weight then $\la \eta, \av \ra \ge 0$. The $\a$-weight string through $\eta$ is
\[
    \eta, ~\eta-\a,  ~\ldots~, ~r_\a(\eta)=\eta-  \la \eta, \av \ra \a,
\]
and $\eta$ is the only $F$-weight in the string. In particular, $\eta+\a$ is not a weight. There exists an $F$-weight $\eta_+$ such that $\la\eta_+, \av\ra > 0$.

{\rm (b)} If $\a \in \D_n(F)$ and $\eta$ is an $F$-weight then $\la \eta, \av \ra \le 0$. The $\a$-weight string through $\eta$ is
  \[
    \eta, ~\eta+\a,  ~\ldots~, ~r_\a(\eta)=\eta-  \la \eta, \av \ra \a,
  \]
and $\eta$ is the only $F$-weight in the string. In particular, $\eta-\a$ is not a weight. There exists an $F$-weight $\eta_-$ such that $\la\eta_-, \av\ra < 0$.
\end{theorem}
\begin{proof} We only prove (a); the proof of (b) is similar. We have $F=wF'$ for some $w\in W$ and a nonempty fundamental face $F'$, and $\eta$, $\a$ can be written as $\eta=w\eta'$ with $\eta'\in F'$, and $\a=w\a'$ with $\a'\in \Drp \setminus \D(F')$. Then $r_{\a}\eta=w (r_{\a'}\eta')$, and the $\a$-weight string through $\eta$ and the $\a'$-weight string through $\eta'$ are related by
\begin{eqnarray*}
 P(V)\cap (\eta+\Z\a) =  w( P(V)\cap (\eta'+\Z\a') ).
\end{eqnarray*}
Moreover, $\la \eta, \a^\vee\ra=\la \eta', w^{-1}\a^\vee\ra=\la \eta, (w^{-1}\a)^\vee\ra=\la \eta',\a'^\vee\ra$. Therefore, it sufficient to prove (a) for a fundamental face $F$.

We first show that $\la\eta,\a^\vee\ra\geq 0$ for all $\eta\in F$. Since $F=\text{co}(W_{\lambda(F)} \mu)$, it suffices to show this for all $\eta\in W_{\lambda(F)}\mu$. Let $w_1\in W_{\lambda(F)}$. Then $w_1^{-1}\a \in w_1^{-1} (\D^{re}_+\setminus \D(F))\subseteq \D^{re}_+$ and $(w_1^{-1}\a)^\vee>0$. Hence
\begin{eqnarray*}
  \la w_1\mu,\a^\vee\ra =  \la \mu, w_1^{-1}\a^\vee\ra =   \la \mu, (w_1^{-1}\a)^\vee\ra\geq 0.
\end{eqnarray*}

We now describe the $\a$-weight string through the $F$-weight $\eta$ when $\la \eta, \av \ra > 0$. Suppose that $\eta$ is not an end of the string. Since $F$ is a face of the weight hull $H$, the whole string is contained in $F$. In particular $r_\a\eta$ is an $F$-weight. But this is not possible since $\la r_\a\eta, \a^\vee\ra=- \la\eta,\a^\vee\ra<0$. So the weight string is of the form
\[
    \eta, ~\eta-\a,  ~\ldots~, ~r_\a(\eta) = \eta - \la \eta, \a^\vee\ra \a.
\]
In particular, $\eta+\a$ is not a weight.

We next describe the $\a$-weight string through the $F$-weight $\eta$ when $\la \eta, \av \ra = 0$. Clearly, $r_\a(\eta) = \eta\in F\cap r_\a F$. Since $F\subseteq \{ \eta'\in\h_\R^*\mid \la\eta',\a^\vee\ra\geq 0\}$ the face $F_1:=F\cap r_\a F$ is fixed pointwise by $r_\a$. Hence $\alpha\in\Delta_*(F_1)$. By Theorem \ref{wsF} (b) the $\alpha$-weight string through $\eta$ consists of $\eta$ only.

Suppose that no $F$-weight $\eta_+$ satisfies $\la \eta_+, \av \ra > 0$. Then  $\la \eta, \av \ra = 0$ for all $F$-weights $\eta$. Lemma \ref{alphaInDLowerStar} (b) implies that $\a\in \D_*(F)$, which contradicts $\a\in \D_p(F)=\Delta^{re}_+\setminus\Delta(F)$.
\hfill$\Box$\end{proof}

\subsection{Isotropy monoids, isotropy groups, and stabilizers}
Let $X$ be temporally a linear subspace of $V$, and let $Y$ be a subgroup of $G$. The {\it left isotropy monoid} of $X$ in $Y$, which is a submonoid of $Y$, is defined by
\[
        N_Y^\subseteq(X) := \{y\in Y\mid y X \subseteq X\}.
\]
The {\it isotropy group} of $X$ in $Y$, which is a subgroup of $Y$, is defined by
\[
        N_Y(X) := \{y\in Y\mid yX = X\}.
\]
The {\it stabilizer} of $X$ in $Y$, which is a normal subgroup of $N_Y(X)$, is defined by
\[
        Z_Y(X) := \{y\in Y\mid yv = v \,\text{ for all }\, v\in X\}.
\]

We have $N_Y(X)= N_Y^\subseteq(X)\cap  N_Y^\subseteq(X)^{-1}$. If $ N_Y^\subseteq(X)$ is a group then $N_Y(X)= N_Y^\subseteq(X)$.
We denote by $X^\perp$ the biggest subspace of $V$ orthogonal to $X$ with respect to $(\,\mid\,)$. It is easy to see that if $V=X\oplus X^\perp$ then
\begin{equation}\label{orthisostab}
   N_Y^\subseteq(X^\perp) =N_{Y^\Cai}^\subseteq(X)^\Cai \quad \text{ and }\quad  N_Y (X^\perp) =N_{Y^\Cai}(X)^\Cai .
\end{equation}
For $g\in G$ we have $ N_Y^\subseteq(gX) = g N^\subseteq_{g^{-1}Yg}(X) g^{-1} $, and $N_Y (gX) = g N_{g^{-1}Yg}(X) g^{-1}$, and $Z_Y(gX) =g Z_{g^{-1}Yg}(X) g^{-1}$. Trivially, we have
\begin{eqnarray}\label{isostabtrivialcases}
     Z_Y(\{0\})=N_Y(\{0\})= N_Y^\subseteq(\{0\})=Y \quad \text{and}\quad   Z_Y(V)=\{1\},\; N_Y (V)= N_Y^\subseteq(V)=Y .
\end{eqnarray}

One of the reasons to introduce these concepts is the following proposition, which describes how to perform certain calculations with the idempotents in $E$. These are basic for the investigation of $M$.
\begin{proposition}\label{YYOmega}
    Let $F, F'$ be faces of $H$. Let $Y$ be a subgroup of $G$, and $y\in Y$. Then
\begin{eqnarray*}
      {\rm (a)} \quad y e(F)  \:= &e(F)ye(F)  &\Leftrightarrow \quad y\in N_Y^\subseteq(V_F).\\
      {\rm (b)} \quad e(F)y  \:= &e(F)ye(F)  &\Leftrightarrow \quad y\in N_{Y^\Cai}^\subseteq(V_F)^\Cai.
\end{eqnarray*}
In addition,
\begin{eqnarray*}
      {\rm (c)} \quad e(F) \,= &ye(F)y^{-1} &\Leftrightarrow \quad y\in  N_Y^\subseteq(V_F)\cap N^\subseteq_{Y^\Cai}(V_F)^\Cai = N_Y(V_F)\cap N_{Y^\Cai}(V_F)^\Cai .
\end{eqnarray*}
Furthermore, we have
    \begin{eqnarray*}
         {\rm (d)} \quad ye(F)\:= &e(F')  \quad~ &\Leftrightarrow \quad F' = F \quad\text{and}\quad y\in Z_Y(V_F).\\
         {\rm (e)} \quad e(F)y \:= &e(F') \quad~ &\Leftrightarrow \quad F' = F \quad\text{and}\quad y\in Z_{Y^\Cai}(V_F) ^\Cai.
\end{eqnarray*}

\end{proposition}
\begin{proof} It is straightforward to see that (b) can be obtained from (a) by applying the Chevalley anti-involution. Indeed, $e(F)y=e(F)ye(F)$ if and only if $y^\Cai e(F)=e(F) y^\Cai e(F)$, if and only if $y^\Cai\in N_{Y^\Cai}(V_F)$ by (a) for the group $Y^\Cai$. Similarly, (e) is obtained from (d).

We now prove (a). We have $ye(F) =e(F) ye(F)$ if and only if $ye(F)v =e(F) ye(F)v$ for all $v\in V$, if and only if $yv=e(F)yv$ for all $v\in V_F$, if and only if $yv\in V_F$ for all $v\in V_F$, in other words, $y\in N^\subseteq_Y(V_F)$.

We next prove (c). By (a) and (b) we obtain $y\in  N_Y^\subseteq(V_F)\cap N^\subseteq_{Y^\Cai}(V_F)^\Cai$ if and only if $ye(F) =e(F) ye(F)$ and $e(F)y =e(F) ye(F)$, if and only if $ye(F) =e(F) y$.

Two linear projections coincide if their images and their kernels coincide. Therefore, $ye(F)y^{-1}=e(F)$ if and only if $y V_F = V_F$ and $y V_F^\perp = V_F^\perp$, if and only if $y\in N_Y(V_F)$ and $y\in N_Y(V_F^\perp) = N_{Y^\Cai}(V_F)^\Cai$ by (\ref{orthisostab}).

To prove (d) let $ye(F)=e(F')$. Comparing the kernels we get $V_F^\perp = V_{F'}^\perp$, from which it follows that $F = F'$. We have  $ye(F) = e(F)$ if and only if $ye(F)v=e(F)v$ for all $v\in V$, if and only if $yv=v$ for all $v\in V_F$, that is, $y\in Z_Y(V_F)$.
\hfill$\Box$\end{proof}

Our next aim is to determine $N_G^\subseteq (V_F)$, $N_G(V_F)$, and $Z_G(V_F)$, which is not straightforward. As intermediate steps we determine $N_Y^\subseteq (V_F)$, $N_Y(V_F)$, and $Z_Y(V_F)$ for $Y=U_\a$, $\a\in\Dr$, for $Y=T$, and for $Y=N$.

For $v=\sum_\eta v_\eta \in V$ with $v_\eta\in V_\eta$, $\eta\in P(V)$, we define
\begin{equation*}
    \text{supp}(v):=\{\eta\in P(V) \mid v_\eta\neq 0\} .
\end{equation*}
In addition to the results on the weight strings of Theorems \ref{wsF} and \ref{WeightString1} we need the following Lemma to determine $N_{U_\a}^\subseteq (V_F)$, $N_{U_\a}(V_F)$, and $Z_{U_\a}(V_F)$, $\a\in\Dr$.
\begin{lemma}\label{suppandstring} Let $\eta\in P(V)$ and  $\a\in\Dr$.
\begin{itemize}
\item[\rm (a)] For all $v_\eta\in V_\eta$ and $u\in U_\a$ we have $\text{\rm supp}(uv_\eta)\subseteq P(V)\cap (\eta+\Z_+\a)$.

\item[\rm (b)] There exists a weight vector $v_\eta\in V_\eta$ such that for all $u\in U_\a\setminus\{1\}$ we have $\text{\rm supp}(uv_\eta)=P(V)\cap (\eta+\Z_+\a)$.
\end{itemize}
\end{lemma}
\begin{proof} We first prove (a). If $v_\eta\in V_\eta$ and $u = \br(\exp x_\a)$ for some $x_\a\in\g_\a$ then
\begin{equation*}\label{uv}
    uv_\eta =  \br(\exp x_\a) v_\eta = \sum_{j\in\Z_+,\,\eta+j\a\in P(V)} v_{\eta + j\a} \quad \text{ with }\quad v_{\eta + j\a} = \frac{1}{j!} \br(x_\a)^j v_\eta \in V_{\eta+ j\a}.
\end{equation*}
Hence, $\text{supp}(uv_\eta)\subseteq P(V)\cap (\eta+\Z_+\a)$.

We now prove (b). Choose $x_\a\in\g_\a$, $x_{-\a}\in\g_{-\a}$ such that $[x_\a,x_{-\a}]=\a^\vee$. Then ${\bf s}_\a:=\C x_\a+\C \a^\vee +\C x_{-\a}$ is a Lie subalgebra of $\g$ isomorphic to $sl(2,\C)$. Since $V$ is integrable, it decomposes into a direct sum of irreducible finite-dimensional ${\bf s}_\a$-modules, whose $\C \a^\vee$-weight spaces are also $\h$-weight spaces. The $\a$-weight string through $\eta$ is finite. Hence, among these modules exists a module $D$, whose set $P(D)$ of $\h$-weights coincides with $P(V)\cap (\eta+\Z\a)$, the $\a$-weight string through $\eta$. Choose $v_\eta\in D_\eta\setminus\{0\}$. Then $\text{supp}(\br(\exp(c x_\a)) v_\eta)=P(V)\cap (\eta+\Z_+\a)$ for all $c\in\C^\times$.
\hfill$\Box$\end{proof}

\begin{theorem}\label{isostabUa} Let $F$ be a face. Then
\[
    \begin{aligned}
         {\rm (a)} \quad &N_{U_\a}^\subseteq(V_{F}) = N_{U_\a}(V_{F}) = \bigg\{
                                \begin{aligned}
                                    &U_\a,   \quad\text{if } \a\in \D_p(F)\cup\D(F), \\
                                    &\{1\},      ~\text{ otherwise.}
                                \end{aligned}\\
         {\rm (b)} \quad &Z_{U_\a}(V_{F})    = \bigg\{
                                \begin{aligned}
                                    &U_\a,   \quad\text{if }\a\in \D_p(F) \cup \D_*(F) , \\
                                    &\{1\},      ~\text{ otherwise.}
                                \end{aligned}\\
    \end{aligned}
\]
\end{theorem}

\begin{proof} The theorem holds for $F=\emptyset$ by (\ref{isostabtrivialcases}). Let $F\neq\emptyset$. We make repeated use of Lemma \ref{suppandstring} in the proof without mentioning it further. Let $u\in U_\a$ and $v_\eta\in V_\eta$ with $\eta\in F$. If $\a\in \D_p(F) \cup  \D_*(F) $ then
\begin{eqnarray*}
  \text{supp}(uv_\eta)\subseteq P(V)\cap(\eta+\Z_+\a)=\{\eta\}
\end{eqnarray*}
by Theorem \ref{WeightString1} (a) and Theorem \ref{wsF} (b). Therefore, $uv_\eta=v_\eta$. Hence, $U_\a\subseteq Z_{U_\a}(V_F)\subseteq N_{U_\a}(V_F)\subseteq N^\subseteq_{U_\a}(V_F)\subseteq U_\a$.
If $\a\in\D^*(F) $ then
\begin{eqnarray*}
  \text{supp}(uv_\eta)\subseteq P(V)\cap(\eta+\Z_+\a)\subseteq F
\end{eqnarray*}
by Theorem \ref{wsF} (a). Therefore, $uv_\eta\in V_F$. Hence, $N^\subseteq_{U_\a}(V_F)= U_\a$. Since $N^\subseteq_{U_\a}(V_F)$ is a group, we have $N_{U_\a}(V_F)= N^\subseteq_{U_\a}(V_F)$.

Let $\a\in \D_n(F)$ and $u\in U_\a\setminus\{1\}$. From Theorem \ref{WeightString1} (b) there exists an $F$-weight $\eta$ such that $P(V)\cap(\eta+\Z_+\a)\not\subseteq F$. Moreover, there exists $v_\eta\in V_\eta$ such that
\begin{eqnarray*}
  \text{supp}(uv_\eta)= P(V)\cap(\eta+\Z_+\a)\not\subseteq F.
\end{eqnarray*}
Thus $uv_\eta\notin V_F$. It follows that $\{1\}\subseteq Z_{U_\a}(V_F)\subseteq N_{U_\a}(V_F)\subseteq N^\subseteq_{U_\a}(V_F)\subseteq \{1\}$.
Let $\a\in \D^*(F)$ and $u\in U_\a\setminus\{1\}$. It follows from Theorem \ref{wsF} (c) that there exists an $F$-weight $\eta$ such that $P(V)\cap(\eta+\Z_+\a)\neq \{\eta\}$. Moreover, there exists $v_\eta\in V_\eta$ such that
\begin{eqnarray*}
  \text{supp}(uv_\eta)= P(V)\cap(\eta+\Z_+\a)\neq\{\eta\}.
\end{eqnarray*}
Hence $uv_\eta\neq v_\eta$. We have shown that $ Z_{U_\a}(V_F)=\{1\}$.
\hfill$\Box$
\end{proof}

\begin{theorem}\label{isostabT} Let $F$ be a fundamental face. Then $N^\subseteq_T(V_{F}) = N_T(V_{F})=T$. Moreover, if $F$ is nonempty then
\begin{equation*}
  Z_T(V_F)=\bigg\{ \prod_{j=1}^{2m-l}c_j^{- \la\mu, \a_j^\vee\ra}\br(t_j(c_j))   ~\bigg | ~ c_1,\ldots, c_{2m-l}\in\C^\times,\,\prod_{j=1}^{2m-l}c_j^{\la\ai, \,\a_j^\vee\ra}=1 \text{ for } i\in \lambda^*(F) \bigg\} ,
\end{equation*}
and $Z_T(V_\emptyset)=T$. In particular, $T_{\lambda_*(F)}\subseteq Z_T(V_F)$.
\end{theorem}
\begin{proof} The theorem holds for $F=\emptyset$ by (\ref{isostabtrivialcases}). Let $F\neq\emptyset$. We have $N^\subseteq_T(V_{F}) =T$ since $V_F$ is a direct sum of  $T$-invariant weight spaces. Since  $N^\subseteq_T(V_{F})$ is a group we get $N_T(V_{F}) = N^\subseteq_T(V_{F}) $.

Let $t=c\prod_{j=1}^{2m-l}\br(t_j(c_j))\in T$. If $\eta\in P(V)$ and $v_\eta\in V_\eta$ then $ t v_\eta = c\prod_{j=1}^{2m-l} c_j^{\la\eta,\a_j^\vee\ra} v_\eta$. Therefore, $t\in Z_T(V_F)$ if and only if
\begin{equation}\label{eq-ZT-face}
   c\prod_{j=1}^{2m-l} c_j^{\la\eta,\a_j^\vee\ra} = 1 \quad \text{ for all }\quad \eta\in P(V)\cap F.
\end{equation}
Since $\mu\in P(V)\cap F$ it follows from (\ref{eq-ZT-face}) that
\begin{equation}\label{eq-ZT-hw}
   c\prod_{j=1}^{2m-l} c_j^{\la\mu,\a_j^\vee\ra} = 1.
\end{equation}
Let $\a\in\lambda^*(F)$. By Theorem \ref{wsF} (c) and (a) there exists an $F$-weight $\eta$ such that $\la\eta,\a^\vee\ra < 0$, and the whole $\a$-string through $\eta$ is contained in $P(V)\cap F$. In particular, $\eta+\a\in P(V)\cap F$. Inserting into (\ref{eq-ZT-face}), we find
\begin{equation*}
   c\prod_{j=1}^{2m-l} c_j^{\la\eta,\a_j^\vee\ra} = 1 \quad \text{ and }\quad   c\prod_{j=1}^{2m-l} c_j^{\la\eta+\a,\a_j^\vee\ra} = 1  .
\end{equation*}
We conclude that
\begin{equation}\label{eq-ZT-roots}
     \prod_{j=1}^{2m-l} c_j^{\la\a,\a_j^\vee\ra} = 1  \quad\text{for all}\quad \a\in \lambda^*(F).
\end{equation}
Conversely, (\ref{eq-ZT-hw}) and (\ref{eq-ZT-roots}) imply (\ref{eq-ZT-face}) because we have $P(V)\cap F\subseteq \mu-\R^+\lambda^*(F)$ by Theorem \ref{fiface}. We have shown that $t=c\prod_{j=1}^{2m-l}\br(t_j(c_j))\in Z_T(V_F)$ if and only if (\ref{eq-ZT-hw}) and (\ref{eq-ZT-roots}) hold.
From
\begin{align*}
    \lambda_*(F) =&\;\{ j\in J_0\mid \la \a_i,\a_j^\vee \ra= 0 \text{ for all }i\in \lambda^*(F)\}\\
               =&\;\{ j\in \n \mid \la \mu,\a_j^\vee \ra= 0 \text{ and } \la \a_i,\a_j^\vee \ra= 0 \text{ for all }i\in \lambda^*(F)\}
\end{align*}
we get $T_{\lambda_*(F)}\subseteq Z_T(V_F)$.
\hfill $\Box$\end{proof}

\begin{theorem}\label{isostabN} If $F$ is a fundamental face then
\begin{equation*}
  N^\subseteq_N(V_F) =N_N(V_F) = T N_{\lambda(F)}  \quad\text{ and }\quad  Z_N(V_F) = Z_T(V_F) N_{\lambda_*(F)}.
\end{equation*}
\end{theorem}
\begin{proof} For $i\in\n$ we have
\begin{equation*}
    \br(n_i)= \br(\exp(e_i)\exp(-f_i)\exp(e_i)) \in U_{\a_i}U_{-\a_i}U_{\a_i}.
\end{equation*}
From Theorem \ref{isostabUa} it follows that $\br(n_i)\in N_N(V_F)$ for all $i\in\lambda(F)$, and $\br(n_i)\in Z_N(V_F)$ for all $i\in \lambda_*(F)$. Hence
\begin{eqnarray*}
    T N_{\lambda(F)} = N_T(V_F) N_{\lambda(F)}\subseteq N_N(V_F) \subseteq N_N^\subseteq (V_F)\quad \text{ and }\quad Z_T(V_F) N_{\lambda_*(F)}\subseteq  Z_N(V_F).
\end{eqnarray*}

Let $n_w\in  N^\subseteq_N(V_{F})$ project to $w\in W$. From $n_w V_F\subseteq V_F$ and (\ref{NTAction2}) it follows that $w(P(V) \cap F)\subseteq F$. In the proof of Theorem \ref{alphaInDLowerStar} we have seen that $F=\text{co}(P(V)\cap F)$. Therefore $wF=\text{co}(w(P(V)\cap F))\subseteq F$. Since $wF$ and $F$ are faces of $H$ of the same dimension we obtain $wF=F$. Hence, $w\in W(F)$ and $n_w\in TN_{\lambda(F)}$.

Let $n_w\in  Z_N(V_{F})$ project to $w\in W$. From $ n_w v_\eta = v_\eta$ for all $\eta\in P(V)\cap F$ and (\ref{NTAction2}) we get $w\eta=\eta$ for all $\eta\in P(V)\cap F$.  Hence, $w\in W_{\lambda_*(F)}$ by Lemma \ref{alphaInDLowerStar} (b), and we can write $n_w=tn$ with $t\in T$ and $n\in N_{\lambda_*(F)}\subseteq Z_N(V_F)$. We conclude that $t=n_w n^{-1}\in Z_N(V_F)\cap T=Z_T(V_F)$. Thus, $n_w\in Z_T(V_F)N_{\lambda_*(F)}$.
\hfill $\Box$\end{proof}

\begin{theorem}\label{isostabG} If $F$ is a fundamental face then
\begin{equation*}
     N^\subseteq_G(V_{F}) =N_G(V_F)   = P_{\lambda(F)}  \quad\text{ and }\quad  Z_G(V_{F})=G_{\lambda_*(F)}Z_T(V_{F})\ltimes U^{\lambda(F)}.
\end{equation*}
\end{theorem}
\begin{proof} From Theorem \ref{isostabUa} (a) it follows that $U\subseteq N_G(V_F)$. Let $g=unu_1$ where $n\in N$ and $u, u_1\in U$. From Theorem \ref{isostabN} we observe
\begin{equation*}
    gV_F \subseteq V_F  \;\Leftrightarrow\;  nu_1 V_F \subseteq u^{-1}V_F
                \;\Leftrightarrow\;  nV_F \subseteq V_F
               \;\Leftrightarrow \; n \in TN_{\lambda(F)}.
\end{equation*}
Thus $N^\subseteq_G(V_F) = U TN_{\lambda(F)} U = P_{\lambda(F)}$. Since  $N^\subseteq_G(V_{F})$ is a group, we have $N_G(V_{F}) = N^\subseteq_G(V_{F}) $.

The group $G_{\lambda_*(F)}$ is generated by the root groups $U_{\a}$, $\a\in \D_*(F)$, which stabilize $V_F$ pointwise by Theorem \ref{isostabUa} (b). Hence  $G_{\lambda_*(F)} \subseteq Z_G(V_{F})$. From Theorem \ref{isostabUa} (b) we obtain that $U_\a\subseteq Z_G(V_F)$ for all $\a\in \Drp \setminus \D(F) $. Moreover, $Z_G(V_F)$ is a normal subgroup of $N_G(V_F)$, and the group $N_G(V_F)$ contains $U$. Since $U^{\lambda(F)}$ is the normal subgroup of $U$ generated by $U_\a$, $\a\in\Drp\setminus\D(F)$, we have $U^{\lambda(F)}\subseteq Z_G(V_F)$. Clearly, $Z_T(V_F)\subseteq Z_G(V_{F})$. So, $G_{\lambda_*(F)}Z_T(V_F) U^{\lambda(F)}\subseteq Z_G(V_{F})$.

Now let $g\in Z_G(V_F)\subseteq N_G(V_F)=P_{\lambda(F)}$. Since $P_{\lambda(F)}= TG_{\lambda^*(F)} G_{\lambda_*(F)} U^{\lambda(F)}$ we can write $g$ as a product $g=xh$ for some $x\in TG_{\lambda^*(F)}$ and $h\in G_{\lambda_*(F)} U^{\lambda(F)}\subseteq Z_G(V_F)$.
From the Bruhat decomposition $TG_{\lambda^*(F)} = U_{\lambda^*(F)} T N_{\lambda^*(F)} U_{\lambda^*(F)}$ we can write $x=u_1nu$ with $u_1,u\in U_{\lambda^*(F)}$ and $n\in TN_{\lambda^*(F)}$. Since $x=g h^{-1}\in Z_G(V_F)$ and $u^{-1}V_F=V_F$ we find
\begin{eqnarray*}
   u_1n v_\eta = u_1nu u^{-1}v_\eta= x u^{-1}v_\eta = u^{-1}v_\eta \quad\text{ for all }\quad v_\eta\in V_\eta,\;\eta\in P(V)\cap F.
\end{eqnarray*}
Comparing the components of smallest weight, we obtain $nv_\eta=v_\eta$ for all $v_\eta\in V_\eta$, $\eta\in P(V)\cap F$. Hence, by Theorem \ref{isostabN},
\begin{equation*}
    n\in Z_N(V_F)\cap  TN_{\lambda^*(F)} = Z_T(V_{F}) N_{\lambda_*(F)}\cap  T N_{\lambda^*(F)}\subseteq Z_T(V_F)T_{\lambda_*(F)} = Z_T(V_F).
\end{equation*}
It follows that $x\in Z_T(V_F)U_{\lambda^*(F)}$. With $TG_{\lambda^*(F)} = U_{\lambda^*(F)}^- T N_{\lambda^*(F)} U_{\lambda^*(F)}^-$ we get similarly $x\in Z_T(V_F)U_{\lambda^*(F)}^-$.
We conclude that $x\in  Z_T(V_F)U_{\lambda^*(F)}\cap  Z_T(V_F)U_{\lambda^*(F)}^-= Z_T(V_F)$ and $g=xh\in Z_T(V_F) G_{\lambda_*(F)} U^{\lambda(F)}$.
\hfill$\Box$\end{proof}

As in the theory of reductive linear algebraic monoids the left and right centralizers of $e\in E(M)$ in $G$ are defined by
\begin{eqnarray*}
C^{\,l}_G(e):=\{g\in G\mid ge=ege\} \quad\text{ and  }\quad C^{\,r}_G(e):=\{g\in G\mid eg=ege\}.
\end{eqnarray*}
The centralizer of $e\in E(M)$ in $G$ is defined by
\begin{eqnarray*}
    C_G(e):=\{g\in g\mid ge=eg\}=C^{\,l}_G(e)\cap C^{\,r}_G(e).
\end{eqnarray*}
From Proposition \ref{YYOmega} (a), (b), (c) and Theorem \ref{isostabG}, we obtain the following corollary, which is one of the main results of this section.
\begin{corollary}\label{cG1} Let $F$ be a fundamental face. Then
\begin{itemize}
\item[\rm (a)] $C_G^{\,l}(e(F))=P_{\lambda(F)}$ and $\,C_G^{\,r}(e(F))=P_{\lambda(F)}^-$.
\item[\rm (b)] $C_G(e(F))=L_{\lambda(F)}$.
\end{itemize}
\end{corollary}

As in the theory of reductive linear algebraic monoids the left and right stabilizers of $e\in E(M)$ in $G$ are defined by
\begin{eqnarray*}
   S^{\,l}_G(e) := \{g\in G\mid ge=e\} \quad \text{ and } \quad S^{\,r}_G(e) :=\{g\in G\mid eg=e\}.
\end{eqnarray*}
The stabilizer of $e\in E(M)$ in $G$ is defined by
\begin{eqnarray*}
 S_G(e):=\{g\in G\mid ge=eg=e\}= S^{\,l}_G(e)\cap S^{\,r}_G(e).
\end{eqnarray*}
Combining Proposition \ref{YYOmega} (d), (e) and Theorem \ref{isostabG}, we get the corollary below, which is another main result of this section. Note also that the group $Z_T(V_F)$ has been described explicitly in Theorem \ref{isostabT}.
\begin{corollary}\label{cG2} Let $F$ be a fundamental face. Then
\begin{itemize}
\item[\rm (a)] $S_G^{\,l}(e(F))= G_{\lambda_*(F)}Z_T(V_{F})\ltimes U^{\lambda(F)}$ and $\,S_G^{\,r}(e(F))= U^{\lambda(F)}_-\rtimes G_{\lambda_*(F)}Z_T(V_{F})$.
\item[\rm (b)] $S_G(e(F))= G_{\lambda_*(F)}Z_T(V_{F}) $.
\end{itemize}
\end{corollary}

From Corollaries \ref{cG1} and \ref{cG2}, or alternatively, from Proposition \ref{YYOmega} and Theorem \ref{isostabUa} we get the following properties of the root groups, which we often use.
\begin{corollary}\label{RennerLemma} Let $F$ be a face and set $e=e(F)$. Let $\a\in\D^{re}$ and $u\in U_\a\setminus\{1\}$.

{\rm (a)} If $\a \in \D_*(F)$ then $u e = e u=e$.

{\rm (b)} If $\a \in \D^*(F)$ then $u e = e u\neq e$.

{\rm (c)} If $\a \in \D_p(F) $ then $u e = e\neq eu$.

{\rm (d)} If $\a \in \D_n(F)$ then $eu = e\neq ue$.
\end{corollary}

We employ many times the following consequence of Corollaries \ref{cG1} and \ref{cG2}.
\begin{corollary}\label{UeF}  Let $F$ be a fundamental face.

{\rm (a)} Write $u\in U$ as a product $u=u_1u_2$ with $u_1\in U_{\lambda^*(F)}$ and $u_2\in U_{\lambda_*(F)}\ltimes U^{\lambda(F)}$. Then
\[
 ue(F)=u_1 e(F)=e(F)u_1 .
\]

{\rm (b)} Write $u\in U^-$ as a product $u=u_1 u_2$ with  $u_1\in U^{\lambda(F)}_-\rtimes  U_{\lambda_*(F)}^- $ and $u_2\in U_{\lambda^*(F)}^-$. Then
\[
     e(F)u = e(F)u_2 =  u_2e(F) .
\]
\end{corollary}

Let $F$ be a fundamental face. We denote the projections that belong to the semidirect product decompositions $P_{\lambda(F)} = L_{\lambda(F)} \ltimes U^{\lambda(F)}$ and $P_{\lambda(F)}^- = U^{\lambda(F)}_- \rtimes L_{\lambda(F)}$ by
\[
   \theta_F :  P_{\lambda(F)} \rightarrow L_{\lambda(F)}  \quad\text{ and }\quad  \theta^-_F :  P^-_{\lambda(F)}\rightarrow L_{\lambda(F)}.
\]
Note that the projections $\theta_F$ and $\theta^-_F$ are morphisms of groups whose restrictions to $L_{\lambda(F)} = P_{\lambda(F)}\cap P^-_{\lambda(F)}$ coincide.

\begin{proposition}\label{PeF} Let $F$ be a fundamental face. If $g\in P_{\lambda(F)}$ and $h\in P^-_{\lambda(F)}$, then
\[
    \begin{aligned}
        {\rm (a)} \quad ge(F) &= \theta_F (g)e(F) = e(F)\theta_F(g).\\
        {\rm (b)} \quad e(F)h &= e(F)\theta^-_F(h) = \theta^-_F(h) e(F).\\
    \end{aligned}
\]
\end{proposition}
\begin{proof}
Let $g=\theta_F (g)g_1$ where $\theta_F (g)\in L_{\lambda(F)}$ and $g_1\in U^{\lambda(F)}$. From Corollaries \ref{cG1} and \ref{cG2} we find that
\[
        ge(F)   = \theta_F (g)g_1 e(F)  = \theta_F (g)e(F) = e(F)\theta_F (g),
\]
which shows {\rm  (a)}. Applying the Chevalley anti-involution, we obtain {\rm (b)}.
\hfill $\Box$  \end{proof}

In Theorem \ref{PutchaDecomposition} below we show that the elements of $M$ can be written in the form
\begin{eqnarray*}
         a e(F)b \quad \text { where }\quad a,b\in G,\;F\in {\mathcal F } .
\end{eqnarray*}
The following theorem is another main result of this section. It describes the equality of such expressions. In combination with the Birkhoff decomposition of $G$ and Lemma \ref{NNormalizesE}, the theorem allows to describe the multiplication of two such expressions.
\begin{theorem}\label{geF} Let $g, h\in G$ and $F, F'$ be fundamental faces. The following statements are equivalent.
\[
  \begin{aligned}
    {\rm (a)}\quad &ge(F) = e(F')h. \\
    {\rm (b)}\quad &F=F',~ g\in P_{\lambda(F)}, ~ h\in P^-_{\lambda(F)}, \text{ and } \theta_F (g)=\theta_F^-(h)g_1 \text{ for some }
    g_1\in  G_{\lambda_*(F)}Z_T(V_{F}) .
  \end{aligned}
\]
\end{theorem}
\begin{proof}
We first show that (b) implies (a). Since $F'=F$, from Proposition \ref{PeF} (a), Corollary \ref{cG2}, and Proposition \ref{PeF} (b) we get
\begin{equation*}
               ge(F) = \theta_F (g)e(F) = \theta_F^-(h)g_1e(F)= \theta_F^-(h)e(F)= e(F')h .
\end{equation*}

Next, we prove that {\rm (a)} implies {\rm (b)}. Comparing the images of $g e(F)$ and $e(F')h$ we find $g V_F =  V_{F'}$. Let $g=unv$ where $u, v\in U$, and $n\in N$ projecting to $w\in W$. From Theorem \ref{isostabG} we obtain $v V_F =  V_{F}$ and $u^{-1}V_{F'} =  V_{F'}$. So, $n V_F =  V_{F'}$. Hence, $wF = F'$, where $F$ and $F'$ are fundamental faces. It follows that $F = F'$.

 We observe from $ge(F) = e(F)h$ that $g V_F =  V_{F}$ and $h V_F^\perp =  V_{F}^\perp$. Then $g\in N_G(V_F)=P_{\lambda(F)}$ and $h\in N_G(V_F^\perp)= N_{G^\Cai}(V_F)^\Cai = N_G(V_F)^\Cai = P^-_{\lambda(F)}$ by Theorem \ref{isostabG} and (\ref{orthisostab}). Thanks to Proposition \ref{PeF}, we obtain
  \[
        \theta_F (g)e(F) = ge(F)=e(F)h=\theta_F^-(h)e(F).
  \]
From Corollary \ref{cG2} it follows that
  \[
    \theta_F^- (h)^{-1}\theta_F (g) \in ( G_{\lambda_*(F)}Z_T(V_{F})\ltimes U^{\lambda(F)} )\cap L_{\lambda(F)}=  G_{\lambda_*(F)}Z_T(V_{F}).
  \]
Thus $\theta_F (g) = \theta_F^-(h) g_1$ for some $g_1\in   G_{\lambda_*(F)}Z_T(V_{F})$.
\hfill $\Box$ \end{proof}

From the preceding theorem and Lemma \ref{NNormalizesE} we obtain easily the following result.
\begin{corollary}\label{conjugateFaces}  Let $F$ and $F'$ be faces of $H$. If $Ge(F)G = Ge(F')G$, then there exists $w\in W$ such that $F'=wF$.
\end{corollary}

\subsection{The cross-section lattice}
The {\it cross-section lattice} of $M$ relative to $T$ and $B$ is defined by
\begin{equation}\label{defcrl}
    \Lambda := \{e(F) \mid F \in \mathcal F\}\subseteq E(\T).
\end{equation}
It is easily seen that $\Lambda$ is a finite monoid isomorphic to $({\mathcal F},\cap)$. Its elements are called {\it fundamental idempotents}.

To obtain alternative descriptions of $\Lambda$ we need the following lemma, but we omit its proof since it is straightforward.
\begin{lemma}\label{ZL}
Let $e\in E(\T)$. Then the following conditions are equivalent.

  {\rm (a)} $Be=eBe$.

  {\rm (b)} $Be\subseteq eBe$.

  {\rm (c)} $be=ebe$ for all $b\in B$.
%
%
%
%
\end{lemma}

The cross-section lattice $\Lambda$ can be obtained by the Borel subgroup $B$ as follows.
\begin{theorem}\label{PutchaLattice} We have $\Lambda = \{e\in E(\T)\mid Be\subseteq eB \}= \{e\in E(\T)\mid Be=eBe\}$.
\end{theorem}
\begin{proof} Let $e\in\Lambda$. Since $Te=eT$ it follows from Corollary \ref{UeF} that $Be\subseteq eB$. Clearly, $Be\subseteq eB$ implies $Be\subseteq eBe$, which is equivalent to $Be = eBe$ by Lemma \ref{ZL}.

Now let $e\in E(\T)\setminus\Lambda$. Then $e=e(wF')$ for some fundamental face $F'$ and $w\in W\setminus W(F')$. Since $w^{-1}\notin W(F')=W_{\lambda(F')}$ it follows from \cite[Chapter 5, Proposition 3]{MP95} that there exists a root $\a\in\Drp$ such that $w^{-1}\a\in\Drm\setminus (W_{\lambda(F')}\lambda(F'))= \Drm\setminus \D(F')$. Thus $\a\in \Drp\cap \Delta_n(F)$. Let $u\in U_\a\setminus\{1\}$. From {\rm (d)} of Corollary \ref{RennerLemma} we have $eu=e\neq ue$, hence $eue=e\neq ue$. Therefore, $Be\ne eBe$ by Lemma \ref{ZL}.
\hfill $\Box$\end{proof}

The cross-section lattice $\Lambda$ and the $G\times G$-orbit decomposition of $M$ are related.
\begin{theorem}\label{PutchaDecomposition}
\[
    M = G\Lambda G = \bigsqcup_{e\in \Lambda} G e G  .
\]
\end{theorem}
\begin{proof}
  From Lemma \ref{NNormalizesE} it follows that every element of $M$ can be written in the form
  \[
         g_0 e_1 g_1 e_2 g_2\cdots e_k g_k
  \]
where $g_0,\ldots, g_k\in G$ and $e_1, \ldots, e_k\in\Lambda$. We now show how this element can be further reduced to an element of $G\Lambda G$. It suffices to show that $e(F)ge(F') \in G\Lambda G$ for all $g\in G$ and $F, F'\in\mathcal F$. Since $G=B^-NB$, we see that $g=vnu$ for some $v\in U^-, n\in N$, and $u\in U$. Let $v=v_1v_2$ for some $v_1\in U^{\lambda(F)}_-  U_{\lambda_*(F)}^-$ and $v_2\in  U_{\lambda^*(F)}^-$, and let
 $u=u_1u_2$ for some $u_1\in U_{\lambda^*(F')}$ and $u_2\in U_{\lambda_*(F')} U^{\lambda(F')}$. From Corollary \ref{UeF} it follows that
\[
              e(F)ge(F') =  e(F) vnu e(F')  = v_2e(F) n e(F') u_1 = v_2 e(F) (n e(F') n^{-1}) nu_1 .
\]
Here $e(F) (n e(F') n^{-1})\in E(\T)$ by Lemma \ref{NNormalizesE} and (\ref{multE}). Hence,
\[
   e(F) (n e(F') n^{-1})=n_2e(F'')n_2^{-1}
\]
for some $n_2\in N$, $F''\in\mathcal F$. Thanks to Theorem \ref{geF}, the decomposition is disjoint.
\hfill$\Box$\end{proof}

\subsection{The unit regularity of $M(\br)$}

The following Lemma is used to determine the idempotents of $M$.
\begin{lemma}\label{decFF} Let $g\in G$ and let $F$ be a fundamental face. Then $e(F)ge(F)=e(F)$ if and only if $g$ has a product decomposition of the form
\begin{equation}\label{dec}
    g=u_- x u_+ \quad\text{ with }\quad u_-\in U^{\lambda(F)}_- ,\; x\in G_{\lambda_*(F)}Z_T(V_F) ,\;u_+\in U^{\lambda(F)} .
\end{equation}
Moreover, this decomposition is unique, and
\begin{eqnarray}\label{pdec}
    e(F)u_-=e(F) \quad\text{and}\quad xe(F)= e(F)x=e(F)  \quad\text{and}\quad u_+e(F)=e(F) .
\end{eqnarray}
\end{lemma}
\begin{proof}
Suppose that $g\in G$ has a decomposition $g=u_-xu_+$ of the form (\ref{dec}). By Corollary \ref{cG2} we obtain the equations in (\ref{pdec}), from which we get $e(F)ge(F)=e(F)$.

If $g=u_-'x'u_+'$ is another decomposition of the form (\ref{dec}) then
\begin{eqnarray}\label{cdec}
        G_{\lambda_*(F)}Z_T(V_F) \ltimes U^{\lambda(F)} \ni x u_+(u_+')^{-1} = (u_-)^{-1}u_-'x' \in U^{\lambda(F)}_- \rtimes G_{\lambda_*(F)}Z_T(V_F)
\end{eqnarray}
Since $P_{\lambda(F)}\cap P_{\lambda(F)}^-=L_{\lambda(F)}$ we get $ x u_+(u_+')^{-1} = (u_-)^{-1}u_-'x' \in  G_{\lambda_*(F)}Z_T(V_F) $. We conclude that $u_+(u_+')^{-1} , (u_-)^{-1}u_-' \in  G_{\lambda_*(F)}Z_T(V_F) $. From the semidirect decompositions used in (\ref{cdec}) it follows that $u_+(u_+')^{-1} = 1 $ and $ (u_-)^{-1}u_-'=1$. Hence, also $x=x'$.

Now let $g\in G$ such that $e(F)ge(F)=e(F)$. We may write $g$ in the form $g=u_1u_2 n v_1 v_2$ where $u_1\in U^{\lambda(F)}_- $, $u_2\in U_{\lambda(F)}^-$, $v_1\in U_{\lambda(F)}$, $v_2\in U^{\lambda(F)}$, and $n\in N$ projecting to $w\in W$. From Proposition \ref{PeF} it follows that
\[
  e(F) = e(F)ge(F) = u_2 e(F) n e(F) v_1  = u_2 e(F) (n e(F)n^{-1})n v_1  .
\]
By Lemma \ref{NNormalizesE} and (\ref{multE}) we find that $e(F) (n e(F)n^{-1})=e(F) e(wF)=e(F\cap wF)$. Hence, $e(F)=u_2e(F\cap wF) n v_1$.
From Corollary \ref{conjugateFaces} we find that the dimensions of the faces $F$, $wF$, and $F\cap wF$ are the same. Since $F\cap wF$ is contained in $F$ and $wF$, we obtain $F=F\cap wF=wF$. Thus, $w\in W_{\lambda(F)}$ and $u_2^{-1}e(F)=e(F)nv_1$. From Theorem \ref{geF} it follows that $u_2^{-1}=nv_1g_1$ for some $g_1\in  G_{\lambda_*(F)}Z_T(V_{F})$. Therefore, $ g=u_1u_2n v_1v_2= u_1 g_1^{-1} v_2$.
\hfill $\Box$\end{proof}

\begin{theorem}\label{ui}
{\rm (a)} The unit group of $M$ is $G.$

{\rm (b)} The set of idempotents of $M$ is
\[
    \begin{aligned}
        E(M)    = \{geg^{-1}\mid e\in E(\T), ~ g\in G\}
                = \{geg^{-1}\mid e\in\Lambda,  ~g\in G\}.
    \end{aligned}
\]

{\rm (c)} $M = G E(M) = E(M) G$. Hence, the monoid $M$ is unit regular.
\end{theorem}
\begin{proof} We make use of Theorem \ref{PutchaDecomposition} in the proof without mentioning it further.
Clearly, $G$ is a subgroup of the unit group of $M$. Now let $ge(F)h$ be a unit of $M$, where $g,h\in G$ and $F$ is a fundamental face. Then
\[
    e(F) = g^{-1}(ge(F)h)h^{-1}
\]
is a unit. Thus $e(F)=1$, and hence $ge(F)h=gh\in G$.

It is easily seen that $\{geg^{-1}\mid e\in\Lambda,  ~g\in G\}\subseteq \{geg^{-1}\mid e\in E(\T), ~ g\in G\}\subseteq E(M)$. Now let $ge(F)h$ be an idempotent of $M$, where $g,h\in G$ and $F$ is a fundamental face. Then $e(F)=e(F)hge(F)$. Let $hg=u_-xu_+$ be a decomposition as in Lemma \ref{decFF}. By (\ref{pdec}) we find
\begin{eqnarray*}
   ge(F)h= ge(F)hgg^{-1}= g u_+^{-1}u_+e(F)u_-xu_+ g^{-1}= gu_+^{-1} e(F) ( gu_+^{-1})^{-1} .
\end{eqnarray*}

If  $x\in M$ then $x=geh$ for some $e\in\Lambda$ and $g, h\in G$. So $x=gh(h^{-1}eh)=(geg^{-1})gh$. Hence, $M=GE(M)=E(M)G$.
\hfill $\Box$
\end{proof}

\begin{corollary}\label{conjugateFaces3}
    If two idempotents of $M$ are in the same $G\times G$-orbit, then they are $G$-conjugate.
\end{corollary}
\begin{proof}
Let $e, e_1$ be two idempotents of $M$ in the same $G\times G$-orbit. Then $e=g e'g^{-1}$ and $e_1=g_1 e_1' g_1^{-1}$ for some $g ,g_1\in G$ and $e',e_1'\in\Lambda$. From Theorem \ref{PutchaDecomposition} it follows that $e'=e_1'$. Hence, $e$ and $e_1$ are $G$-conjugate.
\hfill $\Box$\end{proof}

\subsection{The Renner monoid}
We obtain a congruence relation on $\overline N$ by
\[
    x_1 \sim x_2    \quad \text{if and only if} \quad x_1 T = x_2 T,
\]
where $x_1, x_2\in \overline N$. The quotient monoid
\[
        R  : = \N / \sim  \,=  \{xT \mid x\in\N\}
\]
is called the Renner monoid of $M$ relative to $T$.
We denote by $ \phi: \N \rightarrow R $ the quotient morphism: $\phi(n):=nT$, $n\in N$.

For $Y\subseteq R$ we denote by $E(Y)$ the set of idempotents of $R$ contained in $Y$. Clearly, if $Y$ is a submonoid of $R$ then $E(Y)$ is the set of idempotents of $Y$.

\begin{theorem}\label{RennerMonoidStructure}
{\rm (a)} The unit group of $R$ is $W$.

{\rm (b)} The quotient morphism restricts to a bijective map from $E(\N)$ to $E(R)$. In par-\
\hspace*{3.3em}ticu\-lar,  $E(R)$ is a commutative submonoid of $R$.

{\rm (c)} $R= W E(R) = E(R)W$. Hence, the monoid $R$ is unit regular.
\end{theorem}
\begin{proof} We denote the unit group of $R$ by $R^\times$. Applying the quotient morphism to $\N= N E(\N)=E(\N) N$, we obtain
\begin{equation*}\label{RMStreq1}
   R= W \phi(E(\N)) = \phi( E(\N)) W,
\end{equation*}
where $W=\phi(N)\subseteq R^\times$ and $\phi(E(\N))\subseteq E(R)$. Moreover, $\phi(E(\N))$ is a commutative submonoid of $R$, since $E(\N)$ is a commutative submonoid of $\N$.

Let $x\in R^\times$. Then $x=we$ for some $w\in W$ and $e\in \phi(E(\N))\subseteq E(R)$. We conclude that  $e=w^{-1}x\in E(R)\cap R^\times=\{1\}$. Thus $x=w\in W$.

Let $x\in E(R)$. Then $x=\phi(n e(F))$ for some $n\in N$ and $F\in\F(H)$. We set $w=\phi(n)\in W$. Since $x=x^2$, it follows from Lemma \ref{NNormalizesE} and (\ref{multE}) that there exists $t\in T$ such that
\begin{eqnarray}\label{RMStreq2}
    e(F) t =  e(F) n e(F)=  e(F) n e(F )n^{-1} n=  e(F\cap wF) n .
\end{eqnarray}
From Proposition \ref{YYOmega} (e) we obtain $F=F\cap wF$, which is equivalent to $F\subseteq wF$. Since $F$ and $wF$ are faces of the same dimension we get $wF=F$. From Lemma \ref{NNormalizesE} and (\ref{RMStreq2}) it follows that
\begin{eqnarray*}
   n e(F) =  e(wF) n =e(F) n = e(F) t ,
\end{eqnarray*}
which shows that $x=\phi (n e(F))= \phi(e(F)t)=\phi(e(F))\in \phi(E(\N))$.

Let $\phi(e(F))=\phi(e(F_1))$ with $F,F_1\in \F(H)$. Then there exists $t\in T$ such that $e(F)t=e(F_1)$. By Proposition \ref{YYOmega} (e) we get $F=F_1$. Hence $e(F)=e(F_1)$.
\hfill$\Box$\end{proof}

We identify $E(\N)$ and $E(R)$ by the quotient morphism, which is possible by Theorem \ref{RennerMonoidStructure} (b). This is common in the theory of reductive linear algebraic monoids. In particular, we write $e(F)$ instead of $e(F)T$ for all $F\in\F(H)$.
If it is not clear from the context to which space the idempotent $e(F)$, $F\in\F(H)$, belongs, we write explicitly $e(F)\in \N$ or $e(F)\in R$.

We denote the image of the cross-section lattice $\Lambda$ defined in (\ref{defcrl}) under the quotient morphism still by $\Lambda$. Thus, $\Lambda=\{e(F)\mid F\in\F \}$ in $\N$, as well as in $R$.

\begin{corollary}\label{Rinverse} The monoid $R$ is an inverse monoid. Its inverse map \,$\mbox{}^{\rm inv}:R\to R$ is the anti-involution whose restriction to $W$ is the inverse map of $W$, and to $E(R)$ is the identity map.
\end{corollary}
\begin{proof} From \cite[Theorem 5.1.1]{Ho95} and Theorem \ref{RennerMonoidStructure} (b), (c) it follows that $R$ is an inverse monoid. By Theorem \ref{RennerMonoidStructure} (c) every element of $R$ is of the form $\sigma e$ with $\sigma\in W$ and $e\in E(R)$. Now,
$ (\sigma e) (e \sigma^{-1})(\sigma e)=\sigma e$ and $ (e \sigma^{-1})(\sigma e)(e \sigma^{-1})= e\sigma^{-1}$. Hence $(\sigma e)^{\rm inv}=e\sigma^{-1}$.
\hfill$\Box$\end{proof}

\begin{remark}
{\rm
The restriction of the Chevalley anti-involution to $\N = \N^{\,\Cai}$ induces the inverse map on $R$: $\phi(n)^{\rm inv}= \phi(n^\Cai)$ for all $n\in \N$.
}
\end{remark}

We have seen in Theorem \ref{RennerMonoidStructure} (b) that the set of idempotents $E(R)$ is a commutative submonoid of $R$. By \cite[Proposition 1.3.2]{Ho95} we obtain a partial order on $E(R)$ by
\[
    e \le f \quad \Leftrightarrow\quad ef =e,\quad e,f\in E(R),
\]
and $(E(R),\leq)$ is a lower semilattice. Furthermore, the multiplication of $E(R)$ coincides with the semilattice intersection:
\[
    ef= e \wedge f ,\quad e,f\in E(R).
\]
The Weyl group $W$ acts on $E(R)$ by conjugation, preserving the partial order.

\begin{corollary}\label{FL-EL} The map
\begin{eqnarray*}
  \F(H) &\to & E(R)\\
    F\;\;\:  & \mapsto &\: e(F)
\end{eqnarray*}
has the following properties:

{\rm (a)} It is a $W$-equivariant isomorphism of the partially ordered sets $(\F(H),\subseteq )$ and\\
\hspace*{3.1em}$(E(R),\leq )$. In particular, $(E(R),\leq)$ is a lattice.

{\rm (b)} It is a $W$-equivariant isomorphism of the monoids $(\F(H),\cap )$ and $E(R)$.

\end{corollary}
\begin{proof} We first show (b). The $W$-equivariance of the map follows by applying the quotient morphism $\phi$ to the equation in Lemma \ref{NNormalizesE}.
The idempotents $e(F)\in \N$, $F\in\F(H)$, are pairwise different. Hence the map is bijective by Proposition \ref{idempotentOfNbar}. We conclude from (\ref{multE}) that it is also a monoid homomorphism.

For $F, F_1\in \F(H)$ we have $F\subseteq F_1$ if and only if $F=F\cap F_1$. Therefore, we obtain (a) from (b) and the definition of the partial order on $E(R)$.
\hfill$\Box$\end{proof}

The following lemma refines Theorem \ref{RennerMonoidStructure} (c).
\begin{lemma}\label{W-orbits}
\begin{equation*}
 R =\bigsqcup_{e\in E(R)} We = \bigsqcup_{e\in E(R)} eW,
\end{equation*}
and $E(We)=E(eW)=\{e\}$ for all $e\in E(R)$.
\end{lemma}
\begin{proof} We have $R =\bigcup_{e\in E(R)} eW$ by Theorem \ref{RennerMonoidStructure} (c). To show that this union is disjoint let $e,e_1\in E(R)$ such that $eW\cap e_1W\neq\emptyset$. Then $e_1=ew$ for some $w\in W$. We observe that $ee_1 = ew = e_1$. From $e=e_1w^{-1}$ we get $e_1e = e$. But  $e$ and $e_1$ commute by Theorem \ref{RennerMonoidStructure} (b), so $e=e_1$. It also follows that $E(eW)=\{e\}$.

The remaining statements $R =\bigsqcup_{e\in E(R)} We$, and  $E(We)=\{e\}$ for all $e\in E(R)$ follow similarly.
\hfill$\Box$\end{proof}

For $x\in R$ we denote by $Cl_W(x):=\{wxw^{-1}\mid w\in W\}$ the set of $W$-conjugates of $x$. From Corollary \ref{FL-EL} and Corollary \ref{ff-cs} we obtain
\begin{equation}\label{transversal}
         E(R) = \bigcup_{w\in W} w\Lambda w^{-1} = \bigsqcup_{e\in\Lambda} Cl_W(e).
\end{equation}

\begin{lemma}\label{WW-orbits} The Renner monoid R is generated by $S$ and $\Lambda$. We have
\begin{equation*}
 R = W\Lambda W = \bigsqcup_{e\in \Lambda} WeW,
\end{equation*}
and $E(WeW)=Cl_W(e)$ for all $e\in \Lambda$.
\end{lemma}
\begin{proof} Let $e\in\Lambda$. Clearly, $Cl_W(e)\subseteq E(WeW)$. Conversely, for any idempotent $e'\in E(WeW)$, there exist $w, w_1 \in W$ such that $e' = w ew_1 = (w e w^{-1}) w w_1$. Then $e' = w e w^{-1}$ by Lemma \ref{W-orbits}.

Theorem \ref{RennerMonoidStructure} (c) and (\ref{transversal}) show that $R=\bigcup_{e\in \Lambda} W e W$. So $R$ is generated by $S$ and $\Lambda$. To show that the union is disjoint let $e,f\in\Lambda$ such that $W e W \cap W f W \neq\emptyset$. Then $WeW=W f W$,
from which we get $Cl_W(e)=Cl_W(f)$. Now $e=f$ follows from (\ref{transversal}).
\hfill$\Box$\end{proof}

We continue to use the examples in Section \ref{s3} to describe their cross-section lattices and Renner monoids.
\begin{exam}
{\rm
Let $H$ be the weight hull of $\mu$ as in Example \ref{A2}, where $\mu=3\m_1+2\m_2$ is a dominant weight of the Kac-Moody algebra $\g$ of type $A_2$.

The set of idempotents $E(R)$ consists of 14 elements since $H$ has 14 faces. The cross-section lattice
\[
    \Lambda = \{0, ~e_1, ~e_2,  ~e_3, ~1 \}
\]
has five idempotents, where $e_1$ is the idempotent determined by the vertex face $\{\mu\}$, $e_2$ is the idempotent corresponding to the face $\overline{\mu \,r_1\mu}$, and $e_3$ corresponds to the face $\overline{\mu\, r_2\mu}$. The Renner monoid is finite and
\[
    R= \{0 \} \sqcup W e_1 W \sqcup W e_2 W \sqcup W e_3 W \sqcup W.
\]
}
\end{exam}

\begin{exam}
{\rm
Let $H$ be the weight hull of $\mu$ as in Example \ref{A11}, where $\mu=\mu_1$ is the first fundamental dominant weight of the affine Kac-Moody Lie algebra $\g$ of type $A_1^{(1)}$.

The set of idempotents $E(R)$ is infinite since $H$ has infinitely many faces. The cross-section lattice
\[
    \Lambda = \{0, ~e_1, ~e_2, ~1 \}
\]
has 4 elements, where $e_1$ is the idempotent determined by the vertex face $\{\mu\}$, and $e_2$ is the idempotent determined by the face $\overline{\mu\, r_1\mu}$. The Renner monoid is infinite and
\[
    R = \{0 \} \sqcup W e_1 W \sqcup W e_2 W \sqcup W.
\]
}
\end{exam}

\begin{exam}{\rm
Let $H$ be the weight hull of the dominant weight $\mu=\mu_1+\mu_2$ of the indefinite, strongly hyperbolic Kac-Moody Lie algebra $\g(A)$ of Example \ref{Aab}.

The set of idempotents $E(R)$ is infinite since $H$ has infinitely many faces. The cross-section lattice
\[
     \Lambda = \{0, ~e_1, ~e_2,  ~e_3, ~1 \}
\]
has 5 elements, where $e_1$ is the idempotent determined by the vertex face $\{\mu\}$, and $e_2$ is the idempotent corresponding to the face $\overline{\mu\,r_1\mu}$, and $e_3$ corresponds to the face $\overline{\mu\,r_2\mu}$. The Renner monoid is infinite and
\[
    R = \{0 \} \sqcup W e_1 W \sqcup W e_2 W \sqcup W e_3 W \sqcup W.
\]
}
\end{exam}

\vspace*{0.5em}
Let $e\in E(R)$. The left and right centralizers of $e$ in $W$ are defined by
\begin{eqnarray*}
C^{\,l}_W(e):=\{w\in W\mid we=ewe\} \quad\text{ and  }\quad C^{\,r}_W(e):=\{w\in W\mid ew=ewe\}.
\end{eqnarray*}
The centralizer of $e$ in $W$ is defined by $C_W(e):= \{w\in W\mid we=ew\}= C^{\,l}_W(e)\cap C^{\,r}_W(e)$.
Similarly, the left and right stabilizers of $e$ in $W$ are defined by
\begin{eqnarray*}
   S^{\,l}_W(e) := \{w\in W\mid we=e\} \quad \text{ and } \quad S^{\,r}_W(e) :=\{w\in W\mid ew=e\}.
\end{eqnarray*}
The stabilizer of $e$ in $W$ is defined by $ S_W(e):=\{w\in W\mid we=ew=e\}= S^{\,l}_W(e)\cap S^{\,r}_W(e)$.
\begin{theorem}\label{centralizersOfe}  Let $e=e(F)$ where $F$ is a face. Then:
\begin{itemize}
\item[\rm (a)] $C_W^{\,l}(e)= C_W^{\,r}(e)= C_W(e)= W(F)$.
\item[\rm (b)] $S_W^{\,l}(e)=S_W^{\,r}(e)= S_W(e) = W_*(F)$.
\end{itemize}
\end{theorem}
\begin{proof} Applying the inverse map of $R$ described in Corollary \ref{Rinverse} to the equations which define the centralizers and stabilizers, we find  $C_W^{\,r}(e)= C_W^{\,l}(e)^{-1}$ and  $S_W^{\,r}(e)= S_W^{\,l}(e)^{-1}$.

Because of $C_W^{\,r}(e)= C_W^{\,l}(e)^{-1}$ and $C_W(e)=C_W^{\,r}(e)\cap C_W^{\,l}(e)$ it is sufficient to show $C_W^{\,l}(e)=W(F)$ in (a). For $w\in W$ we find by Corollary \ref{FL-EL} that
\begin{eqnarray*}
   we(F) = e(F)w e(F) = w w^{-1}e(F)w e(F)= w e(w^{-1}F\cap F)
\end{eqnarray*}
if and only if $F=w^{-1}F\cap F$, if and only if $F\subseteq w^{-1}F$. Since $F$ and $w^{-1}F$ are faces of the same dimension, $F\subseteq w^{-1}F$ is equivalent to $F=w^{-1}F$, which in turn is equivalent to $w\in W(F)$.

Since $S_W^{\,r}(e)= S_W^{\,l}(e)^{-1}$ and $S_W(e)=S_W^{\,r}(e)\cap S_W^{\,l}(e)$ it is sufficient to show $S_W^{\,l}(e)=W_*(F)$ in (b).
Moreover, it suffices to show this for a fundamental face $F$.

Let $w\in W$. Then $we(F)=e(F)$ in $W$ if and only if there exists $n_w\in N$ projecting to $w$ such that $n_w e(F)=e(F)$ in $G$. By Proposition \ref{YYOmega} (d) and Theorem \ref{isostabN} this is equivalent to $w\in W_{\lambda_*(F)}$. From Theorem \ref{faceIsotropyGroup} we have $W_{\lambda_*(F)}=W_*(F)$.
\hfill$\Box$
\end{proof}

As a supplement we now rewrite some of the results of this section in a form encountered in the theory of $J$-irreducible reductive linear algebraic monoids.

We define the {\it type map} $\lambda: \Lambda \rightarrow 2^\Pi$, and the maps $\lambda_*: \Lambda \rightarrow 2^\Pi$ and $\lambda^*: \Lambda \rightarrow 2^\Pi$ as follows: If $e=e(F)$ where $F$ is a fundamental face of $H$ then
\begin{equation*}
   \lambda(e):=\lambda(F) \quad\text{and}\quad  \lambda_*(e):=\lambda_*(F) \quad\text{and}\quad  \lambda^*(e):=\lambda^*(F) .
\end{equation*}
From Theorems \ref{faceIsotropyGroup} and \ref{centralizersOfe} we obtain the following characterizations. In the theory of reductive linear algebraic monoids these are often used as definitions, where the cross-section lattice $\Lambda$ is obtained as in Theorem \ref{PutchaLattice}.
\begin{corollary}\label{centralizersOfe2} Let $e\in\Lambda$. Then we have
\begin{eqnarray*}
    &&\lambda(e) = \{\a\in \Pi \mid r_\a e = e r_\a\},\\
     && \lambda_*(e) = \{\a \in \Pi \mid r_\a e = e r_\a =  e\},\\
    && \lambda^*(e)  =\{\a \in \Pi \mid r_\a e = e r_\a \ne e\}.
\end{eqnarray*}
\end{corollary}

Corollary \ref{centralizersOfe2} can be used to eliminate in previous results the dependence on the faces of the orbit hull $H$. An example: For $e\in E(R)$ we set, as in the theory of reductive linear algebraic monoids, $W(e):=C_W(e)$ and $W_*(e):= S_W(e)$. From Theorems \ref{faceIsotropyGroup} and \ref{centralizersOfe} we get the following corollary.
\begin{corollary}\label{centralizersOfe-rew} Let $e\in\Lambda$. Then
\begin{equation*}
           W(e) = W_{\lambda(e)} =   W_{\lambda_*(e)} \times  W_{\lambda^*(e)}  \quad\text{and}\quad W_*(e) = W_{\lambda_*(e)}.
\end{equation*}
Furthermore, $eW(e)=eW(e)e=W(e)e$ is a group with identity $e$, isomorphic to $W_{\lambda^*(e)} $.
\end{corollary}

The next corollary is a consequence of Corollary \ref{dfct} and Theorem \ref{faceIsotropyGroup}. The results are similar to those for $J$-irreducible reductive linear algebraic monoids obtained by M. S. Putcha and L. E. Renner in Corollary 4.12, Theorem 4.16, and Corollary 4.11 of \cite{PR88}.
\begin{corollary}\label{dcsl-R} The map $\lambda^*:\Lambda\to 2^\Pi$ restricts to an isomorphism of partially ordered sets
\begin{equation*}
  \lambda^*:\Lambda\setminus\{0\} \to\{ I\subseteq \Pi \mid I \text{ is $\mu$-connected } \}.
\end{equation*}
For $e\in\Lambda\setminus\{0\}$ we have $\lambda_*(e)=\{\a\in J_0\setminus\lambda^*(e)\mid r_\a r_\beta = r_\beta r_\a \text{ for all }\beta\in\lambda^*(e)\}$.
\end{corollary}
%
%

Generalized Renner-Coxeter systems have been introduced by E. Godelle in \cite[Definition 1.4]{Go} as a common concept for various sorts of monoids, called Renner monoids in the literature. Equivalent concepts can be found implicitly in the work of M. S. Putcha and L. E. Renner.
\begin{theorem}\label{grm} The triple $(R,\Lambda, S)$ is a generalized Renner-Coxeter system, that is, the triple has the following properties:
\vspace{-1.5mm}
\begin{enumerate}[{\rm(a)}]
    \item R is a unit regular monoid, and the idempotents of $R$ commute.
    \vspace{-3mm}
  \item $(R^\times, S)$ is a Coxeter system.
    \vspace{-3mm}
    \item $\Lambda$ is a sub-semilattice of $E(R)$, and a cross-section for the action of $R^\times$ on $E(R)$ by conjugation.
    \vspace{-3mm}
    \item For each pair $e_1\le e_2$ in $ E(R)$ there exist $w\in R^\times$ and a pair $f_1\leq f_2$ in $\Lambda$, such that $e_1 = w f_1 w^{-1}$ and $e_2=w f_2 w^{-1}$.
    \vspace{-3mm}
    \item For each $e\in \Lambda$, both $C_{R^\times}(e)$ and $S_{R^\times}(e)$ are standard parabolic subgroups of $R^\times$.
    \vspace{-3mm}
    \item The map $\lambda_S^*:\Lambda\to 2^S$, defined on $e\in\Lambda$ by $\lambda_S^*(e):=\{s\in S\mid se = es\neq e\}$, is non-decreasing: $e\le f\Rightarrow \lambda_S^*(e)\subseteq\lambda_S^*(f)$.
  \end{enumerate}
\end{theorem}

\begin{proof}
Parts (a) and (b) hold by Theorem \ref{RennerMonoidStructure}. Part (c) follows from Corollary \ref{ff-sl} and (\ref{transversal}). Part (e) holds by Corollary \ref{centralizersOfe-rew}. Part (f) is obtained from Corollary \ref{dcsl-R} and $\lambda^*(0)=\emptyset$.

It remains to prove (d). There exist two faces $F_1\subseteq F_2$ such that $e_1=e(F_1)$ and $e_2=e(F_2)$. By Corollary \ref{ff-cs} there exists $w'\in W$ such that $w'F_2$ is a fundamental face. Clearly, $w'F_1$ is a face of $w'F_2$. It follows from Corollary \ref{ffl2} that there exists $w_1\in W(w'F_2)$ such that $w_1w'F_1$ is a fundamental face contained in $w'F_2$. Moreover, $w_1w'F_1$ is a face of $w_1w'F_2=w'F_2$. Set $w:=w_1w'$, $f_1:=e(wF_1)$, and $f_2:=e(wF_2)$. Then $f_1, f_2\in \Lambda$ and $f_1\le f_2$. Furthermore, $w^{-1}f_iw = e_i$ for $i = 1, 2$.
 \hfill$\Box$
\end{proof}

\subsection{The Bruhat and Birkhoff decompositions}
The following decompositions of $M$ are the first step to establish the Bruhat and Birkhoff decompositions.
\begin{theorem}\label{GeB} Let $\epsilon\in \{+,-\}$. Then
\[
    M = \bigsqcup_{e\in E(\T)}GeB^\epsilon =  \bigsqcup_{e\in E(\T)}B^\epsilon e G .
\]
\end{theorem}
\begin{proof} Theorem \ref{PutchaDecomposition} shows that
\begin{equation}\label{GeBeq1}
   M = \bigsqcup_{f\in \Lambda}GfG .
\end{equation}

Now let $f\in\Lambda$. From $G=\bigsqcup_{w\in W}B^- wB^\epsilon$ and Corollary \ref{PeF} (b) we obtain
\begin{equation*}
    GfG = \bigcup_{w\in W}Gf B^- wB^\epsilon  =\bigcup_{w\in W}GfwB^\epsilon = \bigcup_{e\in Cl_W(f)}GeB^\epsilon.
\end{equation*}
This decomposition is also disjoint. Let $e=wfw^{-1}$ and $e_1=w_1fw_1^{-1}$, and suppose that $GeB^\epsilon \cap Ge_1B^\epsilon\neq\emptyset$. Then there exist $g\in G$, $b\in B^\epsilon $, and $n,n_1\in N$ projecting respectively to $w,w_1\in W$ such that $  g n f n^{-1}b=n_1 f n_1^{-1}$. So
\[
    n^{-1}g^{-1}n_1 f = f n^{-1}b n_1.
\]
It follows from Theorem \ref{geF} that $n^{-1}bn_1\in P^-_{\lambda(f)}$. Equivalently, $bn_1\in w P^-_{\lambda(f)}$.
But $G$ is a disjoint union of $B^\epsilon x P^-_{\lambda(f)}$ where $x\in W/W_{\lambda(f)}$.
Hence $w_1\in wW_{\lambda(f)}$, so $e _1= w_1 f w_1^{-1} = w f w^{-1} = e$ by Corollary \ref{centralizersOfe-rew}.

Inserting in (\ref{GeBeq1}), we get
\begin{equation*}
   M = \bigsqcup_{f\in \Lambda,\,e\in Cl_W(f)}GeB^\epsilon =\bigsqcup_{e\in E(\T)}GeB^\epsilon,
\end{equation*}
and $M= \bigsqcup_{e\in E(\T)}B^\epsilon eG$ is obtained by applying the Chevalley anti-involution.
\hfill$\Box$
\end{proof}

We also need a technical Lemma.

\begin{lemma}\label{chstabW} Let $F$ be a face of $H$ and $w \in W(F)$. The following are equivalent:
\begin{itemize}
\item[\rm (a)] $w\in W_*(F)$.
\item[\rm (b)] $w\eta\geq\eta$ for all $F$-weights $\eta$.
\item[\rm (c)]  $w\eta\leq\eta$ for all $F$-weights $\eta$.
\end{itemize}
\end{lemma}
\begin{proof} Obviously, (a) implies (b) as well as (c). Now suppose that (b) holds. Let $\eta'$ be an $F$-weight. Since $w\in W(F)$, the weights $w\eta'$, $w^2\eta'$,  $w^3\eta'$, \ldots are $F$-weights. By (b) we find
\begin{equation*}
 \eta'\leq w\eta'\leq w^2\eta'\leq w^3\eta'\leq\cdots  .
\end{equation*}
All elements of this chain are smaller than or equal to $\mu$. Since there are only finitely many weights of $V$ between $\eta'$ and $\mu$ this chain gets stationary. Thus, there exists $k\in\Z_+$ such that $w^{k+1}\eta'=w^k\eta'$. Hence $w\eta'=\eta'$. Now (a) follows from Lemma \ref{alphaInDLowerStar}.

Suppose that (c) holds. Let $\eta$ be an $F$-weight. Because $w^{-1}\in W(F)$, the weight $w^{-1}\eta$ is an $F$-weight. We obtain $\eta=w w^{-1}\eta\leq w^{-1}\eta$ by (c). From the equivalence of (a) and (b) we get $w^{-1}\in W_*(F)$, from which (a) follows.\hfill$\Box$
\end{proof}

Now we can show the following result.
\begin{theorem}\label{bd} Let $\epsilon,\delta\in\{+,-\}$. Then
\[
  M = \bigsqcup_{x\in R}B^\epsilon xB^\delta  .
\]
\end{theorem}
\begin{proof}
Every idempotent $e\in E(\T)$ can be written uniquely in the form $e= \sigma f\sigma ^{-1}$ with $f\in \Lambda$ and $\sigma \in W^{\lambda(f)}$.
Moreover, $W^{\lambda(f)} = \{w\in W\mid w\a\in\Drp\text{ for all } \a\in\lambda(f)\}$ by \cite[Proposition 2.20]{AB08} and \cite[Lemma 3.11 a)]{Kac90}. From Theorem \ref{GeB} we obtain
\begin{equation}\label{bdeq1}
  M = \bigsqcup_{e\in E(\T)}GeB^\delta = \bigsqcup_{\sigma \in W^{\lambda(f)}, \,f\in \Lambda}Gf\sigma ^{-1}B^\delta .
\end{equation}

Now we consider a set $G f\sigma ^{-1}B^\delta$ of this union. Inserting $G= \bigsqcup_{w\in W^{\lambda(f)}}B^\epsilon w P_{\lambda(f)}$ and the decomposition
\[
    P_{\lambda(f)} = TG_{\lambda^*(f)} G_{\lambda_*(f)}\ U^{\lambda(f)} = U_{\lambda^*(f)}^\epsilon
T N_{\lambda^*(f)} U_{\lambda^*(f)}^\delta G_{\lambda_*(f)}\ U^{\lambda(f)},
\]
we obtain from Corollaries \ref{cG2} and \ref{cG1} that
\begin{align}
  Gf \sigma ^{-1}B^\delta
  \nonumber & = \bigcup_{w\in W^{\lambda(f)}}B^\epsilon w P_{\lambda(f)}f \sigma ^{-1}B^\delta
                      = \bigcup_{w\in W^{\lambda(f)}}B^\epsilon w  U_{\lambda^*(f)}^\epsilon W_{\lambda^*(f)} U_{\lambda^*(f)}^\delta f \sigma ^{-1}B^\delta
 \\
  \nonumber & = \bigcup_{w\in W^{\lambda(f)}}B ^\epsilon w  U_{\lambda^*(f)}^\epsilon w^{-1} wW_{\lambda^*(f)} f \sigma^{-1}\sigma U_{\lambda^*(f)}^\delta\sigma^{-1}B^\delta.
 \end{align}
Here, $w U_{\lambda^*(f)}^\epsilon w^{-1}\subseteq B^\epsilon$ and $\sigma  U_{\lambda^*(f)}^\delta \sigma ^{-1}\subseteq B^\delta$. In view of $f=W_{\lambda_*(f)}f$ we get
\begin{equation} \label{bdeq2}
  Gf \sigma ^{-1}B^\delta
     = \bigcup_{w\in W^{\lambda(f)}}B^\epsilon w W_{\lambda(f)} f \sigma ^{-1} B^\delta
     = \bigcup_{w\in W}B^\epsilon w \sigma ^{-1}e B^\delta
     = \bigcup_{x\in We}B^\epsilon x B^\delta.
\end{equation}

We next show that this union is disjoint. Note that $e=e(F)$ for a face $F$ of the weight hull $H$. If $B^\epsilon we B^\delta \cap B^\epsilon \tilde{w} e B^\delta\neq\emptyset $ there exist $u_\epsilon\in U^\epsilon$, $\tilde{u}_\delta\in U^\delta$ and $n, \tilde{n}\in N$ projecting to $w,\tilde{w}$, respectively, such that $u_\epsilon \tilde{n} e(F) \tilde{u}_\delta  = n e(F) $. Therefore,
\begin{equation}\label{eqBdis}
  (n^{-1}u_\epsilon n) (n^{-1}\tilde{n}) e(F) =  e(F) \tilde{u}_\delta^{-1}  .
\end{equation}

Comparing the images of both sides in (\ref{eqBdis}), we get $ (n^{-1}u_\epsilon n)V_{w^{-1}\tilde{w}F}= V_{F}$, and equivalently $ V_{w^{-1}\tilde{w}F}= (n^{-1}u_\epsilon^{-1} n)V_{F}$. By the action of $n^{-1} U^\epsilon n$ on the weight spaces we conclude that $w^{-1}\tilde{w}(F\cap P(V))\subseteq F\cap P(V)$ and $w^{-1}\tilde{w}(F\cap P(V))\supseteq F\cap P(V)$. From Lemma \ref{alphaInDLowerStar} (a) we obtain $w^{-1}\tilde{w}\in W(F)$.

Let $\eta$ be an $F$-weight. Evaluating both sides of (\ref{eqBdis}) at $v_\eta\in V_\eta\setminus\{0\}$, we find
\begin{equation*}\label{eqBdis3}
 \  w^{-1}\tilde{w}\eta\in \text{supp} ( (n^{-1}u_\epsilon n) ( n^{-1}\tilde{n}) v_\eta) = \text{supp} (  e(F) \tilde{u}_\delta^{-1}  v_\eta )\subseteq \eta+ Q_\delta  .
\end{equation*}
If $\delta$ is equal to $+$, we get $w^{-1}\tilde{w}\eta \geq \eta$ for all $F$-weights $\eta$. If $\delta$ is equal to $-$, we have $w^{-1}\tilde{w}\eta \leq \eta$ for all $F$-weights $\eta$. From Lemma \ref{chstabW} we find $w^{-1}\tilde{w}\in W_*(F)$ in both cases. Hence, $w e(F)= \tilde{w}e(F)$ by Theorem \ref{centralizersOfe}.

Inserting the disjoint union (\ref{bdeq2}) in (\ref{bdeq1}) we obtain
\[
  M = \bigsqcup_{e\in E(\T)} \,\bigsqcup_{x\in We } B^\epsilon x B^\delta = \bigsqcup_{x\in R}B^\epsilon x B^\delta .
\]
\hfill$\Box$\end{proof}

In the following proposition we generalize some properties of the twin BN-pairs $(B^\pm,N)$ of $G$ given in \cite[Definition 6.55, Lemma 6.80]{AB08} to the monoid $M$.
\begin{proposition}  Let $\a\in\Pi$, $x\in R$, and $\epsilon,\delta \in\{+,-\}$. Then
\begin{itemize}
\item[\rm (a)] $ (B^\epsilon r_\a B^\epsilon)(B^\epsilon x B^\delta) \subseteq B^\epsilon r_\a x B^\delta \cup B^\epsilon x B^\delta $.
\item[\rm (b)] $  (B^\delta x B^\epsilon)(B^\epsilon r_\a B^\epsilon) \subseteq B^\delta  x r_\a B^\epsilon \cup B^\delta x B^\epsilon$.
\end{itemize}
\end{proposition}
\begin{proof} It suffices to show (a). Then (b) follows from (a) by applying the Chevalley anti-involution. We write $x$ in the form $x=we$ with $w\in W$ and $e\in E(R)$. Recall that
\begin{eqnarray*}
  \Delta^{re} = \Delta_p(e) \cup \Delta_*(e)\cup\Delta^*(e)\cup\Delta_n(e).
\end{eqnarray*}

From \cite[Proposition 4.2]{Kac90} we get $\bU^+= \bU_\a \bU^\a_+=\bU^\a_+ \bU_\a$ where $\bU^\a_+:=\bU^+\cap r_\a\bU^+ r_\a$.  Applying $\br$ and the Chevalley anti-involution, we find
\begin{equation*}
     U^\epsilon = U^{\epsilon \a} U_{\epsilon \a} \quad \text{ where }\quad  U^{\epsilon \a}:=U^\epsilon\cap r_\a U^\epsilon r_\a.
\end{equation*}
If $\epsilon w^{-1}\a\in \Delta_p(e)\cup \Delta_*(e)\cup  (\Delta^*(e)\cap\Delta^{re}_\delta)$ then by Corollary \ref{RennerLemma} we obtain
\begin{eqnarray*}
   r_\a B^\epsilon we = r_ \a U^{\epsilon\a} U_{\epsilon\a }wT e =  U^{\epsilon \a} r_\a U_{\epsilon\a} w T e = U^{\epsilon\a} r_\a w T U_{\epsilon w^{-1}\a}e \subseteq B^\epsilon r_\a weB^\delta.
\end{eqnarray*}
Hence $ B^\epsilon r_\a B^\epsilon we B^\delta = B^\epsilon r_\a we B^\delta$.

We have $ r_\a B^\epsilon r_\a \subseteq B^\epsilon \cup B^\epsilon r_\a B^\epsilon$ because $(B^\epsilon, N)$ is a BN-pair of $G$. If $\epsilon w^{-1}\a\in (\Delta^*(e)\cap\Delta^{re}_{-\delta})\cup  \Delta_n(e)$ then $\epsilon (r_\a w)^{-1}\a = - \epsilon w^{-1}\a\in\Delta_p(e)\cup (\Delta^*(e)\cap\Delta^{re}_\delta)$, and we find
\begin{eqnarray*}
      B^\epsilon r_\a B^\epsilon we B^\delta\hspace{-5mm} &&=   B^\epsilon r_\a B^\epsilon  r_\a r_\a we B^\delta \subseteq B^\epsilon(B^\epsilon \cup B^\epsilon  r_\a B^\epsilon) r_\a we B^\delta\\
    &&= B^\epsilon r_\a we B^\delta \cup B^\epsilon r_\a B^\epsilon  r_\a we B^\delta =  B^\epsilon r_\a we B^\delta \cup B^\epsilon r_\a r_\a we B^\delta\\
    &&= B^\epsilon r_\a we B^\delta \cup B^\epsilon we B^\delta .
\end{eqnarray*}
\hfill$\Box$\end{proof}

\vspace{10mm}
\noindent
Zhenheng Li

\noindent
College of Mathematics and Information Science, Hebei University, Baoding, Hebei, 071002, China, and Department of Mathematical Sciences, University of South Carolina Aiken, Aiken, SC 29803, USA. Email: zhenhengl@usca.edu

\vspace{2mm}
\noindent Zhuo Li

\noindent
Department of Mathematics, Xiangtan University, Xiangtan, Hunan, 411105, China. Email: zli@xtu.edu.cn

\vspace{2mm}
\noindent  Claus Mokler

\noindent
Department of Mathematics, University Bochum, 44780 Bochum, Germany. Email: claus.mokler@web.de

\end{document}